\documentclass[11pt]{amsart}

\usepackage{enumerate}

\RequirePackage[english]{babel}
\RequirePackage[numbers]{natbib}

\usepackage[%backref,
pdftex,hyperfootnotes]{hyperref}
\hypersetup{
 colorlinks=true,
 linkcolor=NavyBlue, 
 urlcolor=RoyalPurple,
 citecolor=OliveGreen,
 pdftitle={Finite Markov chains and multiple orthogonal polynomials},
 bookmarks=true,
 pdfpagemode=FullScreen,
}

\RequirePackage{amssymb,latexsym,amsmath,amsthm,bm}
\RequirePackage{mathtools,mathdots}
\mathtoolsset{showonlyrefs}

\RequirePackage[usenames,dvipsnames,svgnames,table,x11names]{xcolor}
\RequirePackage{pgfplots}
\RequirePackage{tikz}
\RequirePackage{tikz-3dplot}
\usetikzlibrary{automata,quotes, chains,matrix,calc,shadows,shapes.callouts,shapes.geometric,shapes.misc,positioning,patterns,decorations.shapes,
decorations.pathmorphing,decorations.markings,decorations.fractals,decorations.pathreplacing,shadings,fadings,arrows.meta,bending}
\tikzstyle{block} = [draw, rectangle,minimum height=3em, minimum width=2em]
\RequirePackage[renew-dots,renew-matrix]{nicematrix}

\RequirePackage[utf8]{inputenc}

\RequirePackage{Baskervaldx}
\RequirePackage[]{newtxmath}

\RequirePackage[textwidth=17.5cm,textheight=22.275cm,
height=24.275cm,
width=18cm]{geometry}

\theoremstyle{plain}
\newtheorem{teo}{Theorem}[section]
\newtheorem{coro}[teo]{Corollary}
\newtheorem{lemma}[teo]{Lemma}
\newtheorem{pro}[teo]{Proposition}

\newtheorem{defi}[teo]{Definition}

\theoremstyle{remark}
\newtheorem{rem}[teo]{Remark}
\newtheorem{exa}[teo]{Example}
\renewcommand{\d}{\operatorname{d}}
\newcommand{\Exp}[1]{\operatorname{e}^{#1}}

\newcommand{\diag}{\operatorname{diag}}

\newcommand{\C}{\mathbb{C}}

\newcommand{\N}{\mathbb{N}}
\newcommand{\R}{\mathbb{R}}

\newcommand*\pFqskip{8mu}
\catcode`,\active
\newcommand*\pFq{\begingroup
 \catcode`\,\active
 \def ,{\mskip\pFqskip\relax}%
 \dopFq
}
\catcode`\,12
\def\dopFq#1#2#3#4#5{%
 {}_{#1}F_{#2}\biggl[\genfrac..{0pt}{}{#3}{#4};#5\biggr]%
 \endgroup
}
\newcommand{\KF}[5]{F^{#1}_{#2}\left[{#3\atop #4}\Bigg\vert #5\right]}

\begin{document}
 
 \title[Finite Markov chains and multiple orthogonal polynomials]{Finite Markov chains
 \\
 and multiple orthogonal polynomials}
 \author[A Branquinho]{Amílcar Branquinho\( ^{1}\)}
 \address{\( ^1\)CMUC, Departamento de Matemática,
 Universidade de Coimbra, 3001-454 Coimbra, Portugal}
 \email{\( ^1\)ajplb@mat.uc.pt}
 
 \author[JEF Díaz]{Juan EF Díaz\( ^{2}\)}
 \address{\( ^2\)CIDMA, Departamento de Matemática, Universidade de Aveiro, 3810-193 Aveiro, Portugal}
 \email{\( ^2\)juan.enri@ua.pt}

 \author[A Foulquié]{Ana Foulquié-Moreno\( ^{3}\)}
 \address{\( ^3\)CIDMA, Departamento de Matemática, Universidade de Aveiro, 3810-193 Aveiro, Portugal}
 \email{\( ^3\)foulquie@ua.pt}

 \author[M Mañas]{Manuel Mañas\( ^{4}\)}
 \address{\( ^4\)Departamento de Física Teórica, Universidad Complutense de Madrid, Plaza Ciencias 1, 28040-Madrid, Spain}
 \email{\( ^4\)manuel.manas@ucm.es}
 
 \keywords{Multiple orthogonal polynomials, hypergeometric series, Hessenberg matrices, recursion matrix, Markov chains, stochastic matrices, classes, recurrence, stationary states, ergodicity, expected return times, Hahn, Laguerre, Meixner, Jacobi--Piñeiro, AT systems}
 
 \subjclass{42C05, 33C45, 33C47, 47B39, 47B36, 60J10, 60J22}

\begin{abstract}
 This paper investigates stochastic finite matrices and the corresponding finite Markov chains constructed using recurrence matrices for general families of orthogonal polynomials and multiple orthogonal polynomials. The paper explores the spectral theory of transition matrices, using both orthogonal and multiple orthogonal polynomials. Several properties are derived, including classes, periodicity, recurrence, stationary states, ergodicity, expected recurrence times, time-reversed chains, and reversibility. Furthermore, the paper uncovers factorization in terms of pure birth and pure death processes.
The case study focuses on hypergeometric representations of orthogonal polynomials, where all the computations can be carried out effectively. Particularly within the Askey scheme, all descendants under Hahn such as Hahn itself, Jacobi, Meixner, Kravchuk, Laguerre, Charlier, and Hermite, present interesting examples of recurrent reversible birth and death finite Markov chains. Additionally, the paper considers multiple orthogonal polynomials, including multiple Hahn, Jacobi--Piñeiro, Laguerre of the first kind, and Meixner of the second kind, along with their hypergeometric representations and derives the corresponding recurrent finite Markov chains and time-reversed chains.
A Mathematica code, publicly accessible in repositories, has been crafted to analyze various features within finite Markov chains.
\end{abstract}

\maketitle

\setcounter{tocdepth}{3}

\section{Introduction}
\allowdisplaybreaks

The interplay between orthogonal polynomials and stochastic processes has a long history. For example, Hermite polynomials have played a significant role in the theory of stochastic processes and integration with respect to the Wiener process \cite{Ito, wiener}.

In the 1950s, there were important advancements in understanding the connections between orthogonal polynomials and stochastic processes. Influential papers during this time focused on the spectral representation of probabilities in birth and death processes. Notable contributions include the works by Kendall, Ledermman, and Reuter \cite{Kendall, Kendall2, Ledermann}, as well as the seminal papers by Samuel Karlin and James McGregor \cite{KmcG1957-1, KmcG1957-2}. These papers delved into birth and death Markov processes, examining differential and classification aspects, and highlighting the integral representation of transient probability matrices that revealed the close relationship between birth and death processes and the theory of the Stieltjes moment problem. Karlin and McGregor also explored random walks, which are uncountable Markov chains \cite{KmcG}. Their work introduced the Karlin--McGregor representation formula, which provided an integral representation of relevant probabilistic quantities of stochastic processes in terms of orthogonal polynomials, and analyzed the recurrence and absorption characteristics of these processes. For further details, see \cite{mirta, Grunbaum1, Schoutens}.

In modern times, the concept of a ``random walk polynomial sequence'' has emerged. It refers to a polynomial sequence that is orthogonal with respect to a measure on \( [-1, 1] \), and satisfies a three-term recurrence relation with nonnegative coefficients. Any measure for which a random walk polynomial sequence is orthogonal is known as a random walk measure. Random walk polynomials have become a well-studied topic in the literature on orthogonal polynomials. Refer to \cite[Chapter \(4\)]{Ismail} for a recent account of some of the key~aspects.

A resurgence of interest in these ideas occurred in the 1970s and 1980s. The work by Whitehurst \cite{Whitehurst} explored simple random walks and their integral representation in terms of orthogonal polynomials and the support of the spectral measure of the transition matrix. Papers by Ogura \cite{Ogura} and Engel \cite{Engel} established an integral relation between the Poisson process and the discrete Charlier orthogonal polynomials. In a separate study \cite{Diaconis}, the Stein equations for well-known distributions, including Pearson's class, were connected with their corresponding orthogonal polynomials.

Orthogonal polynomials and the ideas put forth by Karlin and McGregor have found applications in queuing problems \cite{Charris, Chiara, Van Doorn0}. In \cite{Kijima}, the authors represented the conditional limiting distribution of a birth and death process using birth-death polynomials. Van Doorn and Schrinjer \cite{Van Doorn} studied random walk polynomials and random walk measures, which are relevant in the analysis of random walks. They provided properties of random walk measures and polynomials 
obtaining a limit theorem for random walk measures, which is significant in the study of random walks.
Additionally, in \cite{Coolen}, the same authors explored discrete-time birth-death processes or random walks and emphasized the role of orthogonal polynomials. They demonstrated how to determine it, if a given sequence of orthogonal polynomials is a sequence of random walk polynomials, as well as, whether a given random walk measure corresponds to a unique random walk. 

It is crucial to underscore that the aforementioned papers focused exclusively on countable simple homogeneous Markov chains, i.e. those of death and birth types with an infinite number of states. However, a significant gap exists in the literature, as we have not yet found a thorough discussion connecting finite Markov chains with orthogonal polynomials and their corresponding zeros.

In recent years, significant progress has been made in understanding the intricate relationship between probability and orthogonal polynomials. For instance, in \cite{Schoutens}, the Kravchuk polynomials were shown to play a crucial role in stochastic integration theory with respect to the binomial process. This study connected classical orthogonal polynomials with Stein's method for Pearson's and Ord's classes of distributions. The work also extended Karlin--McGregor's results by considering doubly limiting conditional distributions, providing a probabilistic interpretation for many orthogonal families in the Askey scheme.

In \cite{Kovchegov}, proposals that go beyond near neighbors were presented, focusing on studying the spectrum of a polynomial derived from a given transition matrix. On the other hand, \cite{Obata} reformulated the Karlin--McGregor formula in terms of one-mode interacting Fock spaces and provided an integral expression for the moments of an associated operator. This integral expression led to an extension of the Karlin--McGregor formula to the graph of paths connected with a clique.

Alberto Grünbaum and his collaborators have made significant contributions to the field, particularly in the study of generalized orthogonal polynomials for Markov chains beyond birth and death chains. In their work, Grünbaum proposed exploring matrix orthogonal polynomials to describe Markov chains with jumps that extend beyond nearest neighbors \cite{Grunbaum1, Grunbaum11, Grunbaum2, grunbaum, Grunbaum3}. This proposal, aligning with the ideas discussed in \cite{Kovchegov}, aims to broaden the scope of orthogonal polynomial theory to encompass a wider range of Markov chain dynamics. Matrix orthogonal polynomials are also relevant in the theory of integrable systems, as shown in \cite{AGMM} and related references. The important role of random matrix theory in integrable systems is well-known, and the intriguing interplay between orthogonal polynomials, integrable systems, and probability theory forms an interesting triangle of research.

Very recently, in \cite{delaIglesia}, de la Iglesia comprehensively covers various aspects of the spectral theory of Markov processes using orthogonal polynomials. The book specifically delves into birth and death processes, as well as diffusion processes.

%\enlargethispage{.25cm}

In previous works, we have explored countable Markov chains beyond birth and death constructed from Jacobi--Piñeiro multiple orthogonal polynomials \cite{JP} and hypergeometric multiple orthogonal polynomials~\cite{Hypergeometric}. 
In a manner akin to the approach employed for random walk polynomials (as illustrated in \cite{Van Doorn0,Van Doorn}), we do it for multiple orthogonal polynomials.
The concept behind these works is somewhat similar to the one developed here, where we start with a nonnegative transition matrix, which is a Hessenberg matrix, and devise a procedure to obtain stochastic matrices from it. Subsequently, we derive the corresponding properties of the associated Markov chains. 
The approach presented here is, in a certain sense, the converse of the one developed by Karlin and McGregor \cite{KmcG}. In their work, the stochastic matrix is given, and the authors seek a spectral representation of that homogeneous Markov chain in terms of the corresponding measures and orthogonal polynomials. In contrast, here we are given the family orthogonality, and we aim to find the associated Markov chains and their most relevant properties.

In this paper we study the connection of finite homogeneous Markov chains with orthogonal polynomials sequences, the corresponding zeros and truncations of the recursion matrix. 
To achieve this goal, we expand on the ideas presented in the aforementioned papers. 
 In the context of a finite Markov chain with \( n\) states, we substitute the role of \( 1\) with the largest zero of the orthogonal polynomial \( P_n(x) \). 
We use the zeros of the orthogonal or multiple orthogonal polynomials and, in the case of multiple orthogonal polynomials, we 
%utilize
use determinants of type I multiple orthogonal polynomials. The general theory presented here provides a step-by-step approach to construct stochastic matrices using sets of orthogonal or multiple orthogonal polynomials. These matrices are connected to birth and death Markov chains in the case of orthogonal polynomials and to Markov chains beyond birth and death for multiple orthogonal polynomials. We always consider homogeneous Markov chains \cite{Bremaud}. The Perron–Frobenius theorem for irreducible, nonsingular, nonnegative matrices is a crucial tool in these constructing of Markov chains.

We investigate various properties of these finite Markov chains. These properties include the recurrence or transience of a state, its periodicity, ergodicity, stationary states, expected return times, reversibility and reversal chains. Furthermore, we describe a procedure that enables us to factor the stochastic matrix into bidiagonal stochastic matrices, each modeling a pure birth or pure death Markov chain. This is especially relevant concerning the construction of corresponding urn models \cite{Grunbaum2,grunbaum}. Bidiagonal stochastic matrices give rise to either pure birth or pure death chains. The bidiagonal stochastic factorization represents the Markov chain as a
composition of the most elementary chains—pure birth or pure death chains. In this formulation, each transition within the Markov chain corresponds to a series of consecutive transitions governed by these fundamental pure birth or pure death chains.

This general construction is then applied to families of orthogonal and multiple orthogonal polynomials in the Askey scheme that have a nonnegative recursion matrix.
This allows us to apply all the elements of the general theory, as we have explicit hypergeometric expressions at our disposal. These expressions lead to concrete numerical examples once the zeros of certain polynomials are numerically determined. 
The families of orthogonal polynomials in the Askey scheme include Hahn, Jacobi, Meixner, Kravchuk, Laguerre, Charlier, and Hermite. The Bessel family is excluded since the normalized Jacobi matrix has its extreme diagonal negative. Additionally, in the multiple Askey scheme for multiple orthogonal polynomials, we have multiple Hahn, Jacobi--Piñeiro, multiple Meixner of the second kind, and multiple Laguerre of the first kind.
We have provided two Mathematica codes available in public repositories that, given a set of parameters, generate various probabilistic elements for these Markov chains.

%\enlargethispage{.25cm}

The paper follows the layout outlined below. In the introduction, we provide a concise overview of Markov chains and the Perron--Frobenius theorem in matrix theory. In Section 2, we develop the theory for orthogonal polynomials and the Jacobi matrix.
The primary distinction from previous 
%analyses
studies lies in our focus on finite Markov chains rather than countable ones.
Consequently, the role of the point \( 1\) needs to be replaced by an appropriate zero of an orthogonal polynomial in the sequence. Proposition \ref{propestocasticaescalar} presents the construction scheme of stochastic matrices linked with orthogonal polynomials. In Proposition \ref{pro:reversibility} we prove that the Markov chain is reversible.
By applying the spectral properties of the Jacobi matrix, we characterize the corresponding birth and death Markov chains.

Moving on to Section 3, we apply these findings to Hahn, Jacobi, Meixner, Kravchuk, Laguerre, Charlier, and Hermite polynomials, which are all descendants of the Hahn polynomials in the Askey scheme. It is worth noting that all of these polynomials have a nonnegative recurrence matrix.

In Section 4, we proceed with the construction scheme for multiple orthogonal polynomials with respect to an algebraic Chebyshev (AT) system of two weights in the stepline, along with the corresponding two types of multiple orthogonal polynomials. These polynomials have nonnegative recurrence matrices and are associated with finite non simple Markov chains; i.e., chains beyond birth and death. Theorems~\ref{propestocasticaII} and~\ref{propestocasticaI} provide the procedures for constructing tetradiagonal stochastic matrices.
In this section, applying the spectral properties of the truncated recurrence Hessenberg matrix, we present the main properties of these associated stochastic processes.

Finally, in Section 5, we explore the four possible cases within the multiple Askey scheme that lead to stochastic matrices. We use the available hypergeometric expressions to find explicit numerical examples of finite Markov chains. 

We have uploaded two Mathematica notebooks to both the Mathematica Notebook Archive and GitHub. These notebooks allow the interested reader to compute their own stochastic matrices by choosing different sets of parameters for each family. Additionally, the notebooks provide the corresponding steady states, expected return times, and the pure birth/pure death stochastic factorization.

\subsection{Elements of Markov chains}
Let us start with a concise overview of fundamental concepts related to Markov chains \cite{gallager,feller}, which will be used throughout this study. A countable Markov chain is an integer-time process represented by a sequence of random variables \( \lbrace X_n\rbrace_{n\in\mathbb N_0} \), where the probability of each event depends solely on the state reached in the preceding event. In this paper, we focus exclusively on finite Markov chains, where the random variables \( \lbrace X_n\rbrace_{n\in\mathbb N_0}\) take values from a finite support \( \lbrace{1,\ldots,m}\rbrace\subset\mathbb N_0 \). Each element in this finite support is referred to as a ``state.'' For a discussion on semi-infinite Markov chains, please refer to \cite{JP}.
 
The conditional probabilities, denoted by
\( P_{i,j}\coloneq \Pr(X_{n+1}=j|X_{n}=i) \), 
for states \( i,j\in\{1,\ldots,m\} \), form an \( m\times m\) stochastic matrix \( P\coloneq \left[ \begin{matrix} P_{i,j} \end{matrix} \right]_{i,j=1}^m \). This matrix satisfies the following properties:
\begin{align}
 \label{nonegatividad}
 P_{i,j}&\geq0, \\
 \label{estocasticidad}
 \sum_{k=1}^m P_{i,k}&=1, 
\end{align}
for \( i,j\in\{1,\ldots,m\} \). 

The first condition defines a nonnegative matrix, while the second condition states that
\begin{align*} 
%\begin{aligned}
 P e&=e, & e&\coloneqq
 \begin{bNiceMatrix}
 1 &\Cdots &1
 \end{bNiceMatrix}^\top\in\R^m.
%\end{aligned}
\end{align*}
According to the Chapman--Kolmogorov equation, the probabilities of transitioning from one state to another after \( r\) transitions are given by the entries of the \( r\)-th power of \( P \). In other words, 
\begin{align*}
 \Pr(X_{n+r}=j|X_{n}=i)&=(P^r)_{i,j}.
\end{align*}
Using this result, we can define the period of a state.
\begin{defi}[Period and aperiodic states]
 \label{defperiodo}
 Let \( i\in\{1,\ldots,m\}\) be a state of a Markov chain. The {period} \( d(i)\) of state \( i\) is defined as the greatest common divisor of all the natural numbers \( r\in\mathbb N\) such that \( (P^r)_{i,i}>0 \). If \( d(i)=1 \), the state \( i\) is said to be {aperiodic}.
\end{defi}
\begin{rem}
 \label{condicionsuficienteaperiodicidad}
 Note that if \( P_{i,i}>0 \), then state \( i\) is aperiodic, since the greatest common divisor of a set of numbers that includes \( 1\) must be \( 1 \). 
\end{rem}

\begin{defi}[First-passage-time probability]
 The {first-passage-time probability} is defined as
 \begin{align*}
 f^r_{i,j}=\operatorname {Pr} (X_{r}=j,X_{r-1}\neq j,\ldots ,X_{1}\neq j|X_{0}=i),
 \end{align*}
 for states \( i,j\in\{1,\ldots,m\}\) and number of transitions \( r\in\N_0 \), which represents the probability of transitioning from state \( i\) to state \( j\) for the first time after \( r\) transitions.
\end{defi}

The corresponding generating functions are given by
\begin{align*}
%\begin{aligned}
 P_{i,j}(s)&\coloneq \sum_{r=0}^\infty(P^r)_{i,j}s^r, && F_{i,j}(s) \coloneq \sum_{r=1}^\infty f^r_{i,j}s^r,
%\end{aligned}
\end{align*}
 for states \( i,j\in\{1,\ldots,m\} \), and they satisfy the relation
\begin{align*}
 F_{i,i}(s)=1-\dfrac{1}{P_{i,i}(s)}.
\end{align*}
Using this, we can provide the following definitions:
\begin{defi}[Recurrence and transience]
 \label{defrecurrente}
 A state \( i\) of a Markov chain is called {recurrent} if the probability of returning to state \( i\) is \( 1 \), or equivalently, if
 \begin{align*}
 \lim_{s\rightarrow1^{-}}F_{i,i}(s)=1.
 \end{align*}
 If the probability does not reach \( 1 \), the state \( i\) is said to be {transient}. If all states are recurrent, the chain is referred to as a recurrent chain; otherwise, it is classified as a transient chain.
\end{defi}
\begin{rem}
 \label{condicionsuficienteparalarecurrencia}
 From the previous expressions, it can be observed that state \( i\) is recurrent if and only if \( \lim\limits_{s\rightarrow1^{-}}P_{i,i}(s)\) diverges. 
\end{rem}
\begin{defi}[Ergodicity]
 \label{defergodico}
 A state \( i\) of a Markov chain is called {ergodic} if it is both aperiodic and recurrent. If all states are ergodic, the chain itself is referred to as an ergodic chain.
\end{defi}

\begin{defi}[Classes of states]
 We say that state \( j\) is {accessible} from state \( i\) (\( i\to j\)) if there exists a ``path'' from \( i\) to \( j\); i.e., if \( (P^n)_{i,j}>0\) for some \( n\in\N \). Two states {communicate} if \( i\to j\) and \( j\to i \), meaning that there exists a pair of natural numbers \( n\) and \( m\) such that \( (P^n)_{i,j}>0\) and \( (P^m)_{j,i}>0 \). 
 Two states \( i\) and \( j\) are said to be in the same {class} if they communicate. 
\end{defi}

Then, for two states \( i\) and \( j\) in the same class, starting from state \( i \), state \( j\) can be reached after a finite number of transitions, and vice versa. 

\begin{teo}[Class properties]
 The states in the same class are all recurrent or all transient and have the same period. 
\end{teo}

\begin{defi}[Irreducible Markov chain]
If there is only one class, we say that the Markov chain is irreducible. 
\end{defi}

In the examples we discuss, there will be only one class of states, making them irreducible Markov chains. 

\begin{defi}[Expected return times]
 Let us define the first time to get to state \( j\) from state \( i\) as the random quantity
 \begin{align*}
 T _{i,j}\coloneq \min\{n\geq 1:X_{n}=j|X_{0}=i\}.
 \end{align*}
 The ``first time passage'' from state \( i\) to state \( j\) is the mean or expectation of \( T_{i,j}\); i.e,
 \begin{align*}
 \bar t _{i,j}\coloneq \operatorname {E}(T _{i,j})=\sum_{n=1}^{\infty }n
 f^n_{i,j}.
 \end{align*}
Specifically, when \( i=j \), we write \( \bar t_i\) instead of \( \bar t_{i,i}\) and refer to it as the expected (or mean) return time of state \( i \). 
\end{defi}
This expectation \( \bar t_i\) represents the expected number of steps it takes for the chain to return to the recurrent state \( i \). 

\begin{defi}[Probability vectors]
A vector \( \pi = \begin{bNiceMatrix}
 \pi_1 & \Cdots & \pi_m
\end{bNiceMatrix} \), where \( \pi_i \geq 0\) for states \( i \in {1, \ldots, m}\) and \( \pi e = \sum_{i=1}^m \pi_i = 1 \), is referred to as a probability vector. This vector encodes a probability distribution within the Markov chain, where \( \pi_i\) represents the probability of being in state \( i \). 
\end{defi}

Assuming a probability vector \( {\pi}(0)\) as the initial state, after one transition, the new probability of being in state \( k\) is given by \( \pi_k(1)=\sum_{i=0}^\infty\pi_i(0)P_{i,k} \). 
Thus, the new probability vector will be
\( \pi(1)=\pi(0)P \). Note that \( \pi(1) e=\pi(0)P e=\pi(0) e=1 \). 
Therefore, after \( n\) transitions, the probability vector will be \( {\pi}(n)={\pi}(0 )P^n \), 
which follows from the Chapman--Kolmogorov equation \( P^{n+m}_{i,j}=\sum_{k=0}^\infty P^{n}_{i,k}P^{m}_{k,j} \). 

\begin{defi}
 A steady (or stationary) state is an invariant probability vector \( {\pi} \), meaning that it does not change over time, and satisfies \( {\pi}= \pi P \). 
\end{defi}
\begin{rem}
 In terms of entries we find \( \pi_i=\sum_{j=1}^{m}\pi_j P_{j,i} \). These are the balance equations.
\end{rem}

\begin{teo}
\begin{enumerate}[\rm i)]
 \item For an irreducible Markov chain, if there exists a steady state, it is unique and
 the Markov chain is recurrent.
 \item If the Markov chain is recurrent, there exists a unique steady state. In this case, the steady state entries are \( \pi_i=\frac{1}{\bar t_i}\) in terms of the expected return times \( \bar t_i \), for states \( i\in\{1,\ldots,m\} \). 
\item For ergodic Markov chains, we have the limit property
\begin{align}\label{eq:limit}
 \lim_{r\to\infty}P^r_{i,j}=\pi_j.
\end{align}
\end{enumerate}
\end{teo}
\begin{rem}
 Note that for any probability vector \( \pi(0)\) of a Markov chain satisfying property \eqref{eq:limit}, the long-term evolution will be
 \begin{align*}
 \pi_j(\infty)=\lim_{r\to\infty} \sum_{i=1}^m\pi_i(0)P^r_{i,j} =\sum_{i=1}^m\pi_i(0)\pi_j=\pi_j.
 \end{align*}
 In other words, the steady state is an equilibrium state that all probability vectors tend to as time goes to infinity.
\end{rem}

\begin{rem}
 Simple random walks can be identified with birth and death Markov chains where the transition matrix is tridiagonal. Starting from a given state \( i \), the only possible transitions are to remain in the same state or to move to the neighboring states \( i+1\) or \( i-1 \). 
\end{rem}

\begin{rem}[Time reversal]
Assuming the steady state satisfies \( \pi_i>0\) for \( i\in\{1,\ldots,m\} \), we can define a matrix \( Q\) with entries \( Q_{i,j}=\frac{\pi_j}{\pi_i}P_{j,i} \). 

The matrix \( Q\) is stochastic and, using Bayes' retrodiction formula with the initial distribution \( \pi \), it can be written as \( Q_{i,j}=\operatorname{Pr}\left(X_n=j|X_{n+1}=i\right) \). Hence, this matrix serves as the transition matrix of the initial Markov chain when time is reversed. Moreover, if there exists a stochastic matrix \( Q\) satisfying detailed balance equations \( \pi_iQ_{i,j}=\pi_j P_{j,i}\) for a probability distribution \( \pi \), then \( \pi\) represents a steady state. A Markov chain is considered reversible when \( Q=P \), meaning that the detailed balance \( \pi_iP_{i,j}=\pi_j P_{j,i}\) is satisfied. In other words, the Markov chain and its time-reversed chain are statistically the same chain,~\cite{Bremaud,Haggstrom}.
\end{rem}

\begin{rem}
The product \( PQ\) of two stochastic matrices, denoted as \( P\) and \( Q \), retains the property of being stochastic. Consequently, \( PQ\) serves as a transition matrix for a distinct Markov chain. Given an initial probability vector \( \pi\) representing a probability distribution, the resulting vector \( \pi PQ\) signifies a two-step transition. Initially, there is a transition from \( \pi\) to \( \pi P = \pi' \), representing a Markov chain modeled by \( P \). Subsequently, a second transition occurs from \( \pi'\) to \( \pi' Q \), modeled by \( Q \). This two-step process embodies the evolution of the Markov chain under the combined influence of both matrices \( P\) and \( Q \). 
 \end{rem}

\subsection{Some facts on matrix analysis}

For what follows, we now recall some relevant facts from matrix analysis regarding positivity and total positivity.

\subsubsection{Positivity}
A matrix is said positive if all its entries are positive and nonnegative if all its entries are nonnegative. 
\begin{defi}[Irreducible matrices]
 A matrix \( M\) is called {irreducible} if it cannot be written in the form
%\vspace{-.25cm}
 \begin{align*}
 P
 \, 
 \begin{bNiceArray}{cw{c}{2cm}c|w{c}{1cm}c}[margin] \Block{3-3}<\Large >{A} & & & \Block{3-2}<>{B} & \\
 & & & & \\
 &&&&\\
 \hline
 \Block{2-3}<>{0} & & & \Block{2-2}<>{C} & \\
 &&&&
 \end{bNiceArray}P^{-1} ,
\end{align*}
 where \( A \) and \( C \) are nontrivial square matrices (that is, with a size greater than zero), \( B \) is a rectangular matrix, and~\( P \) is a permutation matrix (a matrix with exactly one nonzero entry of value \( 1 \) in each row and column).
\end{defi}

The transition matrix of an irreducible Markov chain is itself an irreducible matrix.
\begin{defi}[Spectral radius]
 \label{defradioespectral}
 For a matrix \( M\in \R^{m\times m}\) with spectrum (i.e., the set of its eigenvalues) \( \sigma(M) \), the {spectral radius} is defined as
\( \rho(M)\coloneq \sup_{\lambda\in\sigma(M)}|\lambda| \). 
\end{defi}

Now, we recall the Perron--Frobenius theorem that will be instrumental in what follows, we refer the reader to \cite{nnm}.
\begin{teo}[Perron–Frobenius theorem for irreducible nonnegative matrices]
 \label{PerronFrobenius}
 Let \( M\in \R^{m\times m}\) be an irreducible nonnegative 
 matrix with spectral radius 
 \( \rho (M)=r \). Then, the following hold:
 \begin{enumerate}[\rm i)]
 \item 
 The spectral radius is positive, \( r>0 \), and \( r\) is a simple eigenvalue of the matrix \( M \). 
 \item Both right and left eigenspaces corresponding to \( r\)
 are one-dimensional.
 \item The right and left eigenvectors corresponding to \( r\) have
 all their components of the same sign, which can be chosen to be positive. 
 \item Moreover, these are the only eigenvectors whose components can be all positive.
 \end{enumerate}
\end{teo}
 The spectral radius is called the Perron–Frobenius eigenvalue.

\subsubsection{Total positivity}
Totally nonnegative matrices are those with all their minors nonnegative
\cite{Fallat-Johnson,Gantmacher-Krein}
and the set of nonsingular TN matrices are denoted by InTN. 
A matrix is oscillatory %matrices
\cite{Gantmacher-Krein} 
if it is totally nonnegative, irreducible 
and nonsingular \cite{Fallat-Johnson}.
Notice that the set of oscillatory matrices is denoted by IITN (irreducible invertible totally nonnegative) in \cite{Fallat-Johnson}. An oscillatory matrix~\( A\) is equivalently defined as a totally non negative matrix \( A\) such that for some \( n\) we have that \( A^n\) is totally positive (all minors are positive). From Cauchy--Binet Theorem one can deduce the invariance of these sets of matrices under the usual matrix product. Thus, following \cite[Theorem 1.1.2]{Fallat-Johnson} the product of matrices in InTN is again InTN.

The total positivity of matrices is related to orthogonal polynomials and birth-death processes; see, for example, \cite{Viennot} and \cite{Sokal}. It is important to note that this topic has a direct connection to calculating weighted paths, which is a part of enumerative combinatorics. Recently, substantial progress has been made in understanding the connections between multiple orthogonal polynomials and total positivity. In \cite{stbbm0}, bounded oscillatory banded Hessenberg matrices admitting a positive bidiagonal factorization were connected with multiple orthogonal polynomials through a Favard spectral theorem and a Gauss quadrature. The corresponding Karlin–McGregor representation was described for Markov chains composed of pure death and pure birth chains. Then, in \cite{stbbm,stbbm-nuevo}, the discussion was extended to the bounded banded case and mixed multiple orthogonal polynomials.

\section{Orthogonal polynomials and birth and death Markov chains}

We are now ready to connect all these concepts with orthogonal polynomial theory. Let's start with some basic definitions.

\subsection{Orthogonal polynomials and Jacobi matrices}
A sequence of monic polynomials \( \left\lbrace p_n(x)\right\rbrace_{n=0}^N \), with \( \deg p_n =n\) and \( N\in\mathbb N_0\) or \( N=\infty \), is said to be orthogonal with respect to a weight function \( w:\Delta\subseteq\R\to\R_{\geq 0}\) if it satisfies
\begin{align*}
%\begin{aligned}
 \int_{\Delta}x^jp_n(x)w(x)\d x&=0, & j&\in\{0,\ldots ,n-1\} .
%\end{aligned}
\end{align*}
Similarly, the sequence satisfies discrete orthogonality with respect to a weight function \( w:\Delta\subseteq\mathbb Z\rightarrow\R_{\geq 0}\) if it satisfies
\begin{align*}
%\begin{aligned}
 \sum_{k\in\Delta}k^jp_n(k)w(k)&=0, && j\in\{0,\ldots ,n-1\} .
%\end{aligned}
\end{align*}
These orthogonality relations can be equivalently expressed as
\begin{align*}
 \int_{\Delta}p_n(x)p_m(x)w(x)\d x&=\delta_{n,m}h_n, & \sum_{k\in\Delta}p_n(k)p_m(k)w(k)=h_n\delta_{n,m},
\end{align*}
respectively. Here, \( h_n>0\) represents the squared \( L^2\) norm of the polynomial \( p_n \). 

Both continuous and discrete orthogonal polynomials satisfy a three-term recurrence relation of the form 
\begin{align}
 \label{recurrenciaescalar}
%\begin{aligned}
	 xp_n(x)&=p_{n+1}(x)+b_n p_n(x)+c_n p_{n-1}, & n&\in \{0,\ldots,N-1\} ,
%\end{aligned}
\end{align}
%for \( n\in \{0,\ldots,N-1\} \), 
with \( p_{-1}=0 \).
For the applications we will require \( c_n > 0\) and \( b_n \geq 0 \), to ensure that the matrix is nonnegative. 
It is worth noting that
\( c_n=\dfrac{h_{n+1}}{h_{n}} \).
In matrix form, this three-term recurrence relation can be written as
\begin{align*}
\begin{aligned}
 J \begin{bNiceMatrix}
 p_0(x)\\ p_1(x)\\ p_2(x)\\ \Vdots
 \end{bNiceMatrix}
 &=x\begin{bNiceMatrix}
 p_0(x)\\ p_1(x)\\ p_2(x)\\ \Vdots
 \end{bNiceMatrix}, 
 & 
 J & \coloneq 
 \begin{bNiceMatrix}
 b_0 & 1 & 0 & \Cdots[shorten-end=6pt]&\\
 c_1 & b_1 & 1 & \Ddots[shorten-end=8pt]&\\
 0 & c_2 & b_2 & \Ddots[shorten-end=8pt]&\\
 \Vdots[shorten-start=5pt,shorten-end=2pt]& \Ddots[shorten-end=-2pt] & \Ddots[shorten-end=-2pt] &\Ddots[shorten-end=7pt] &
 \end{bNiceMatrix},
\end{aligned}
\end{align*}
if the sequence is infinite. If the sequence is finite or we want to truncate it, we have
\begin{align}\label{eq:specralJ_dm}
 J_m
 \begin{bNiceMatrix}
 p_0(x)\\ p_1(x)\\ \Vdots\ \\\\ p_{m-1}(x)
 \end{bNiceMatrix}
 &=x\begin{bNiceMatrix}
 p_0(x)\\ p_1(x)\\ \Vdots\ \\\\ p_{m-1}(x)
 \end{bNiceMatrix}
 -\begin{bNiceMatrix}
 0\\ \Vdots\\ \\ 0\\ p_{m}(x)
 \end{bNiceMatrix}, 
\end{align}
where 
 \begin{align}
 J_m&\coloneq \begin{bNiceMatrix}[columns-width=.65cm]
 b_0 & 1 & 0 & \Cdots & & 0\\
 c_1 & b_1 & 1 & \Ddots & & \Vdots\\
 0 & c_2 & b_2 & \Ddots & & \\
 \Vdots & \Ddots[shorten-end=7pt]& \Ddots[shorten-end=4pt]&\Ddots[shorten-end=2pt] & & 0\\[4pt]
 & & & & b_{m-2} & 1\\
 0 & \Cdots & & 0 & c_{m-1} & b_{m-1}
 \end{bNiceMatrix}, 
\end{align}
for \( m\in\{1,\ldots,N\} \). 
The matrices \( J_m\) and \( J\) are known as Jacobi matrices and are irreducible due to the nonzero coefficients \( c_n \). 

\begin{rem}
 Notice that the Jacobi matrix is symmetrizable and its entries satisfy \( h_k^{-1}J_{k,l}=h_l^{-1}J_{l,k}\) for \( k,l\in\N_0 \). 
\end{rem}

\begin{rem}
 For the Favard theorem, the requirement for a Jacobi matrix is that \( c_n > 0\) and \( b_n \in \mathbb{R} \). However, for applications to Markov chains, we need an additional condition, namely \( b_n \geq 0 \), to ensure that the matrix is nonnegative. It is worth noting that for a given
 bounded
Jacobi matrix \( J \), there exists a number \( b\) such that for \( s \geq b \), the matrix \( J_m+sI_m\) becomes nonnegative.
Moreover, in this particular case, there exists a number
\(\displaystyle \tilde{b} = \max_{j \in \{ 1, \ldots , m \}} | x_{m,j}| \) 
(where %by
\(
 %\lbrace 
 x_{m,j} 
 %\rbrace^m_{j=1}
 \)
 are the
%we meant the 
%set of increasing 
zeros of the 
%\( m\)-th 
polynomial \( p_m \))
such that for 
\( s > \tilde{b} \), 
the matrix \( J_m+sI_m\) is oscillatory and admits a positive bidiagonal factorization~(PBF), see for instance~\cite{stbbm}.
\end{rem}

\subsection{Spectral properties}
Let's denote by \( \lbrace x_{n,i}\rbrace^n_{i=1}\) the set of increasing zeros of the \( n\)-th polynomial \( p_n \). Due to the known properties of orthogonal polynomials, we know that these zeros are simple and contained in the interior of \( \Delta \). Moreover, the zeros of \( p_{n-1}\) interlace the zeros of \( p_n \), i.e.,
\begin{align*}
 x_{n,i-1}&<x_{n-1,i-1}<x_{n,i},
\end{align*}
for \( n\in \{2,\ldots,N\}\) and \( i\in\{2,\ldots,n\} \). 
Looking at Equation \eqref{eq:specralJ_dm}, it is easy to notice that the eigenvalues of \( J_m\) are exactly the zeros of \( p_m \), while the associated right eigenvectors are given by
\begin{align}
 \label{eigenvectorrightescalar}
 &\begin{bNiceMatrix}
 p_0(x_{m,i})\\
 \Vdots\\
 p_{m-1}(x_{m,i})
 \end{bNiceMatrix}, 
\end{align}
for \( i\in\{1,\ldots,m\} \). Similarly, the left eigenvector associated with the eigenvalue \( x_{m,i}\) is given by
\begin{align}
 \label{eigenvectorleftescalar}
 &\begin{bNiceMatrix}
 \dfrac{p_0(x_{m,i})}{h_0}&\Cdots &\dfrac{p_{m-1}(x_{m,i})}{h_{m-1}}
 \end{bNiceMatrix}, 
\end{align}
for \( i\in\{1,\ldots,m\} \). In terms of these eigenvectors, we define the matrices
\begin{align*}
 \mathcal U&\coloneq\begin{bNiceMatrix}[]
 p_0(x_{m,1}) & \Cdots & p_0(x_{m,m})\\
 \Vdots & & \Vdots\\
 p_{m-1}(x_{m,1}) & \Cdots & p_{m-1}(x_{m,m})
 \end{bNiceMatrix}, \\ \mathcal V&\coloneq \begin{bNiceMatrix}[]
 \frac{p_0(x_{m,1})}{h_0\sum^m_{l=1}p^2_{l-1}(x_{m,1})h^{-1}_{l-1}} & \Cdots & \frac{p_{m-1}(x_{m,1})}{h_{m-1}\sum^m_{l=1}p^2_{l-1}(x_{m,1})h^{-1}_{l-1}}\\
 \Vdots & & \Vdots\\
 \frac{p_0(x_{m,m})}{h_0\sum^m_{l=1}p^2_{l-1}(x_{m,m})h^{-1}_{l-1}} & \Cdots & \frac{p_{m-1}(x_{m,m})}{h_{m-1}\sum^m_{l=1}p^2_{l-1}(x_{m,m})h^{-1}_{l-1}}
 \end{bNiceMatrix}.
\end{align*}
Notice that
\begin{align*}
 \mathcal U\mathcal V=\mathcal V\mathcal U=I_m.
\end{align*}
Hence, the rows of \( \mathcal V\) are biorthogonal to the columns of \( \mathcal U \), and vice versa. Also, the rows of \( \mathcal U\) are biorthogonal to the columns of \( \mathcal V \). 

 These spectral objects can be used to diagonalize \( J_m\) so that
\begin{align*}
 J_m=\mathcal U \begin{bNiceMatrix}[columns-width=auto]
 x_{m,1} & 0&\Cdots &&0\\
 0 & x_{m,2} & \Ddots& &\Vdots\\
 \Vdots &\Ddots&\Ddots[,shorten-end=-5pt] &&\Ddots\\
 &&&&0\\[4pt]
 0&\Cdots&&0&x_{m,m}
 \end{bNiceMatrix}\mathcal V,
\end{align*}
which provides a useful spectral representation for obtaining an expression for the entries of any power of 
\begin{align*}
 J_m^r=\mathcal U \begin{bNiceMatrix}[columns-width=auto]
 x_{m,1}^r & 0&\Cdots &&0\\
 0 & x_{m,2}^r & \Ddots& &\Vdots\\
 \Vdots &\Ddots[shorten-end=5pt]&\Ddots[shorten-end=-2pt] &&\Ddots\\
 &&&&0\\[4pt]
 0&\Cdots&&0&x_{m,m}^r
 \end{bNiceMatrix}\mathcal V,
\end{align*}
namely
\begin{align}
 \label{Jescalaralar}
 (J_m^r)_{i,j}&=\sum^{m}_{k=1}x^r_{m,k}\dfrac{p_{i-1}(x_{m,k})p_{j-1}(x_{m,k})}{h_{j-1}}\dfrac{1}{\sum^{m}_{l=1}p^2_{l-1}(x_{m,k})h^{-1}_{l-1}}, 
\end{align}
for \( r\in\mathbb N\) and \( i,j\in\N_0 \). 

\subsection{Markov chains}
But let's go back to expression \eqref{eq:specralJ_dm} to see how the Jacobi matrix \( J_m\) can be linked to the birth and death Markov chains theory.
We define
 \begin{align*}
 \sigma_m&\coloneq \diag 
 \left[ \begin{NiceMatrix} p_{0}(x_{m,m}) & \cdots & p_{m-1}(x_{m,m}) \end{NiceMatrix} \right].
 \end{align*}

\begin{pro}
\label{propestocasticaescalar}
Let be the recurrence relation expressed in \eqref{eq:specralJ_dm} with
 \( J_m\) a nonnegative matrix and \( \lbrace{x_{m,i}}\rbrace_{i=1}^m\) the set of increasing zeros of \( p_m\) and \( x_{m,m}>0 \). Then, 
\begin{align}
\label{estocasticaescalar}
 P_m&\coloneq \dfrac{1}{x_{m,m}}\sigma_m^{-1}J_m\sigma_m,
\end{align}
is a stochastic matrix.
\end{pro}

\begin{proof}
 We have to check that \( P_m\) is nonnegative and the sum of every one of its rows equals \( 1 \). That is, conditions \eqref{nonegatividad} and \eqref{estocasticidad}.
 In the one hand, using \eqref{eq:specralJ_dm} we get \( P_me=e \). 
 So \( P_m\) satisfies \eqref{estocasticidad}.
 On the other hand, since \( J_m\) is nonnegative by hypothesis and \( x_{m,m}>0 \), the only condition to ensure \( P_m\) is nonnegative is that \( p_0(x_{m,m}),\ldots,p_{m-1}(x_{m,m})\) have all the same sign. But notice these are exactly the components of the eigenvector of \( J_m\) associated to the eigenvalue \( x_{m,m}\); which, in addition, is the spectral radius of \( J_m\) (see Definition \ref{defradioespectral}). Since \( J_m\) is nonnegative and irreducible, the Perron--Frobenius theorem applies and we find that \( p_0(x_{m,m}),\ldots,p_{m-1}(x_{m,m})\) have all the same sign by (iii) in Perron--Frobenius Theorem \ref{PerronFrobenius}. So \( P_m\) satisfies \eqref{nonegatividad}.
\end{proof}

The stochastic matrix \( P_m\) reads
\begin{align*}
P_m =
 \dfrac{1}{x_{m,m}}
 \setlength{\arraycolsep}{1pt}
 \begin{bNiceMatrix}
 b_0 & \frac{p_{1}(x_{m,m})}{p_{0}(x_{m,m})} & 0 & \Cdots & & 0\\
 \frac{p_{0}(x_{m,m})}{p_{1}(x_{m,m})}c_1 & b_1 & \frac{p_{2}(x_{m,m})}{p_{1}(x_{m,m})} & \Ddots & & \Vdots\\
 0 & \frac{p_{1}(x_{m,m})}{p_{2}(x_{m,m})}c_2 & b_2 & \Ddots[shorten-start=-1pt,shorten-end=-10pt] & & \\
 \Vdots & & &\Ddots[shorten-end=-30pt] & & 0\\
 & & \Ddots& \Ddots[shorten-start=-4pt,shorten-end=-28pt] & \hspace*{30pt}b_{m-2} & \frac{p_{m-1}(x_{m,m})}{p_{m-2}(x_{m,m})}\\
 0 & \Cdots & & 0& \frac{p_{m-2}(x_{m,m})}{p_{m-1}(x_{m,m})}c_{m-1} & b_{m-1}
 \end{bNiceMatrix}.
\end{align*}

\begin{rem}
Notice that, alternatively to the Perron--Frobenius theorem, we can use the fact that the zeros of the orthogonal interlace and we take the maximal one.
If we had taken any of the other zeros of \( p_m\) to construct \( P_m \), nonnegativeness is impossible to achieve due to (iv) in Perron--Frobenius Theorem \ref{PerronFrobenius}.
\end{rem}

This matrix \( P_m\) describes a finite \( m\) state Markov chain where there is only transition probability up to first neighbors, i.e., a birth and death Markov chain. 
\begin{pro}\label{pro:uni_class}
 The stochastic matrix \( P_m\) has only one class. 
 \end{pro}
 
 \begin{proof}
 We need to show that all states are communicated. That is, for any couple \( (i,j)\in\N_0^2\) we can find \( n\in\N\) such that \( (P^n_m)_{i,j}>0 \). 
 From the band structure of the stochastic matrix \( P_m\) and the fact that the extreme diagonals have non zero entries one can show that such \( n\)
 does exist.
 As \( P_m\) is tridiagonal, \( P_m^2\) is pentadiagonal, and in general \( P_m^n\) can have up to \( (2n+1)\) diagonal.
 \end{proof}
Since \( P_m\) is, actually, a matrix conjugation of \( J_m \), we can use Equation \eqref{Jescalaralar} to find the following representation formula for the transition probabilities after \( r\) steps.

\begin{pro} The spectral representation for the iterated probabilities 
 \begin{align}
 \label{formularepresentacionescalar}
 (P_m^r)_{i,j}&=\dfrac{1}{x^r_{m,m}}\dfrac{p_{j-1}(x_{m,m})}{h_{j-1}p_{i-1}(x_{m,m})}\sum^{m}_{k=1}x^r_{m,k}\dfrac{p_{i-1}(x_{m,k})p_{j-1}(x_{m,k})}{\sum^{m}_{l=1}p^2_{l-1}(x_{m,k})h^{-1}_{l-1}},
 \end{align}
for a number of transitions \( r\in\mathbb N_0 \), is satisfied.
\end{pro}

With this expression we can prove the following result.
\begin{pro}
 \label{prorecurrentesescalar}
 The Markov chain described by the stochastic matrix \( P_m \), defined by Equation \eqref{estocasticaescalar}, is recurrent.
\end{pro}

\begin{proof}
 To prove it is recurrent we are going to show state \( i\) is recurrent, \( i\in\lbrace{1,\ldots,m}\rbrace \). Recalling Remark~\ref{condicionsuficienteparalarecurrencia} we deduce that the only thing to check is that the limit
 \begin{align*}
 \lim_{s\rightarrow1^{-}}(P_{m}(s))_{i,i}=\lim_{s\rightarrow1^{-}}\sum_{r=0}^{\infty}{(P^r_m)_{i,i}s^r}=\sum_{r=0}^{\infty}{(P^r_m)_{i,i}} ,
 \end{align*}
 diverges. Replacing \( (P_m^r)_{i,i}\) with the previous formula (\ref{formularepresentacionescalar}) we conclude that
 \begin{align*}
 \sum_{r=0}^{\infty}(P_m^r)_{i,i}=\sum^{m}_{k=1}\dfrac{p^2_{i-1}(x_{m,k})}{h_{i-1}\sum^{m}_{l=1}p^2_{l-1}(x_{m,k})h^{-1}_{l-1}}\sum_{r=0}^{\infty}\left(\dfrac{x_{m,k}}{x_{m,m}}\right)^r.
 \end{align*}
 which diverges when \( k=m \). 
\end{proof}

\begin{lemma}
 If there exists a state \( i\in\{1,\ldots,m\}\) such that \( b_{i}>0\) the Markov chain associated with \( P_m\) is aperiodic (period \( 1\)). If \( b_{i}=0 \), \( i\in\{1,\ldots,m\} \), then the Markov chain has period \( 2 \). 
\end{lemma}

\begin{proof}
 If \( b_i>0\) by definition the state \( i\) has period \( 1 \). As all states are in the same class the statement follows. However, if all \( b_i=0 \), then it can be checked that odd powers of \( P_m\) have only their odd diagonals with nonzero entries, and even powers of \( P_m\) have only their even diagonals with nonzero entries. 
Hence, \( (P_m^{2n})_{i,i}>0\) and \( (P_m^{2n+1})_{i,i}=0 \), so that we have period \( 2 \). 
\end{proof}

\begin{coro}
 The Markov chain described by the stochastic matrices \( P_{m} \), defined at Proposition \ref{propestocasticaescalar} is ergodic if and only if there exists a state \( i\in\{1,\ldots,m\}\) such that \( b_{i}>0 \). 
\end{coro}

\begin{proof}
 It follows immediately from the fact that Markov chains are ergodic if and only if they are aperiodic and recurrent.
\end{proof}
Let's discuss some more properties of this stochastic matrix \( P_m \). Looking at the expressions for the right and left eigenvectors of \( J_m\) (Equations \eqref{eigenvectorrightescalar} and \eqref{eigenvectorleftescalar}) and how \( P_m\) is defined (Equation \eqref{estocasticaescalar}), it is easy to deduce the following result:

\begin{pro}
 \label{proestacionarioescalar}
 Let \( P_m\) be the stochastic matrix defined in Equation \eqref{estocasticaescalar}. Then, the probability vector
 \begin{align}
 \label{estacionarioescalar}
 \pi_m\coloneq \dfrac{1}{\sum^{m}_{l=1}p^2_{l-1}(x_{m,m})h^{-1}_{l-1}} \begin{bNiceMatrix}
 \dfrac{p_0^2(x_{m,m})}{h_0} & \Cdots &\dfrac{p_{m-1}^2(x_{m,m})}{h_{m-1}}
 \end{bNiceMatrix} ,
 \end{align}
 is a steady state for \( P_m \), i.e., \( \pi_mP_m=\pi_m \). 
\end{pro}

\begin{proof}
 It follows from
 \begin{align*}
 \pi_mP_m&=\frac{1}{x_{m,m}}\pi_m\sigma_m^{-1} J_m\sigma_m=\frac{1}{x_{m,m}} \dfrac{1}{\sum^{m}_{l=1}p^2_{l-1}(x_{m,m})h^{-1}_{l-1}}\begin{bNiceMatrix}
 \dfrac{p_0(x_{m,m})}{h_0} & \Cdots &\dfrac{p_{m-1}(x_{m,m})}{h_{m-1}}
 \end{bNiceMatrix}J_m\sigma_m\\&= \dfrac{1}{\sum^{m}_{l=1}p^2_{l-1}(x_{m,m})h^{-1}_{l-1}}
 \begin{bNiceMatrix}
 \dfrac{p_0(x_{m,m})}{h_0} & \Cdots &\dfrac{p_{m-1}(x_{m,m})}{h_{m-1}}
 \end{bNiceMatrix}\sigma_m\\&=\pi_m.
 \end{align*}
\end{proof}

\begin{coro}
 Let us assume a Markov chain with the stochastic matrix \( P_m\) defined in Equation \eqref{estocasticaescalar}. Then, the steady state \( \pi_m\) in Equation \eqref{estacionarioescalar} is unique, and the expected return time is given by
\begin{align*}
%\begin{aligned}
 (\bar t_m)_j &= \frac{ h_{j-1}\sum^{m}_{l=1}p^2_{l-1}(x_{m,m})h^{-1}_{l-1}}{p^2_{j-1}(x_{m,m})},
 & j &\in \{ 1, \ldots , m \}.
%\end{aligned}
\end{align*}
\end{coro}

\begin{proof}
 As was proved, the Markov chain is recurrent, so the steady state is unique, and its entries are reciprocal to these expected return times.
\end{proof}

\begin{coro}
 Let us assume a Markov chain with the stochastic matrix \( P_m\) defined in Equation \eqref{estocasticaescalar}. Then, the steady state \( \pi\) is the limit of the iterated probabilities
 \begin{align*}
 (\pi_m)_j=\lim_{r\to\infty}(P_{m}^r)_{i,j} , && i, j \in \{ 1, \ldots , m \} .
 \end{align*}
 Moreover, the convergence is geometric in terms of the second largest zero:
 \begin{align*}
 (P_{m}^r)_{i,j}-(\pi_m)_j
 & =
%\xrightarrow[r\to\infty]{}
 \dfrac{p_{j-1}(x_{m,m})}{p_{i-1}(x_{m,m})h_{j-1}}\dfrac{p_{i-1}(x_{m,m-1})
 p_{j-1}(x_{m,m-1})}{\sum^{m}_{l=1}p_{l-1}^2(x_{m,m-1})h_{l-1}^2}\left(\dfrac{x_{m,m-1}}{x_{m,m}}\right)^r
 + \cdots \xrightarrow[r\to\infty]{} 0 .
 \end{align*}
\end{coro}

\begin{proof}
 It follows from Equation \eqref{formularepresentacionescalar}.
\end{proof}

\begin{rem}
 Then, this steady state is the equilibrium state to which the evolution of any probability vector will tend to. This happens even when the chain is not ergodic, i.e., the Markov chain has period \( 2 \). 
\end{rem}

\begin{pro}[Reversibility]\label{pro:reversibility}
The Markov chain with stochastic matrix \( P_m\) is reversible, meaning that the chain and its time-reversed version are statistically identical.
\end{pro}

\begin{proof}
Let us recall that the Jacobi matrix is symmetrizable, i.e., \( h_i^{-1}J_{i,j}=h_j^{-1} J_{j,i} \). The entries of the stochastic matrix \( P_m\) are given by 
\begin{align*}
 (P_m)_{i,j}=\frac{1}{x_{m,m}}\frac{p_j(x_{m,m})}{p_i(x_{m,m})}J_{i,j}. 
\end{align*}
Therefore, the symmetrization of the Jacobi matrix leads to the following relation:
\begin{align*}
 \frac{(p_i(x_{m,m}))^2}{h_i} (P_m)_{i,j}=\frac{(p_j(x_{m,m}))^2}{h_j} (P_m)_{j,i}.
\end{align*} 
This relation implies that the steady distribution \( \pi\) in Equation \eqref{estacionarioescalar} satisfies the detailed balance equation \( (\pi_m)_{i+1}(P_m)_{i,j}=(\pi_m)_{j+1}(P_m)_{j,i} \). 
\end{proof}

Let's consider the possibility of a positive bidiagonal factorization (PBF) of \( J_m\) in the form:
\begin{align}
 \label{factorizacionbidiagonalescalar}
 J_{m}&=L_mU_m, 
\end{align}
with
\begin{align*}
%\begin{aligned}
 L_m&\coloneq
 \begin{bmatrix}[]
 1 & 0 &\Cdots & & & 0\\
 a_2 & 1 & \Ddots& & & \Vdots\\
 0 & a_4 & \Ddots & & & \\
 \Vdots & \Ddots[shorten-end=2pt]& \Ddots[shorten-end=-8pt]& & & \\
 & & & & & 0\\[3pt]
 0 & \Cdots& & 0 & a_{2m-2} & 1\\
 \end{bmatrix}, & U_m&\coloneq 
 \begin{bmatrix}[]
 a_1 & 1 & 0 & \Cdots & & 0\\
 0 & a_3 & 1 & \Ddots & & \Vdots\\
 \Vdots & \Ddots[shorten-end=3pt] &\Ddots[shorten-end=-4pt]& \Ddots& & \\
 & & & & &0 \\[5pt]
 & & & & & 1\\
 0 & \Cdots& & & 0 & a_{2m-1}\\
 \end{bmatrix},
%\end{aligned}
\end{align*}
where \( a_n>0\) for \( n\in\{1,\ldots,2m-1\} \). 
Let us introduce the following notation:
\begin{align*}
%\begin{aligned}
 &D_{m}\coloneq 
 \diag
 \left[ \begin{NiceMatrix}
 \dfrac{1}{d_{0}} &\cdots &\dfrac{1}{d_{m-1}} 
 \end{NiceMatrix} \right], &
 d_{n}&\coloneq a_{2n+1}p_{n}(x_{m,m})+p_{n+1}(x_{m,m}), 
%\end{aligned}
\end{align*}
for \( n\in\{0,\ldots,m-1\} \). Note that in this tridiagonal scenario, this coincides with the \( LU\) factorization, which holds whenever all the leading principal minors are not zero; i.e., \( \det J_m \neq 0 \), \( m \in \mathbb{N} \). Using the spectral decomposition, this is equivalent to the fact that \( 0\) is not among the zeros of the orthogonal polynomials \( p_n(x) \). For example, if \( b_0=0 \), then \( \det J_0=0\) and \( 0\) is a zero of \( p_1 \). Hence, for the period \( 2\) case, such factorization does not exist. That is, the nonergodic case does not admit a stochastic factorization.

From this perspective, the shift \( J\mapsto J+sI\) translates the zeros of the orthogonal polynomials. As we know, for sufficiently large \( s \), this situation is resolved, and we can ensure that all the zeros are positive. 
%In fact, by the Gershgorin Circle Theorem we only have to consider \( s > \| J \| \). 
The translated Jacobi matrix becomes an oscillatory matrix, and the orthogonal polynomials become \( p_n(x-s)\) with the zeros shifted by \( x_{m,n}\mapsto x_{m,n}+s \). 

\begin{pro}[Pure Birth/Pure Death Factorization]
 \label{profactorizacionestocasticaescalar}
 Let us assume that the Jacobi matrix \( J_m\) admits a PBF. Then, we have the following stochastic bidiagonal factorization of \( P_m\):
 \begin{align*}
 P_m=\Pi_m\Upsilon_m,
 \end{align*}
 with pure birth and pure death stochastic matrices given by
 \begin{align}
 \label{factorizacionbidiagonalescalarestocastica}
 %\begin{aligned}
 \Pi_{m}&\coloneq\dfrac{1}{x_{m,m}}\sigma^{-1}_{m}L_{m}D^{-1}_{m}, &
 \Upsilon_{m}&\coloneq D_{m}U_m\sigma_{m}.
 %\end{aligned}
 \end{align}
\end{pro}

\begin{proof}
Notice that
\begin{align*}
 P_{m}&=\dfrac{1}{x_{m,m}}\sigma^{-1}_{m}J_m\sigma_{m}=\dfrac{1}{x_{m,m}}\sigma^{-1}_{m}L_{m}U_m\sigma_{m}=\dfrac{1}{x_{m,m}}\sigma^{-1}_{m}L_{m}
 D^{-1}_{m} D_{m}U_m\sigma_{m}.
\end{align*}
Here, \( D_{m}\) is defined in such a way that \( \Upsilon_{m}\) is stochastic. Hence, \( \Upsilon_m\) is stochastic by definition.
 Now, as \( P_{m}\) is stochastic we have 
\begin{align*}
%\text{\( P_m \) is stochastic and} &&
\begin{bNiceMatrix}
 1 \\ \Vdots \\ 1
 \end{bNiceMatrix}
=
P_m 
\begin{bNiceMatrix}
1 \\ \Vdots \\ 1
 \end{bNiceMatrix}
= 
\Pi_m \Upsilon_m
\begin{bNiceMatrix}
1 \\ \Vdots \\ 1
 \end{bNiceMatrix}
=
\Pi_m 
\begin{bNiceMatrix}
1 \\ \Vdots \\ 1
 \end{bNiceMatrix} ,
\end{align*}
and so \( \Pi_{m}\) is stochastic. 
%since the product of stochastic matrices is stochastic as well. 
It should be noted that \( \Pi_{m}\) is lower bidiagonal, while \( \Upsilon_{m}\) is upper bidiagonal.
\end{proof}

\begin{rem}
Here, \( \Pi_m\) represents the transition matrix for a pure birth Markov chain, while \( \Upsilon_m\) represents the transition matrix for a pure death Markov chain. Therefore, our Markov chain is a composition of a pure death and a pure birth process.
\end{rem}

\section[Birth and death finite Markov chains in the Askey scheme]{Finite Markov chains in the Askey scheme}

We can now apply this general paradigm to the families of orthogonal polynomial descendants of Hahn in the Askey scheme \cite{askeyescalar} that possess nonnegative Jacobi matrices. In contrast to the multiple case, all the ``classical'' families, namely Hahn, Jacobi, Meixner, Laguerre, Kravchuk, Charlier, and Hermite, fulfill this requirement.
\begin{center}
 \begin{tikzpicture}[node distance=3cm]
 \node[fill=ForestGreen!35,block] (a) {Jacobi}; 
 \node[ fill=ForestGreen!35,block, right of=a] (b) {Meixner};
 \node[ fill=ForestGreen!35,block, right of = b] (c) {Kravchuk};
 \node[ fill=ForestGreen!35,block,above of =b ] (d) {Hahn};
 \node[fill=ForestGreen!35,block, below of = a] (e) {Laguerre};
 \node[ fill=ForestGreen!35,block, below of =b, below= 1cm] (f) {Hermite};
 \node[ fill=ForestGreen!35, block, below of =c] (g) {Charlier};
 \draw[-latex] (d)--(a); 
 \draw[-latex] (d)--(b);
 \draw[-latex] (d)--(c);
 \draw[-latex] (a)--(f);
 \draw[-latex] (b)--(e);
 \draw[-latex] (c)--(f);
 \draw[-latex] (c)--(g);
 \draw[-latex] (b)--(g);
 \draw[-latex] (a)--(e);
 \draw[-latex] (e)--(f);
 \draw[-latex] (g)--(f);
 \draw (3,-6) node
 {\begin{minipage}{0.5\textwidth}
 \begin{center}\small
 \textbf{Descendants of Hahn in the Askey scheme}
 \end{center}
 \end{minipage}};
 \end{tikzpicture}
\end{center}

Before we begin with the examples, it is important to note that all the polynomials we will be studying can be expressed using the generalized hypergeometric series, as described in references \cite{andrews,slater}, 
\begin{align}\label{eq:generalized_hypergeometric_series}
%\begin{aligned}
 \pFq{p}{q}{b_1,\ldots,b_p}{c_1,\ldots,c_q}{x}&\coloneq \sum_{l=0}^{\infty}\dfrac{(b_1)_l\cdots(b_p)_l}{(c_1)_l\cdots(c_q)_l}\dfrac{x^l}{l!}, & b_1,\ldots,b_p,c_1,\ldots, c_q\in\C,
%\end{aligned}
\end{align}
which depend on the Pochhammer symbols
\begin{align}\label{eq:Pochhammer}
 (x)_n=\dfrac{\Gamma(x+n)}{\Gamma(x)}=\begin{cases}
 x(x+1)\cdots(x+n-1)&\text{if \( n\in\N \), }\\
 1 & \text{if \( n=0 \), }
 \end{cases}
\end{align}
in terms of the Eulerian gamma function.
One important property of the Pochhammer symbols is that for \( m\in\mathbb N_0\):
\( (-m)_n=
 0\) if \( m<n \). 
This implies that if any of the previous upper arguments \( b_1,\ldots,b_p\) is a negative integer, then the series is a finite sum; which is what happens in all the polynomial families.

Now, let's provide explicit examples using the explicit hypergeometric expressions and the previous results for the families of Hahn, Jacobi, Meixner, and Laguerre. These families have nonnegative Jacobi matrices. We will consider a \( 5 \)-state Markov chain and choose specific parameters for each orthogonal polynomial family. We will provide the pure birth-pure death decomposition as well as the stationary distribution and expected return times for each example.

The strategy is as follows: our goal is to obtain numerical approximations of the largest zero \( x_{m,m}\) of the orthogonal polynomial \( p_m(x) \). Using the explicit hypergeometric representation of the orthogonal polynomials and the given nonnegative expression for the recursion coefficients, we calculate numerical approximations of the associated stochastic matrix \( P_m \), the steady state, the expected return times, and the stochastic factorization into pure birth and pure death factors. To perform these computations, we 
apply a specialized Mathematica code that has been designed for this specific purpose. See Declarations section at the end of the paper.

\subsection{Hahn finite Markov chains}

The monic Hahn orthogonal polynomials \cite[Chapter \(6.2\)]{Ismail} admit the hypergeometric series expression
\begin{align*}
 Q_n(x;\alpha,\beta,N)\coloneq &\dfrac{(\alpha+1)_n(-N)_n}{(\alpha+\beta+n+1)_n}\pFq{3}{2}{-n,-x, \alpha+\beta+n+1}{-N, \alpha+1}{1}, &&
 n\in\{0,\ldots,N\} ,
\end{align*}
%for \( n\in\{0,\ldots,N\} \), 
with \( N\in\mathbb N_0\) and \( \alpha,\beta>-1\) and satisfy the discrete orthogonality relations
\begin{align*}
 \sum_{k=0}^N (-N+k)_jQ_n(k;\alpha,\beta,N)w(k;\alpha,\beta,N)&=0, 
 & j&\in\{0,\ldots, n-1\} ,
\end{align*}
%for \( j\in\{0,\ldots, n-1\} \), 
with respect to the weight function
\( w %\textcolor{blue}{
(x;\alpha,\beta,N) %}
 =
 \dfrac{\Gamma(\alpha+x+1)}{\Gamma(x+1)\Gamma(\alpha+1)}
 \dfrac{\Gamma(\beta+N-x+1)}{\Gamma(N-x+1)\Gamma(\beta+1)} \). 
 It is important to notice that the Hahn polynomials does not exist for \( n > N \) as, for each \( N \), the Gram matrix associated with this weight is a \( (N+1) \times (N+1) \) matrix.
They satisfy a three term recurrence relation, see \eqref{recurrenciaescalar} and \eqref{eq:specralJ_dm}, with the coefficients given by
\begin{align*}
 b_n(\alpha,\beta,N)&\coloneq\begin{cases}
 a_{2n}(\alpha,\beta,N)+a_{2n+1}(\alpha,\beta,N), & \text{if \( n\leq N \), } \\
 0 & \text{if \( N<n \), }
 \end{cases}
 \\
 c_n(\alpha,\beta,N)
 &\coloneq\begin{cases}
 a_{2n-1}(\alpha,\beta,N)a_{2n}(\alpha,\beta,N),& \text{if \( n\leq N\)},\\
 0 & \text{if \( N<n \), }
 \end{cases}
\end{align*}
which in this case come already in function of the coefficients of a bidiagonal factorization of the form (\ref{factorizacionbidiagonalescalar})
\begin{align*}
 a_{2n+1}(\alpha,\beta,N)&=\dfrac{(N-n)(\alpha+n+1)(\alpha+\beta+n+1)}{(\alpha+\beta+2n+1)(\alpha+\beta+2n+2)},&
 a_{2n}(\alpha,\beta,N)&=\dfrac{n(\beta+n)(\alpha+\beta+N+n+1)}{(\alpha+\beta+2n)(\alpha+\beta+2n+1)} .
\end{align*}
Therefore, the Jacobi matrix is non negative for all values \( \alpha,\beta>-1 \). We also observe that \( b_n>0 \), 
 \( n\in
 %\N 
 \{ 0, \ldots , N \} \), so that these Markov chains are ergodic.

\begin{exa}
 Following (\ref{estocasticaescalar}) for the choice of parameters \( \alpha=0.5 \), \( \beta=0.75 \), \( N=5\); we get the following \( 5\times5\) stochastic matrix and its pure birth/pure death factorization 
 \eqref{factorizacionbidiagonalescalarestocastica}
 \begin{align*}
 P_5(0.5,0.75,5)&\approx
 \begin{bNiceArray}[small]{ccccc}
 0.46 & 0.54 & 0 & 0 & 0 \\
 0.18& 0.50 & 0.32 & 0 & 0 \\
 0 & 0.29 & 0.51 & 0.20 & 0 \\
 0 & 0 & 0.37& 0.52 & 0.11\\
 0 & 0 & 0 & 0.49& 0.51 \\
 \end{bNiceArray}
\approx
 \begin{bNiceArray}[small]{ccccc}
 1 & 0 & 0 & 0 & 0 \\
 0.39 & 0.61 & 0 & 0 & 0 \\
 0 & 0.60 & 0.40 & 0 & 0 \\
 0 & 0 & 0.76 & 0.24 & 0 \\
 0 & 0 & 0 & 0.94& 0.06 \\
 \end{bNiceArray}
 \begin{bNiceArray}[small]{ccccc}
 0.46 & 0.54 & 0 & 0 & 0 \\
 0 & 0.48 & 0.52 & 0 & 0 \\
 0 & 0 & 0.49 & 0.51 & 0 \\
 0 & 0 & 0 & 0.52 & 0.48 \\
 0 & 0 & 0 & 0 & 1 \\
 \end{bNiceArray}.
 \end{align*}
 The corresponding steady state \eqref{estacionarioescalar} is
 \begin{align*}
 \pi_5(0.5,0.75, %\textcolor{blue}{
 5 %}
 ) \approx
\begin{bNiceMatrix}
 0.11 & 0.31 & 0.35 & 0.19 & 0.04
 \end{bNiceMatrix}.
 \end{align*}
 The expected return times are
 \begin{align*}
%\begin{aligned}
 (\bar t_5)_1&\approx 9.09, & (\bar t_5)_2&\approx 3.23, & (\bar t_5)_3&\approx 2.86, & (\bar t_5)_4&\approx 5.26,& (\bar t_5)_5&\approx 25.00.
%\end{aligned}
 \end{align*}
 \begin{center}
 \tikzset{decorate sep/.style 2 args={decorate,decoration={shape backgrounds,shape=circle,shape size=#1,shape sep=#2}}}
 \begin{tikzpicture}[start chain = going right,
 -Triangle, every loop/.append style = {-Triangle}]
 \foreach \i in {1,...,5}
 \node[state, on chain] (\i) {\i};
 \foreach
 \i/\txt in {1/\(0.54\),2/\(0.32\)/,3/\(0.20\),4/\(0.11\)}
 \draw let \n1 = { int(\i+1) } in
 (\i) edge[bend left,"\txt",color=Periwinkle] (\n1);
 \foreach
 \i/\txt in {1/\(0.18\),2/\(0.29\)/,3/\(0.37\),4/\(0.49\)}
 \draw let \n1 = { int(\i+1) } in
 (\n1) edge[bend left,above, "\txt",color=Mahogany,auto=right] (\i); 
 \foreach \i/\txt in {1/\(0.46\),2/\(0.50\)/,3/\(0.51\),4/\(0.52\),5/\(0.51\)}
 \draw (\i) edge[loop above,color=NavyBlue, "\txt"] (\i);
 
 \end{tikzpicture}
 \begin{tikzpicture}
 \draw (4,-1.8) node
 {\begin{minipage}{0.8\textwidth}
 \begin{center}\small
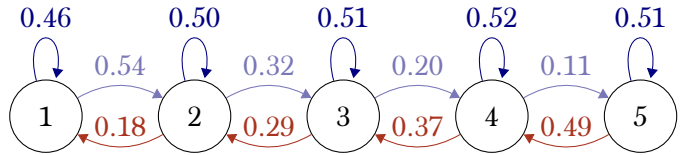
 \textbf{Hahn\( (0.5,0.75, %\textcolor{blue}{
 5 %}
)\) Markov chain diagram}
 \end{center}
 \end{minipage}};
 \end{tikzpicture}
 \end{center}
\end{exa}

\subsection{Jacobi finite Markov chains}

The Jacobi orthogonal 
polynomials 
usually appear defined over the interval \( [-1,1]\) satisfying orthogonality relations with respect to the weight 
function \( (1+x)^\alpha(1-x)^\beta \), see \cite[Chapter \(4\)]{Ismail}. However, here we are going to take the modified version satisfying the orthogonality over the interval \( [0,1] \). This is achieved
by the shift \( x\mapsto y=x+1\) and subsequent rescaling \( y\mapsto z=2y \), that transforms the measure according to \( (1+x)^\alpha(1-x)^\beta\d x\mapsto 2^{\alpha+\beta+1}z^\alpha(1-z)^\beta\d z \). These modifications transform the Jacobi matrix entries 
\( J_{n,m}\mapsto (J_{n,m}+\delta_{n,m})2^{m-n-1} \), so that 
\( b_n\mapsto \frac{1}{2}(b_n+1)\) and \( c_n\mapsto \frac{1}{4}c_n\).

 The hypergeometric expression for the monic Jacobi orthogonal polynomials is:
\begin{align*}
 P_n(x;\alpha,\beta)&\coloneq
 (-1)^n\dfrac{(\alpha+1)_n}{(\alpha+\beta+n+1)_n}\,\pFq{2}{1}{-n,\alpha+\beta+n+1}{\alpha+1}{x},
\end{align*}
for \( n\in\mathbb N_0 \), with \( \alpha,\beta>-1 \). The corresponding orthogonality relations
\begin{align*}
 \int_{0}^{1}x^jP_n(x;\alpha,\beta)w(x;\alpha,\beta)\d x&=0, 
\end{align*}
for \( j\in\{0,\ldots,n-1\} \), with respect to the given weight function
\( w %\textcolor{blue}{
 (x;\alpha,\beta)=
 %}
 x^{\alpha}(1-x)^\beta \). 
A three term recurrence relation, see \eqref{recurrenciaescalar} and \eqref{eq:specralJ_dm}, with coefficients
\begin{align*}
 b_n(\alpha,\beta)&=\dfrac{1}{2}-\dfrac{\beta^2-\alpha^2}{2(2n+\alpha+\beta)(2n+\alpha+\beta+2)},&
 c_n(\alpha,\beta)&=\dfrac{n(n+\alpha)(n+\beta)(n+\alpha+\beta)}{(2n+\alpha+\beta-1)(2n+\alpha+\beta)^2(2n+\alpha+\beta+1)},
\end{align*}
is fulfilled.
Hence, as not all the \( b_n\) can be zero, the chain is ergodic. 
\begin{rem}
 In the standard version for the Jacobi polynomials %with support 
 on \( [-1,1]\) we have 
 \begin{align*}
 b_n(\alpha,\beta)&=\dfrac{\alpha^2-\beta^2}{(2n+\alpha+\beta)(2n+\alpha+\beta+2)},&
 c_n(\alpha,\beta)&=\dfrac{4n(n+\alpha)(n+\beta)(n+\alpha+\beta)}{(2n+\alpha+\beta-1)(2n+\alpha+\beta)^2(2n+\alpha+\beta+1)}.
 \end{align*}
Notice that for the ultraspherical (or Gegenbauer) orthogonal polynomials in \( [-1,1]\) (\( \alpha=\beta\)) we have that \( b_n=0 \), \( n\in\N_0 \), so that the corresponding Markov chain has period \( 2 \). This, in particular, includes the Chebyshev polynomials \( \alpha=\beta=-\frac{1}{2}\) corresponding to gambler ruin Markov chain.
It also admits a positive bidiagonal factorization of the form \eqref{factorizacionbidiagonalescalar} with
\begin{align*}
%\begin{aligned}
 a_{2n+1}(\alpha,\beta)&=\dfrac{(\alpha+n+1)(\alpha+\beta+n+1)}{(\alpha+\beta+2n+1)_2}, &
 a_{2n}(\alpha,\beta)&=\dfrac{n(\beta+n)}{(\alpha+\beta+2n)_2}.
%\end{aligned}
\end{align*}
% \begin{rem}
In the standard version, i.e., supported on \( [-1,1] \), the Gegenbauer Markov chains do not possess a stochastic pure death/pure birth factorization.
%\end{rem}
\end{rem}

\begin{exa}
 Following \eqref{estocasticaescalar} for the election of parameters \( \alpha=0.5 \), \( \beta=0.75\); we get the following \( 5\times5\) stochastic matrix and its corresponding pure birth/pure death factorization \eqref{factorizacionbidiagonalescalarestocastica}
 \begin{align*}
 P_5(0.5,0.75)\approx&
 \begin{bNiceArray}[small]{ccccc}
 0.50 & 0.50& 0 & 0 & 0 \\
 0.14 & 0.53 & 0.33& 0 & 0 \\
 0 & 0.22& 0.54 & 0.24& 0 \\
 0 & 0 & 0.30 & 0.54 & 0.16\\
 0 & 0 & 0 & 0.46& 0.54 \\
 \end{bNiceArray}
\approx
 \begin{bNiceArray}[small]{ccccc}
 1 & 0 & 0 & 0 & 0 \\
 0.28 & 0.72& 0 & 0 & 0 \\
 0 & 0.40 & 0.60 & 0 & 0 \\
 0 & 0 & 0.50 & 0.50& 0 \\
 0 & 0 & 0 & 0.68 & 0.32 \\
 \end{bNiceArray}
 \begin{bNiceArray}[small]{ccccc}
 0.50 & 0.50 & 0 & 0 & 0 \\
 0 & 0.55 & 0.45 & 0 & 0 \\
 0 & 0 & 0.60 & 0.40 & 0 \\
 0 & 0 & 0 & 0.68 & 0.32\\
 0 & 0 & 0 & 0 & 1 \\
 \end{bNiceArray}.
 \end{align*}

 It is also possible to check that its corresponding steady state \eqref{estacionarioescalar} is
 \begin{align*}
 \pi_5(0.5,0.75)\approx \begin{bNiceMatrix}
 0.06& 
 0.23 &
 0.34&
 0.27&
 0.10
 \end{bNiceMatrix}.
 \end{align*}
 The average number of steps to return are
 \begin{align*}
%\begin{aligned}
 (\bar t_5)_1&\approx 15.93, & (\bar t_5)_2&\approx 4.41, & (\bar t_5)_3&\approx 2.94, & (\bar t_5)_4&\approx 3.64,& (\bar t_5)_5&\approx 10.46.
%\end{aligned}
 \end{align*}
 \begin{center}
 \tikzset{decorate sep/.style 2 args={decorate,decoration={shape backgrounds,shape=circle,shape size=#1,shape sep=#2}}}
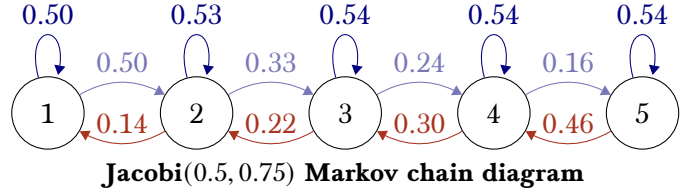
 \begin{tikzpicture}[start chain = going right,
 -Triangle, every loop/.append style = {-Triangle}]
 \foreach \i in {1,...,5}
 \node[state, on chain] (\i) {\i};
 \foreach
 \i/\txt in {1/\( 0.50 \), 2/\( 0.33\)/,3/\( 0.24 \), 4/\( 0.16\)}
 \draw let \n1 = { int(\i+1) } in
 (\i) edge[bend left,"\txt",color=Periwinkle] (\n1);
 \foreach
 \i/\txt in {1/\( 0.14 \), 2/\( 0.22 \), 3/\( 0.30 \), 4/\( 0.46\)}
 \draw let \n1 = { int(\i+1) } in
 (\n1) edge[bend left,above, "\txt",color=Mahogany,auto=right] (\i); 
 \foreach \i/\txt in {1/\( 0.50 \), 2/\( 0.53\)/,3/\( 0.54 \), 4/\( 0.54 \), 5/\( 0.54\)}
 \draw (\i) edge[loop above,color=NavyBlue, "\txt"] (\i);
 
 \end{tikzpicture}
 \begin{tikzpicture}
 \draw (4,-1.8) node
 {\begin{minipage}{0.8\textwidth}
 \begin{center}\small
 \textbf{Jacobi\( (0.5,0.75)\) Markov chain diagram}
 \end{center}
 \end{minipage}};
 \end{tikzpicture}
 \end{center}
\end{exa}

%%%%%%%%%%%%%%%%%%%%%%%%%%%%%%%%%%%%%
\subsection{Meixner finite Markov chains}

The monic Meixner orthogonal polynomials \cite[Chapter \(6.1\)]{Ismail} admit the hypergeometric series expression
\begin{align*}
 M_n(x;\beta,c)&=\left(\dfrac{c}{c-1}\right)^n(\beta)_n\,\pFq{2}{1}{-n,-x}{\beta}{\dfrac{c-1}{c}},
\end{align*}
for \( n\in\mathbb N_0 \), with \( \beta>0\) and \( 0<c<1\); and satisfy the discrete orthogonality relations
\begin{align*}
 \sum_{k=0}^{\infty}(k+\beta)_jM_n(k;\beta,c)w(k;\beta,c)&=0, && j\in\{0,\ldots,n-1\} ,
\end{align*}
%for \( j\in\{0,\ldots,n-1\} \), 
with respect to the given weight function
\( w %\textcolor{blue}{
 (x;\beta,c)
 %}
=\dfrac{\Gamma(\beta+x)}{\Gamma(\beta)}\dfrac{c^x}{\Gamma(x+1)} \). 
These polynomials satisfy a three term recurrence relation, see \eqref{recurrenciaescalar} and \eqref{eq:specralJ_dm}, with coefficients
\begin{align*}
%\begin{aligned}
 &b_n(\beta,c)=\dfrac{n+(\beta+n)c}{1-c};
 &c_n(\beta,c)=\dfrac{nc(\beta+n-1)}{(1-c)^2}.
%\end{aligned}
\end{align*}
Therefore, the chain is ergodic. It also admits a bidiagonal factorization of the form \eqref{factorizacionbidiagonalescalar} with
\begin{align*}
%\begin{aligned}
 a_{2n+1}(\beta,c)&=\dfrac{(\beta+n)c}{1-c}, &
 a_{2n}(\beta,c)&=\dfrac{n}{1-c}.
%\end{aligned}
\end{align*}
They are positive for \( \beta>0\) and \( 0<c<1 \). We see that all these Markov chains are ergodic.

\subsection{Kravchuk finite Markov chains}
The monic Kravchuk orthogonal polynomials are \cite[\S 6.2]{Ismail}
\begin{align*}
 K_n(x;p,N)&=p^n(-N)_n\,\pFq{2}{1}{-n,-x}{-N}{\dfrac{1}{p}},
\end{align*}
for \( n\in\{0,\cdots,N\} \), with \( 0<p<1\) and \( N\in\mathbb N_0\); and satisfy the orthogonality relations
\begin{align*}
 \sum_{k=0}^{N}(-N+k)_jK_n(k;p,N)w(k;p,N)&=0, & j&\in\{0,\ldots,n-1\} ,
\end{align*}
%for \( j\in\{0,\ldots,n-1\} \), 
with respect to the given weight function
\( w %\textcolor{blue}{
 (x;p,N) %}
=\dfrac{\Gamma(N+1)}{\Gamma(x+1)\Gamma(N-x+1)}p^x(1-p)^{N-x} \). 
The three term recurrence relation coefficients, see \eqref{recurrenciaescalar} and \eqref{eq:specralJ_dm}, are
\begin{align*}
%\begin{aligned}
 b_n(p,N)&=(N-n)p+n(1-p), &
 c_n(p,N)&=n(1-p)(N-n+1)p,
%\end{aligned}
\end{align*}
which are always nonnegative, 
and the coefficients of the stochastic bidiagonal factorization \eqref{factorizacionbidiagonalescalar} are
\begin{align*}
%\begin{aligned}
 a_{2n+1}(p,N)&=(N-n)p, &
 a_{2n}(p,N)&=n(1-p).
%\end{aligned}
\end{align*}
The Markov chain is ergodic.
\begin{rem}
Now, if we set \( p=1/2 \), we find that the diagonal coefficients of the Jacobi matrix become uniform, i.e., \( b_n=N/2 \). Consequently, we can consider shifted Kravchuk orthogonal polynomials with
 \( b_n=0 \), describing a period-\( 2\) recurrent reversible Markov chain. While these chains are neither ergodic nor positive, the steady state is still an equilibrium state. Moreover, in such cases, there is no stochastic factorization available.
\end{rem}

\subsection{Laguerre finite Markov chains}
The monic Laguerre orthogonal polynomials \cite[\S 4.6]{Ismail} have the following hypergeometric series representation 
\begin{align*}
 L_n(x;\alpha)&=(-1)^n(\alpha+1)_n\,\pFq{1}{1}{-n}{\alpha+1}{x}, 
\end{align*}
for \( n\in\mathbb N_0 \), with \( \alpha>-1 \), and satisfy the orthogonality relations
\begin{align*}
 \int_{0}^{\infty}x^jL_n(x;\alpha)w(x;\alpha)\d x&=0, & j&\in \{0,\ldots,n-1\} ,
\end{align*}
%for \( j\in \{0,\ldots,n-1\} \), 
with respect to the given weight function
\( w %\textcolor{blue}{
 (x;\alpha) %}
=x^\alpha \Exp{-x} \). 
The three term recurrence relation coefficients, see \eqref{recurrenciaescalar} and \eqref{eq:specralJ_dm}, are
\begin{align*}
%\begin{aligned}
 b_n(\alpha)&=2n+\alpha+1, & c_n(\alpha)&=n(n+\alpha),
%\end{aligned}
\end{align*}
which are nonnegative for \( \alpha>-1 \), 
and the coefficients of the stochastic bidiagonal factorization \eqref{factorizacionbidiagonalescalar} are
\begin{align*}
%\begin{aligned}
 a_{2n+1}(\alpha)&=\alpha+n+1, & a_{2n}(\alpha)&=n.
%\end{aligned}
\end{align*}
The Markov chain is ergodic.

\subsection{Charlier finite Markov chains}
The monic Charlier orthogonal polynomials are \cite[\S 6.1]{Ismail}
\begin{align*}
 C_n(x;b)=(-b)^n\pFq{2}{0}{-n,-x}{--}{-\dfrac{1}{b}} ,
\end{align*}
for \( n\in\N_0 \), with \( b > 0 \), and satisfy the orthogonality relations
\begin{align*}
 \sum_{k=0}^{\infty}k^jC_n(k;b)w(k;b)&=0, & j\in\{0,\ldots,n-1\} ,
\end{align*}
%for \( j\in\{0,\ldots,n-1\} \), 
with respect to the given weight function
\( w %\textcolor{blue}{
 (x;b)
 %}
= \dfrac{b^x}{\Gamma(x+1)} \). 
The three term recurrence relation coefficients, see \eqref{recurrenciaescalar} and \eqref{eq:specralJ_dm}, are
\begin{align*}
 %\begin{aligned}
 b_n(b)&=n+b, & c_n(b)&=nb,
 %\end{aligned}
\end{align*}
which are always nonnegative, and the coefficients of the stochastic bidiagonal factorization \eqref{factorizacionbidiagonalescalar} are
\begin{align*}
%\begin{aligned}
 a_{2n+1}(b)&=b, & a_{2n}(b)&=n.
%\end{aligned}
\end{align*}
The Markov chain is ergodic.

\subsection{Hermite finite Markov chains}
The monic Hermite orthogonal polynomials are \cite[\S 4.6]{Ismail}
\begin{align*}
 H_n(x)&=(-1)^n\sqrt{\pi}\sum_{l=0}^n\dfrac{(-n)_l}{l!}\dfrac{1}{\Gamma\left(\frac{-n+1+l}{2}\right)}x^l=x^n\pFq{2}{0}{-\frac{n}{2},-\frac{n-1}{2}}{--}{-\dfrac{1}{x^2}},
\end{align*}
for \( n\in\N_0 \), and satisfy the orthogonality relations
\begin{align*}
 \int_{-\infty}^{\infty}x^jH_n(x)w(x)\d x&=0, &j&\in\{0,\ldots,n-1\} ,
\end{align*}
%for \( j\in\{0,\ldots,n-1\} \), 
with respect to the given weight which is the Gaussian function
\( w (x) = \Exp{-x^2} \). 
The three term recurrence relation coefficients, see \eqref{recurrenciaescalar} and \eqref{eq:specralJ_dm}, are
\begin{align*}
%\begin{aligned}
 b_n&=0, &
 c_n&=\dfrac{n}{2},
% \end{aligned}
\end{align*}
The corresponding recurrent reversible Markov chain is periodic with period \( 2 \), so that the chain is not ergodic. However, the steady state is an equilibrium state and the convergence to it is geometric. Moreover, there is no pure birth/pure death stochastic factorization. 
 However, any shift \( J_m\mapsto J_m+s I_m\) with \( s> x_{m,m}\) will have such a factorization.

\section{Multiple orthogonal polynomials and finite Markov chains beyond birth and death}

We now turn our attention to multiple orthogonal polynomials. This extended form of orthogonality has found applications in various areas, including simultaneous approximation theory, number theory, random matrices, and Brownian motion. In our case, it also proves useful for constructing Markov chains that go beyond birth and death processes.

For further information on multiple orthogonal polynomials, we recommend consulting \cite{Ismail}, \cite{nikishin_sorokin}, \cite{afm}, and~\cite{bfm}.

\subsection{Multiple orthogonal polynomials}
Now, let's consider two weight functions \( w_1 ,w_2:\Delta\subseteq\mathbb R\rightarrow\R_{\geq 0} \), a measure \( \mu :\Delta\subseteq\mathbb R\rightarrow\R_{\geq 0} \), and a sequence of multi-indices 
\begin{align*}
 \lbrace{(n_1,n_2)}\rbrace_{(n_1,n_2)}\subseteq\mathbb N_0^2.
\end{align*}
 We are interested in finding a sequence of polynomials, referred to as type II multiple orthogonal polynomials, denoted by \( \lbrace{B_{(n_1,n_2)}}\rbrace \), where \( \deg B_{(n_1,n_2)}\leq n_1+n_2 \), that satisfy the orthogonality relations:
\begin{align*}
 \int_{\Delta}x^jB_{(n_1,n_2)}(x)w_i(x)\d\mu(x)&=0, 
	 & j&\in\{0,\ldots,n_i-1\}, & i&\in\{1,2\} . 
\end{align*}
%for \( j\in\{0,\ldots,n_i-1\}\) and \( i\in\{1,2\} \). 
Additionally, we have two sequences of polynomials, known as type I multiple orthogonal polynomials, denoted by \( \lbrace{A_{(n_1,n_2),i}}\rbrace_{i=1,2} \), where \( \deg A_{(n_1,n_2),i}\leq n_i-1 \), that satisfy the orthogonality relations:
\begin{align*}
 \int_{\Delta}x^j{\left(A_{(n_1,n_2),1}(x)w_1(x)+A_{(n_1,n_2),2}(x)w_2(x)\right)}\d\mu(x)&=0, 
 & j &\in \{0,\ldots,n_1+n_2-2\} .
\end{align*}
%with \( j \in \{0,\ldots,n_1+n_2-2\} \). 
Similarly, when we have a system of two weight functions \( w_1 ,w_2:\Delta\subseteq\mathbb Z\rightarrow\mathbb R^+ \), two discrete measures \( \mu _i=\sum_{k\in\Delta}w_i(k)\delta(x-k) \), and a sequence of multi-indices \( \lbrace{(n_1,n_2)}\rbrace_{(n_1,n_2)}\subseteq\mathbb N_0^2 \), it is possible to have the existence of the same sequences of polynomials satisfying discrete orthogonality relations of the form:
\begin{align*}
 \sum_{k\in\Delta}k^jB_{(n_1,n_2)}(k)w_i(k)&=0,
 & j&\in\{0,\ldots,n_i-1\} ,
 & i&\in\{1,2\} ,
\end{align*}
%for \( j\in\{0,\ldots,n_i-1\}\) and \( i\in\{1,2\} \), 
and 
\begin{align*}
 \sum_{k\in\Delta}k^j{\left(A_{(n_1,n_2),1}(k)w_1(k)+A_{(n_1,n_2),2}(k)w_2(k)\right)}&=0,
& j\in&\{0,\ldots,n_1+n_2-2\}.
\end{align*}
%for \( j\in\{0,\ldots,n_1+n_2-2\} \). 
A multi-index \( (n_1,n_2)\) is normal if \(B_{(n_1,n_2)}\) has degree \(n_1 +n_2\) or, equivalently, 
\( A_{(n_1,n_2),1}\) and \( A_{(n_1,n_2),2}\) can be normalized so that
\begin{align*}
\int_{\Delta} x^{n_1+n_2 -1} \left( A_{(n_1,n_2),1} w_1 (x) +A_{(n_1,n_2),2} w_2 (x) \right) \d \mu (x) = 1.
\end{align*}
%If the multi-index \( (n_1,n_2)\) corresponds to polynomials \( B_{(n_1,n_2)} \), \( A_{(n_1,n_2),1}\) and \( A_{(n_1,n_2),2}\) with maximum degree \( (n_1,n_2) \), it is referred to as a ``normal'' multi-index. 
A system is considered ``perfect'' if every multi-index is normal. In this study, we assume the use of AT systems, as described in \cite{Ismail} and \cite{nikishin_sorokin}, which guarantees perfectness.
In the case of finite sequences of polynomials we are interested in the normality til the maximum order of the multi-index, that is
we will focus on a specific set of multi-indices called the ``stepline,'' defined as follows:
\begin{align*}
 \lbrace{(n_1,n_2)}\rbrace_{(n_1,n_2)} = \lbrace{(0,0),(1,0),(1,1),(2,1),(2,2),\ldots}\}, \quad n_1+n_2 \leq N,
\end{align*}
here \( N\) can be taken as any ``finite'' natural number or as an infinite one.

For each \( n\in\mathbb{N}_0 \), we denote the corresponding polynomials as follows:
\begin{align*}
\begin{aligned}
 B^{(2n)} &\coloneq B_{(n,n)}, & B^{(2n+1)} &\coloneq B_{(n+1,n)}, \\
 A^{(2n)}_i &\coloneq A_{(n+1,n),i}, & A_i^{(2n-1)} &\coloneq A_{(n,n),i}, & i&\in\{1,2\}.
\end{aligned}
\end{align*}

\subsection{Recursion and stochastic matrices}

In the case of both continuous and discrete type II polynomials, they satisfy a four-term recurrence relation of the form:
\begin{align}
 \label{recurrenciaII}
 xB^{(n)}(x) &= B^{(n+1)}(x) + b_{n}B^{(n)}(x) + c_{n}B^{(n-1)}(x) + d_{n}B^{(n-2)}(x), 
\end{align}
for \( n\in\{0,\ldots, N-1\} \), 
where \( b_n, c_n \), and \( d_n\) are real coefficients with \( d_n\neq 0 \). We set \( B^{(-1)}=B^{(-2)}=0\) for convenience, so we do not require \( c_0, d_0 \), and \( d_1 \), which are also taken as \( 0 \). This recurrence relation can be represented in matrix form as:
\begin{align*}
\begin{aligned}
 T
 \begin{bmatrix}
 B^{(0)}(x)\\ B^{(1)}(x)\\ B^{(2)}(x)\\ \Vdots[shorten-end=0pt] 
 \end{bmatrix}
 &= x
 \begin{bmatrix}
 B^{(0)}(x)\\ B^{(1)}(x)\\ B^{(2)}(x)\\ \Vdots[shorten-end=0pt] 
 \end{bmatrix},
 &
 T& \coloneqq
 \begin{bmatrix}
 b_0 & 1 & 0 &\Cdots[shorten-end=5pt] & \\
 c_1 & b_1 & 1& \Ddots[shorten-end=5pt] & \\
 d_2 & c_2 & b_2 & \Ddots[shorten-end=5pt] & \\[2pt]
 0 & d_3 & c_3 & b_3 & \\
 \Vdots[shorten-start=7pt,shorten-end=7pt] & \Ddots[shorten-end=-1pt] & \Ddots[shorten-end=-1pt] & \Ddots[shorten-end=-2pt] & \Ddots[shorten-end=10pt] 
 \end{bmatrix},
\end{aligned}
\end{align*}
for an infinite sequence. For a finite sequence or a truncation, we have:
\begin{align}
 \label{recurrenciaIImatricial} T_m
 \begin{bmatrix}
 B^{(0)}(x)\\ B^{(1)}(x)\\ \Vdots\\ B^{(m-1)}(x)
 \end{bmatrix}
 &= x
 \begin{bmatrix}
 B^{(0)}(x)\\ B^{(1)}(x)\\ \Vdots\\ B^{(m-1)}(x)
 \end{bmatrix}
 - \begin{bmatrix}
 0\\ \Vdots\\ 0\\ %0\\ 
 B^{(m)}(x)
 \end{bmatrix},
\end{align}
with
\begin{align}
 T_m &\coloneqq
 \begin{bmatrix}[columns-width=.7cm]
 b_0 & 1 & 0 & \Cdots & & & 0\\
 c_1 & b_1 & 1 & \Ddots[shorten-end=1pt] & & & \Vdots\\[5pt]
 d_2 & c_2 & b_2 & \Cdots & & & \\
 0 & d_3 & c_3 &\Ddots & \Ddots[shorten-end=1pt]& & \\
 \Vdots & \Ddots[shorten-end=5pt] &\Ddots[shorten-end=0pt] & \Ddots[shorten-end=1.5pt] & & & 0\\[5pt]
 & & & & & & 1\\
 0 & \Cdots & & 0& d_{m-1} & c_{m-1} & b_{m-1}
 \end{bmatrix},
\end{align}
where \( m\in\{1,\ldots,N\} \). 

We will require that these recurrence Hessenberg matrix \( T \) or \( T_m\) to be nonnegative matrices. Hence, all \( b_n,c_n\geq 0\) and \( d_n>0 \). 
These matrices are irreducible since all entries \( d_n\) are positive. 

Let \( \lbrace{x_{n,i}}\rbrace_{i=1}^n\) represent the set of increasing zeros of the \( n\)-th type II polynomial \( B^{(n)} \). Since our system is assumed to be AT, we know that all zeros are simple and located in the interior of \( \Delta \). Moreover, the zeros of \( B^{(n-1)}\) interlace with the zeros of \( B^{(n)} \), i.e.:
\begin{align*}
%\begin{aligned}
 x_{n,i-1} &< x_{n-1,i-1} < x_{n,i}, & n &\in \{2,\ldots,N\}, & i &\in \{2,\ldots,n\}.
%\end{aligned}
\end{align*}

From Equation \eqref{recurrenciaIImatricial}, we deduce that the eigenvalues of \( T_m\) are the zeros of \( B^{(m)} \), denoted as \( \lbrace{x_{m,i}}\rbrace_{i=1}^m \). The corresponding right eigenvectors are given by:
\begin{align}
 \label{eigenvectorrightmultiple}
& \begin{bmatrix}
 B^{(0)}(x_{m,i})\\
 \Vdots\\
 B^{(m-1)}(x_{m,i})
 \end{bmatrix}, 
\end{align}
for \( i \in \{1,\ldots,m\} \). 
Now, we can proceed as we did in the previous section. Let us define the diagonal matrix:
\begin{align*}
%\begin{aligned}
 \sigma_{II,m} &\coloneq \diag
 \left[
\begin{NiceMatrix}
 b_m^{(0)} & \cdots & b^{(m-1)}_m
\end{NiceMatrix}\right]
 , &
 b^{(n)}_m&
\coloneq B^{(n)}(x_{m,m}) .
%\end{aligned}
\end{align*}

\begin{teo}
 \label{propestocasticaII}
 Consider the recurrence relation expressed in Equation \eqref{recurrenciaIImatricial} with \( T_m\) as a nonnegative matrix and \( \lbrace{x_{m,i}}\rbrace_{i=1}^m\) as the set of increasing zeros of \( B^{(m)} \), where \( x_{m,m}>0 \). Then, the following matrix is stochastic
 \begin{align}
 \label{estocasticaII}
 P_{II,m} &\coloneq \dfrac{1}{x_{m,m}} \sigma_{II,m}^{-1} T_m \sigma_{II,m}.
 \end{align}
\end{teo}

\begin{proof}
 We need to verify that \( P_{II,m}\) is nonnegative and that the sum of each row equals \( 1 \), satisfying conditions~\eqref{nonegatividad} and \eqref{estocasticidad}. Firstly, it is straightforward to check that \( P_{II,m}e=e\) from the definition of \( P_{II,m} \). Therefore, \( P_{II,m}\) satisfies \eqref{estocasticidad}. Secondly, since \( T_m\) is nonnegative by hypothesis and \( x_{m,m}>0 \), the only condition to ensure the nonnegativity of \( P_{II,m}\) is that \( B^{(0)}(x_{m,m}),\ldots,B^{(m-1)}(x_{m,m})\) all have the same sign. Notice that these are precisely the components of the right eigenvector of \( T_m\) associated with the eigenvalue \( x_{m,m} \), which is also the spectral radius of \( T_m\) (see Definition \ref{defradioespectral}). Since \( T_m\) is nonnegative and irreducible, the Perron--Frobenius Theorem applies, and we conclude that 
\( B^{(0)}(x_{m,m}),\ldots,B^{(m-1)}(x_{m,m})\)
 have the same sign, as stated in (iii) of Perron--Frobenius Theorem \ref{PerronFrobenius}. Thus, \( P_{II,m}\) satisfies \eqref{nonegatividad}
 .
\end{proof}

The stochastic matrix \( P_{II,m}\) explicitly reads:
 \begin{align*}
P_{II,m}= \frac{1}{x_{m,m}} \begin{bNiceMatrix}
 b_0 & \frac{b^{(1)}_m}{b_m^{(0)}} & 0 & \Cdots & & && 0\\
 \frac{b^{(0)}_m}{b^{(1)}_m}c_1 & b_1 & \frac{b^{(2)}_m}{b^{(1)}_m} & \Ddots & & & &\Vdots\\
 \frac{b^{(0)}_m}{b^{(2)}_m}d_2 & \frac{b^{(1)}_m}{b^{(2)}_m}c_2 & b_2 & \Ddots& &&& \\
 0 & \Ddots[shorten-end=-25pt]& \Ddots& \Ddots& && & \Vdots\\
 \Vdots &\Ddots[shorten-end=-30pt]&&&&&&\\
 & & & & & & & 0\\
 & & & && & & \frac{b^{(m-1)}_m}{b^{(m-2)}_m}\\
 0 & \Cdots[shorten-end=-30pt]&& &\hspace*{30pt}0 & \frac{b^{(m-3)}_m}{b^{(m-1)}_m}d_{m-1} & \frac{b^{(m-2)}_m}{b^{(m-1)}_m}c_{m-1} &
 b_{m-1}
 \end{bNiceMatrix}
.
\end{align*}
\begin{rem}
Notice that the Perron--Frobenius theorem is (in fact) not needed 
%fundamental 
for the proof (as the zeros interlace and we take the maximal one).
 If we had chosen any other zero of \( B^{(m)}\) to construct \( P_{II,m} \), non-negativity would not be achievable due to (iv) of the Perron--Frobenius Theorem \ref{PerronFrobenius}.
\end{rem}

\begin{rem}
 We refer to the stochastic matrix \( P_{II,m}\) as of type II since it is constructed using multiple orthogonal polynomials of type II. These type II stochastic matrices describe an \(m\)-state Markov chain where there is a possibility of going up to two steps backwards, but only one step forward at each time.
\end{rem}

Now, we replace the type II multiple orthogonal polynomials with type I multiple orthogonal polynomials. As a result, we obtain what we call type I stochastic matrices. 

\begin{rem}
It is important to notice that although the recurrence matrices are, by definitions, transpose of each other, the corresponding stochastic ones do not share this property.
\end{rem}

In this case, we start with the recurrence relation satisfied by the type I polynomials:
\begin{align}
 \label{recurrenciaI}
 xA_i^{(n)}(x) &= A_i^{(n-1)}(x) + b_{n}A_i^{(n)}(x) + c_{n+1}A_i^{(n+1)}(x) + d_{n+2}A_i^{(n+2)}(x), 
\end{align}
for \( i \in\{1,2\}\) and \( n\in\{0,\ldots,N-3\} \). 
Here, \( A_1^{(-1)}=A_2^{(-1)}=0 \). It can be observed that this recurrence relation is exactly the transposed version of the type II polynomials in Equation \eqref{recurrenciaII}. In matrix notation, it can be written as:
\begin{align*}
 T^\top
 \begin{bNiceMatrix}
 A_i^{(0)}(x)\\[4pt] A_i^{(1)}(x)\\[4pt] A_i^{(2)}(x) \\ \Vdots[shorten-end=2pt]
 \end{bNiceMatrix}
 &=x \begin{bNiceMatrix}
 A_i^{(0)}(x)\\[4pt] A_i^{(1)}(x)\\[4pt] A_i^{(2)}(x) \\ \Vdots[shorten-end=2pt]
 \end{bNiceMatrix},
\end{align*}
for \( i \in\{1,2\} \), if the sequence is infinite. Here, \( T^\top\) denotes the transposed matrix of \( T \). If the sequence is finite or truncated, then for 
\( m\in\{2,\ldots,N-2\}\):
\begin{align*}
{T_{m}^\top}
\begin{bNiceMatrix}
 A_i^{(0)}(x)\\[4pt] A_i^{(1)}(x)\\\Vdots\\[4pt] A_i^{(m-1)}(x)
 \end{bNiceMatrix}
 & =x\begin{bNiceMatrix}
 A_i^{(0)}(x)\\[4pt] A_i^{(1)}(x)\\\Vdots\\[4pt] A_i^{(m-1)}(x)
 \end{bNiceMatrix}
 -\begin{bNiceMatrix}
 0\\ \Vdots\\ 0\\ d_{m}A_i^{(m)}(x)\\[4pt] c_{m}A_i^{(m)}(x)+d_{m+1}A_i^{(m+1)}(x)
 \end{bNiceMatrix},
\end{align*}
for \( i\in\{1,2\} \). In the second term on the right-hand side, there are two possibly nonzero components. We would like to follow a similar procedure as we did with type II multiple orthogonal polynomials. For that purpose, we use the following polynomials constructed in terms of determinants of type I multiple orthogonal polynomials:
\begin{align*}
 \mathcal A^{(n)}_k(x) \coloneq (-1)^k
 \begin{vmatrix}
 A_1^{(n)}(x) & A_2^{(n)}(x)\\[3pt]
 A_1^{(k)}(x) & A_2^{(k)}(x)
 \end{vmatrix},
\end{align*}
as already discussed in \cite{gqmop,stbbm}. Then, the recurrence relations imply:
\begin{align}
\label{recurrenciaImatricial}
 {T_{m}^\top}
\begin{bNiceMatrix}
 \mathcal A_{m}^{(0)}(x)\\[6pt] 
 \mathcal A^{(1)}_{m}(x)\\[6pt] 
 \Vdots\\
 \mathcal A_{m}^{(m-1)}(x)
 \end{bNiceMatrix}
 & =x\begin{bNiceMatrix}
 \mathcal A_{m}^{(0)}(x)\\[6pt] 
 \mathcal A^{(1)}_{m}(x)\\[6pt] 
 \Vdots\\
 \mathcal A_{m}^{(m-1)}(x)
 \end{bNiceMatrix}
 -d_{m+1}
\begin{bNiceMatrix}
 0\\ \Vdots\\ 0\\ \mathcal A_{m}^{(m+1)}(x)
 \end{bNiceMatrix}, 
\end{align}
for \( m\in\{2,\ldots,N-3\} \). Here, we can see that the eigenvalues of \( T^\top_m\) are the zeros of \(\mathcal A^{(m+1)}_{m} \). However, the eigenvalues of a matrix and its transpose coincide, and earlier, we found that the eigenvalues of \( T_m\) are the zeros of \( B^{(m)} \). This is only possible if \(\mathcal A^{(m+1)}_{m}\) and \( B^{(m)}\) are the same up to a multiplicative constant. This was shown to hold in \cite[Proposition 3.2]{stbbm}, see also \cite{gqmop,stbbm0}.

Furthermore, we can observe that the right eigenvector of \( T_m^\top\) corresponding to eigenvalue \( x_{m,i}\) can be taken as:
\begin{align*}
%\label{eigenvectorleftmultiple}
&\begin{bNiceMatrix}
 \mathcal A_{m}^{(0)}(x_{m,i})\\\Vdots\\\mathcal A_{m}^{(m-1)}(x_{m,i})
\end{bNiceMatrix},
&
i &\in \{1,\ldots,m\} .
\end{align*}
%for \( i \in\{1,\ldots,m\} \). 
Similarly to what we did for \( T_m \), we can transform \( T_m^\top\) into a stochastic matrix, which we refer to as a type~I stochastic matrix. First, we define:
\begin{align*}
%\begin{aligned}
 \sigma_{I,m}& \coloneq \diag
 \left[
 \begin{NiceMatrix}
 a_m^{(0)} & \cdots & a^{(m-1)}_m
 \end{NiceMatrix}
 \right], &
 a^{(n)}_m
& \coloneq \mathcal A^{(n)}_m(x_{m,m}) .
%\end{aligned}
\end{align*}
Next, we have the following theorem:
\begin{teo}
 \label{propestocasticaI}
 Consider the recurrence relation given by Equation \eqref{recurrenciaImatricial}, where \( T_m\) is a nonnegative matrix and \( \lbrace{x_{m,i}}\rbrace_{i=1}^m\) are the zeros of 
 %\textcolor{blue}{
\(\mathcal A^{(m+1)}_{m}\) 
 %}
with \( x_{m,m}>0 \). Then, the following matrix:
\begin{align}
 \label{estocasticaI}
 P_{I,m}&\coloneq\dfrac{1}{x_{m,m}}\sigma_{I,m}^{-1} T_{m}^{\top} \sigma_{I,m} ,
\end{align}
 is a stochastic matrix.
\end{teo}

\begin{proof}
 We need to verify that \( P_{I,m}\) is nonnegative and that the sum of each of its rows equals \( 1\). In other words, we need to ensure that conditions \eqref{nonegatividad} and \eqref{estocasticidad} are satisfied. Firstly, it can be directly checked that \( P_{I,m}e=e\) using the definition of \( P_{I,m} \). Thus, \( P_{I,m}\) satisfies \eqref{estocasticidad}. Secondly, since \( T_m^\top\) is nonnegative by assumption and \( x_{m,m}>0 \), the only condition required for \( P_{I,m}\) to be nonnegative is that \(\mathcal A^{(0)}_{m}(x_{m,m}),\ldots \), 
 \(\mathcal A^{(m-1)}_{m}(x_{m,m})\)
have the same sign. Notice that these components correspond exactly to the entries of the left eigenvector of \( T_m^\top\) associated with the eigenvalue \( x_{m,m} \). Moreover, \( x_{m,m}\) is the spectral radius of \( T_m^\top\) (see Definition \ref{defradioespectral}). Since \( T_m^\top\) is nonnegative and irreducible, the Perron--Frobenius Theorem applies, and we can conclude that \(\mathcal A^{(0)}_{m}(x_{m,m}),\ldots, \mathcal A^{(m-1)}_{m}(x_{m,m})\) have the same sign by (iii) of the Perron--Frobenius Theorem \ref{PerronFrobenius}. Therefore, \( P_{I,m}\) satisfies \eqref{nonegatividad}.
\end{proof}
The stochastic matrix \( P_{I,m}\) explicitly reads
\begin{align*}
P_{I,m} = \frac{1}{x_{m,m}}
 \begin{bNiceMatrix}
 b_0 & \frac{a^{(1)}_{m}}{a^{(0)}_{m}}c_1 & \frac{a^{(2)}_{m}}{a^{(0)}_{m}}d_2 & 0&\cdots & & 0 \\
 \frac{a^{(0)}_{m}}{a^{(1)}_{m}} & b_1 & \frac{a^{(2)}_{m}}{a^{(1)}_{m}}c_2 & \Ddots& \Ddots& & \Vdots\\
 0 & \frac{a^{(1)}_{m}}{a^{(2)}_{m}} & b_2 & & & & \\
 \vdots & \Ddots& \Ddots& \ddots& \Ddots& & 0\\
 & & & & & & \frac{a^{(m-1)}_{m}}{a^{(m-3)}_{m}}d_{m-1}\\[6pt]
 & & & & & & \frac{a^{(m-1)}_{m}}{a^{(m-2)}_{m}}c_{m-1}\\
 0 & \Cdots& & & 0 & \frac{a^{(m-2)}_{m}}{a^{(m-1)}_{m}} & b_{m-1}
 \end{bNiceMatrix}
.
\end{align*}

\begin{rem}
 Note that if we had chosen any other zero of
%\textcolor{blue}{
\(\mathcal A^{(m+1)}_{m}\) 
%}
to construct \( P_{I,m} \), non-negativity would not be achievable due to (iv) of the Perron--Frobenius Theorem \ref{PerronFrobenius}.
\end{rem}

\begin{rem}
 We refer to the stochastic matrix \( P_{I,m}\) as type I because it is constructed in terms of type I multiple orthogonal polynomials. These type I stochastic matrices describe an \( m\)-state Markov chain where there is a possibility of moving up to two steps forward and one step backward at each time.
\end{rem}

As we did in Proposition \ref{pro:uni_class} we have:
\begin{pro}
 Markov chains with stochastic matrices \( P_{II,m}\) and \( P_{I,m}\) have only one class. 
\end{pro}

\subsection{Spectral properties}
Matrices \(\mathcal U\) and \(\mathcal V\) are constructed using the right and left eigenvectors of \( T_m\) as follows:
\begin{align*}
 \mathcal U &\coloneq \begin{bNiceMatrix}
 B^{(0)}(x_{m,1}) & \cdots & B^{(0)}(x_{m,m})\\
 \vdots & & \vdots\\
 B^{(m-1)}(x_{m,1}) & \cdots & B^{(m-1)}(x_{m,m})
 \end{bNiceMatrix}, \\ \mathcal V &\coloneq \begin{bNiceMatrix}
 \frac{\mathcal A_{m}^{(0)}(x_{m,1})}{\sum^m_{l=1}B^{(l-1)}(x_{m,1})\mathcal A_{m}^{(l-1)}(x_{m,1})} & \cdots & \frac{\mathcal A_{m}^{(m-1)}(x_{m,1})}{\sum^m_{l=1}B^{(l-1)}(x_{m,1})\mathcal A_{m}^{(l-1)}(x_{m,1})}\\
 \vdots & & \vdots\\
 \frac{\mathcal A_{m}^{(0)}(x_{m,m})}{\sum^m_{l=1}B^{(l-1)}(x_{m,m})\mathcal A_{m}^{(l-1)}(x_{m,m})} & \cdots & \frac{\mathcal A_{m}^{(m-1)}(x_{m,m})}{\sum^m_{l=1}B^{(l-1)}(x_{m,m})\mathcal A_{m}^{(l-1)}(x_{m,m})}
 \end{bNiceMatrix}.
\end{align*}
Since the left and right eigenvectors are orthogonal, and these vectors are normalized to be biorthogonal, we have:
\begin{align*}
 \mathcal U \mathcal V = \mathcal V \mathcal U = I_m,
\end{align*}
where 
%\textcolor{blue}{
\( I _m\)
%}
is the \( m\times m\) identity matrix. Moreover, we have the spectral representation for the tetradiagonal Hessenberg matrix:
\begin{align*}
 T_m = \mathcal U \, \begin{bNiceMatrix}[columns-width=auto]
 x_{m,1} & 0 & \Cdots & 0\\
 0 & x_{m,2} & \Ddots & \Vdots\\
 \Vdots & \Ddots[shorten-end=5pt] & \Ddots[shorten-end=-5pt] & 0\\[4pt]
 0 & \Cdots & 0 & x_{m,m}
 \end{bNiceMatrix} \mathcal V,
\end{align*}
allowing us to easily compute any power of the recursion matrix:
\begin{align*}
 T_m^r = \mathcal U \, \begin{bNiceMatrix}[columns-width=auto]
 x_{m,1}^r & 0 & \Cdots & 0\\
 0 & x_{m,2}^r & \Ddots & \Vdots\\
 \Vdots & \Ddots[shorten-end=5pt] & \Ddots[shorten-end=-3pt] & 0\\[4pt]
 0 & \Cdots & 0 & x_{m,m}^r
 \end{bNiceMatrix} \mathcal V,
\end{align*}
where \( r\) is a positive integer.

\begin{pro} The following spectral representations for the iterated probabilities 
\begin{align}
 (P_{II,m}^r)_{i,j} &= \frac{1}{x^r_{m,m}} \frac{b_m^{(j-1)}}{b_m^{(i-1)}} \sum^{m}_{k=1} x^r_{m,k} \frac{B^{(i-1)}(x_{m,k})\mathcal A_{m}^{(j-1)}(x_{m,k})}{\sum^{m}_{l=1} B^{(l-1)}(x_{m,k})\mathcal A_{m}^{(l-1)}(x_{m,k})}, \label{formularepresentacionII} \\
 (P_{I,m}^r)_{i,j} &= \frac{1}{x^r_{m,m}} \frac{a_{m}^{(j-1)}}{a_{m}^{(i-1)}} \sum^{m}_{k=1} x^r_{m,k} \frac{B^{(j-1)}(x_{m,k})\mathcal A_{m}^{(i-1)}(x_{m,k})}{\sum^{m}_{l=1} B^{(l-1)}(x_{m,k})\mathcal A_{m}^{(l-1)}(x_{m,k})} ,\label{formularepresentacionI}
\end{align}
 are satisfied.
\end{pro}
\begin{proof}
 The spectral representation of \( T_m\) allows us to express every entry of every power of the Hessenberg matrix \( T_m\) as follows:
 \begin{align}
 (T_m^r)_{i,j} = \sum^{m}_{k=1} x^r_{m,k} \frac{B^{(i-1)}(x_{m,k})\mathcal A_{m}^{(j-1)}(x_{m,k})}{\sum^{m}_{l=1} B^{(l-1)}(x_{m,k})\mathcal A_{m}^{(l-1)}(x_{m,k})},
 \end{align}
 where \( r\) is a nonnegative integer. Using this expression, along with \eqref{estocasticaII} and \eqref{estocasticaI}, we can derive the spectral formulas \eqref{formularepresentacionII} and \eqref{formularepresentacionI} for the corresponding probabilities after \( r \) transitions.
\end{proof}

We can use these formulas to show that both Markov chains described by the stochastic matrices \( P_{II,m}\) and \( P_{I,m} \), defined in Equations \eqref{estocasticaII} and \eqref{estocasticaI}, respectively, are recurrent.
\begin{pro}
 Both Markov chains described by the stochastic matrices \( P_{II,m}\) and \( P_{I,m}\) are recurrent.
\end{pro}
\begin{proof}
 We need to show that every state \( i \in \{1, \ldots, m\}\) is recurrent for both chains. According to Remark~\ref{condicionsuficienteparalarecurrencia} and Proposition \ref{prorecurrentesescalar}, we only need to check that the limits
 \begin{align*}
 \lim_{s\rightarrow1^{-}} P_{II,m,ii}(s) & = \lim_{s\rightarrow1^{-}} \sum_{r=0}^{\infty} (P^r_{II,m})_{i,i}s^r = \sum_{r=0}^{\infty} (P^r_{II,m})_{i,i},
 \\
% \end{align*}
% and
% \begin{align*}
 \lim_{s\rightarrow1^{-}} P_{I,m,ii}(s) & = \lim_{s\rightarrow1^{-}} \sum_{r=0}^{\infty} (P^r_{I,m})_{i,i}s^r = \sum_{r=0}^{\infty} (P^r_{I,m})_{i,i},
 \end{align*}
 diverge. By substituting \( (P_{II,m}^r)_{i,i}\) and \( (P_{I,m}^r)_{i,i}\) with the expressions from \eqref{formularepresentacionII} and \eqref{formularepresentacionI}, respectively, we obtain:
 \begin{align*}
 \sum_{r=0}^{\infty} (P_{II,m}^r)_{i,i} &= \sum^{m}_{k=1} \frac{B^{(i-1)}(x_{m,k})\mathcal A_{m}^{(i-1)}(x_{m,k})}{\sum^{m}_{l=1} B^{(l-1)}(x_{m,k})\mathcal A_{m}^{(l-1)}(x_{m,k})} \sum_{r=0}^{\infty} \left(\frac{x_{m,k}}{x_{m,m}}\right)^r, \\
 \sum_{r=0}^{\infty} (P_{I,m}^r)_{i,i} &= \sum^{m}_{k=1} \frac{B^{(i-1)}(x_{m,k})\mathcal A_{m}^{(i-1)}(x_{m,k})}{\sum^{m}_{l=1} B^{(l-1)}(x_{m,k})\mathcal A_{m}^{(l-1)}(x_{m,k})} \sum_{r=0}^{\infty} \left(\frac{x_{m,k}}{x_{m,m}}\right)^r.
 \end{align*}
It is clear that both sums diverge when \( k=m \).
Therefore, every state \( i \) is recurrent for both Markov chains.
\end{proof}
\begin{lemma}
 If there exists \( i \in\N_0\) such that \( b_{i}>0 \), the Markov chains associated with \( P_{II,m}\) and \( P_{I,m}\) are aperiodic with period \( 1 \). If \( b_{i}=0\) for all \( i \in\N_0\) and there exists \( j \in \N_0\) with \( c_j\neq 0 \), then they have period \( 2 \). Finally, if \( b_n=c_n=0\) for all \( n\in\N_0 \), then the Markov chain has period \( 3 \). 
\end{lemma}
\begin{proof}
 If \( b_i>0 \), by definition, the state \( i \) has period \( 1 \). As all states are in the same class, the statement follows. If \( b_{i}=0\) for all states \( i \in\N_0\) and there is \( j \in \N_0\) with \( c_j\neq 0 \), all odd powers of both stochastic matrices have zero entries on the main diagonal, while the even powers have nonzero entries, resulting in a period of \( 2 \). Finally, if \( b_i=c_i=0\) for all states \( i \in\N_0 \), it can be checked that the \( 3 n+1\) and \( 3n+2\) powers of both transition matrices have zero entries on the diagonal, while the \( 3 n\) powers have nonzero entries on the main diagonal.
\end{proof}

\begin{coro}
 Let us assume that there exists a state \( i \in\{1,\ldots,m\}\) such that \( b_{i}>0 \). Then, both Markov chains described by the stochastic matrices \( P_{II,m}\) and \( P_{I,m} \), defined in \eqref{estocasticaII} and \eqref{estocasticaI}, respectively, are ergodic.
\end{coro}

\begin{coro}
 The Markov chains described by the stochastic matrices \( P_{II,m}\) and \( P_{I,m} \), defined in Equations \eqref{estocasticaII} and \eqref{estocasticaI}, are ergodic if and only if there exists a state \( i \) such that \( b_{i}>0 \). 
\end{coro}

\begin{proof}
 As we have seen, ergodicity occurs if and only if the chain is recurrent and aperiodic. Therefore, as recurrence is ensured, we only need to check aperiodicity.
\end{proof}
\begin{pro}
 \label{propestacionariomultiple}
 The probability vector
 \begin{align}
 \label{estacionariomultiple}
 \pi_m\coloneq
 \frac{1}{\sum_{l=1}^m b_m^{(l-1)}a^{(l-1)}_{m}} \begin{bNiceMatrix}
 b_m^{(0)}a^{(0)}_{m}&\cdots&b_m^{(m-1)}a^{(m-1)}_{m}
 \end{bNiceMatrix} ,
 \end{align}
 is the unique steady state for both the type II and type I stochastic matrices, as shown in \eqref{estocasticaII} and \eqref{estocasticaI}, respectively. 
 This means that
 \begin{align*}
%\begin{aligned}
 \pi_mP_{II,m}&=\pi_m, & \pi_mP_{I,m}&=\pi_m.
%\end{aligned}
 \end{align*}
\end{pro}
\begin{proof}
 The proof is straightforward and completely analogous to that of Proposition \ref{proestacionarioescalar} using equations \eqref{estocasticaII} and \eqref{estocasticaI}.
 We have
 \begin{align*}
 \pi_m P_{II,m}&= \frac{1}{\sum_{l=1}^m b_m^{(l-1)}a^{(l-1)}_{m}} \begin{bNiceMatrix}
 b_m^{(0)}a^{(0)}_{m}&\cdots&b_m^{(m-1)}a^{(m-1)}_{m}
 \end{bNiceMatrix} \dfrac{1}{x_{m,m}} \sigma_{II,m}^{-1} T_m \sigma_{II,m}\\
 &=\frac{1}{\sum_{l=1}^m b_m^{(l-1)}a^{(l-1)}_{m}} \begin{bNiceMatrix}
 a^{(0)}_{m}&\cdots&a^{(m-1)}_{m}
 \end{bNiceMatrix} \dfrac{1}{x_{m,m}} T_m \sigma_{II,m}\\&=
 \frac{1}{\sum_{l=1}^m b_m^{(l-1)}a^{(l-1)}_{m}} \begin{bNiceMatrix}
 a^{(0)}_{m}&\cdots&a^{(m-1)}_{m}
 \end{bNiceMatrix} \sigma_{II,m}\\&=\pi_m,\\
 \pi_m P_{I,m}&= \frac{1}{\sum_{l=1}^m b_m^{(l-1)}a^{(l-1)}_{m}} \begin{bNiceMatrix}
 b_m^{(0)}a^{(0)}_{m}&\cdots&b_m^{(m-1)}a^{(m-1)}_{m}
 \end{bNiceMatrix}\dfrac{1}{x_{m,m}} \sigma_{I,m}^{-1} T^\top_m \sigma_{I,m}
 \\
 &= \frac{1}{\sum_{l=1}^m b_m^{(l-1)}a^{(l-1)}_{m}} \begin{bNiceMatrix}
 b_m^{(0)}&\cdots&b_m^{(m-1)}
 \end{bNiceMatrix}\dfrac{1}{x_{m,m}} T^\top_m \sigma_{I,m}\\&=
 \frac{1}{\sum_{l=1}^m b_m^{(l-1)}a^{(l-1)}_{m}} \begin{bNiceMatrix}
 b_m^{(0)}&\cdots&b_m^{(m-1)}
 \end{bNiceMatrix}\sigma_{I,m}\\&=\pi_m.
 \end{align*}
Recall that for recurrent chains, the steady state is unique.
\end{proof}

\begin{coro}
 For both Markov chains described by the stochastic matrices \( P_{II,m}\) and \( P_{I,m} \), defined at \eqref{estocasticaII} and \eqref{estocasticaI}, respectively, the expected return time to state \( j\) is given by
\begin{align*}
 (\bar t_m)_j &= \frac{ \sum_{l=1}^m b_m^{(l-1)}a^{(l-1)}_{m}}{ b_m^{(j-1)}a^{(j-1)}_{m}}, 
 & j&\in\{1,\ldots,m\} . 
 \end{align*}
% for \( j\in\{1,\ldots,m\} \). 
\end{coro}
\begin{coro}
 The steady state \( \pi_m\) is recovered as the limit of the iterated probabilities
 \begin{align*}
 (\pi_m)_j&=\lim_{r\to\infty} (P_{II,m}^r)_{i,j}=\lim_{r\to\infty} 
 %\textcolor{blue}{
(P_{I,m}^r)_{i,j} %}
 ,
 & i ,j&\in\{1,\ldots,m\} .
 \end{align*}
 Moreover, the convergence ratio, which is geometric, is given by
 \begin{align*}
 (P_{II,m}^r)_{i,j}-(\pi_m)_j
 =
 \dfrac{b_m^{(j-1)}}{
 %\textcolor{blue}{
 b_m^{(i-1)}
 %}
 }
 \dfrac{B^{(i-1)}(x_{m,m-1})\mathcal A_{m}^{(j-1)}(x_{m,m-1})}{\sum^{m}_{l=1}B^{(l-1)}(x_{m,m-1})\mathcal A_{m}^{(l-1)}({x_{m,m-1}})}\left(\dfrac{x_{m,m-1}}{x_{m,m}}\right)^r
 + \cdots
 &\xrightarrow[r\to\infty]{} 0, \\
 (P_{I,m}^r)_{i,j}-(\pi_m)_j
 =
 \dfrac{a_{m}^{(j-1)}}{a_{m}^{(i-1)}}\dfrac{B^{(j-1)}(x_{m,m-1})\mathcal A_{m}^{(i-1)}(x_{m,m-1})}{\sum^{m}_{l=1}B^{(l-1)}(x_{m,m-1})\mathcal A_{m}^{(l-1)}({x_{m,m-1}})}\left(\dfrac{x_{m,m-1}}{x_{m,m}}\right)^r + \cdots
 &\xrightarrow[r\to\infty]{} 0.
 \end{align*}
\end{coro}
\begin{proof}
 It follows from \eqref{formularepresentacionII} and \eqref{formularepresentacionI}.
\end{proof}
\begin{rem}
 We observe that even when the Markov chain is not ergodic (we have situations with periods~\(2\) and \(3\)), the above limit property holds.
\end{rem}

\begin{lemma}[Detailed balance]
The following detailed balance equation is satisfied:
\begin{align}\label{eq:detalied_balance}
%\begin{aligned}
	 (\pi_m)_{k+1} (P_{II,m})_{k,l} &= (\pi_m)_{l+1} (P_{I,m})_{l,k}, 
 & k,l&\in\{1,\ldots,m\} .
%\end{aligned}
\end{align}
%for \( k,l\in\{1,\ldots,m\} \). 
\end{lemma}
\begin{proof}
Let us recall that
 \begin{align*}
 %\begin{aligned}
 (P_{II,m} )_{k,l}&=\frac{1}{x_{m,m}}\frac{b_m^{(l)}}{b_m^{(k)}}T_{k,l}, &
 (P_{I,m} )_{k,l}&=\frac{1}{x_{m,m}}\frac{a_m^{(l)}}{a_m^{(k)}}T_{l,k}, & (\pi_m)_k&=\frac{ b_m^{(k-1)}a^{(k-1)}_{m}}{\sum_{l=1}^m b_m^{(l-1)}a^{(l-1)}_{m}},
 %\end{aligned}
 \end{align*}
so that \eqref{eq:detalied_balance} becomes
\begin{align*}
 b_m^{(k)}a^{(k)}_{m}\frac{b_m^{(l)}}{b_m^{(k)}}T_{k,l}&= b_m^{(l)}a^{(l)}_{m}\frac{a_m^{(k)}}{a_m^{(l)}}T_{k,l},
\end{align*}
which is an identity, and the result follows.
\end{proof}
\begin{pro}[Time-reversal]
 The Markov chains built on the stochastic matrices \( P_{I,m}\) and \( P_{II,m}\) are time-reversed versions of each other.
\end{pro}
\begin{proof}
 It follows from the previous Lemma, see \cite[\S 2.4.2]{Bremaud}.
\end{proof}

Finally we consider, the possibility of a PBF of \( T_m\), see \cite{stbbm,JP}. Since \( T_m\) is a tetradiagonal matrix, this factorization has to be of the form 
\begin{align}
 \label{factorizacionbidiagonalmultiple}
 T_{m}&=L_{1,m}L_{2,m}U_{m}, 
\end{align}
where
 \begin{align*}
\begin{aligned}
 L_{1,m}&\coloneq
 \begin{bNiceMatrix}[small]
 1 & 0 & & \cdots & & 0\\
 a_2 & 1 & \Ddots [] & & & \vdots\\
 0 & a_5 & \ddots& & & \\
 \vdots & \ddots[shorten-end=0pt]& \Ddots[shorten-end=-6pt]& & & \\
 & & & & & 0\\[6pt]
 0 & \cdots & & 0 & a_{3m-4} & 1
 \end{bNiceMatrix},&
 L_{2,m}&\coloneq
 \begin{bNiceMatrix}[small]
 1 & 0 & & \cdots & & 0\\
 a_3 & 1 & \Ddots & & & \vdots\\
 0 & a_6 & \ddots& & & \\
 \vdots & \ddots[shorten-end=0pt] & \Ddots[shorten-end=-6pt]& & & \\
 & & & & & 0\\[6pt]
 0 & \cdots & & 0 & a_{3m-3} & 1
 \end{bNiceMatrix},
&
 U_{m}\coloneq\begin{bNiceMatrix}[small]
 a_1 & 1 & 0 & \cdots & & 0\\
 0 & a_4 & 1 & \Ddots[shorten-end=2pt] & & \vdots\\
 \vdots & \ddots& \Ddots[shorten-end=-8pt]& \ddots & & \\
 & & & & & 0\\[6pt]
 & & & & & 1\\[8pt]
 0 & \cdots& & & 0 & a_{3m-2}
 \end{bNiceMatrix} ,
\end{aligned}
\end{align*}
with \( a_n>0 \). 
For \( n\in\{0,\ldots,m-1\} \), we define by
\begin{align*}
%\begin{aligned}
 D_{II,2,m}&\coloneq
 \diag
 \left[
 \begin{NiceMatrix}
 \dfrac{1}{d_{II,2,0}} & \cdots & \dfrac{1}{d_{II,2,m-1}}
 \end{NiceMatrix}
 \right], 
&
d_{II,2,n}&\coloneq a_{3n+1}b_m^{(n)}+b_m^{(n+1)}, 
%\end{aligned}
\end{align*}
and by
\begin{align*}
%\begin{aligned}
 D_{II,1,m}&\coloneq 
 \diag
 \left[
 \begin{NiceMatrix}
 \dfrac{1}{d_{II,1,0}} & \cdots & \dfrac{1}{d_{II,1,m-1}}
 \end{NiceMatrix}
 \right] ,
&
 d_{II,1,n}&\coloneq
 a_{3n-2}a_{3n}b_m^{(n-1)}+(a_{3n}+a_{3n+1})b_m^{(n)}+b_m^{(n+1)}.
%\end{aligned}
\end{align*}

We have set \( a_{-2}=a_{-1}=a_0=0 \). Similarly,
for \( n\in\{1,\ldots,m\}\) we denote by 
\begin{align*}
%\begin{aligned}
 D_{I,2,m}&\coloneq
 \diag
 \left[
 \begin{NiceMatrix}
 \dfrac{1}{d_{I,2,1}} & \cdots & \dfrac{1}{d_{I,2,m}}
\end{NiceMatrix}
 \right]
 , &
 d_{I,2,n}&\coloneq a_{m}^{(n)}+a_{3n-1} a_{m}^{(n+1)},
%\end{aligned}
\end{align*}
and by 
\begin{align*}
%\begin{aligned}
 D_{I,1,m}&\coloneq
 \diag
 \left[
 \begin{NiceMatrix}
 \dfrac{1}{d_{I,1,1}} & \cdots &\dfrac{1}{d_{I,1,m}}
 \end{NiceMatrix}
 \right], 
&
 d_{I,1,n}&\coloneq a_{m}^{(n)}+(a_{3n-1}+a_{3n}) a_{m}^{(n+1)}+a_{3n}a_{3n+2} a_{m}^{(n+2)}.
%\end{aligned}
\end{align*}

\begin{pro}[Pure Birth-Pure Death Factorization]
 \label{profactorizacionestocasticamultiple}
 The PBF \eqref{factorizacionbidiagonalmultiple} can be used to obtain stochastic bidiagonal factorizations of the stochastic matrices \( P_{II,m}\) and \( P_{I,m}\) as follows:
 \begin{align*}
 P_{II,m} &= \Pi_{II,1,m} \Pi_{II,2,m} \Upsilon_{II,m}, &
 P_{I,m} &= \Upsilon_{I,m} \Pi_{I,2,m} \Pi_{I,1,m},
 \end{align*}
 where the pure birth and pure death transition matrices are given by:
 \begin{align*}
 \Pi_{II,1,m} &\coloneq \frac{1}{x_{m,m}} \sigma_{II,m}^{-1} L_{1,m} D_{II,1,m}^{-1}, &
 \Pi_{II,2,m} &\coloneq D_{II,1,m} L_{2,m} D_{II,2,m}^{-1}, \\
 \Upsilon_{II,m} &\coloneq D_{II,2,m} U_m \sigma_{II,m}, &
 \Upsilon_{I,m} &\coloneq \frac{1}{x_{m,m}} \sigma_{I,m}^{-1} U_m^\top D_{I,1,m}^{-1}, \\
 \Pi_{I,2,m} &\coloneq D_{I,1,m} L_{2,m}^\top D_{I,2,m}^{-1}, &
 \Pi_{I,1,m} &\coloneq D_{I,2,m} L_{1,m}^\top \sigma_{I,m}.
 \end{align*}
\end{pro}
\begin{proof}
 We will follow the same reasoning as in Proposition \ref{profactorizacionestocasticaescalar}. Let's start with \( P_{II,m} \). As we know, we can write it as
 \begin{align*}
 P_{II,m} &= \frac{1}{x_{m,m}} \sigma_{II,m}^{-1} T_m \sigma_{II,m} = \frac{1}{x_{m,m}} \sigma_{II,m}^{-1} L_{1,m} L_{2,m} U_m \sigma_{II,m} \\
 &= \frac{1}{x_{m,m}} \sigma_{II,m}^{-1} L_{1,m} D_{II,1,m}^{-1} D_{II,1,m} L_{2,m} D_{II,2,m}^{-1} D_{II,2,m} U_m \sigma_{II,m} = \Pi_{II,1,m} \Pi_{II,2,m} \Upsilon_{II,m},
 \end{align*}
 where \( D_{II,2,m}\) is a diagonal \( m\times m\) matrix such that \( \Upsilon_{II,m}\) is stochastic, and \( D_{II,1,m}\) is also a diagonal \( m\times m\) matrix obtained by imposing \( \Pi_{II,2,m}\) to be stochastic. 
Finally, 
%\( \Pi_{II,1,m}\) is stochastic since the product of stochastic matrices is stochastic.
 as \( P_{II,m}\) is stochastic we have 
\begin{align*}
%\text{\( P_m \) is stochastic and} &&
\begin{bNiceMatrix}
1 & \Cdots & 1
\end{bNiceMatrix}^\top
=
P_{II,m} 
\begin{bNiceMatrix}
1 & \Cdots & 1
\end{bNiceMatrix}^\top
= 
\Pi_{II,1,m} 
%\textcolor{blue}{
\Pi_{II,2,m}
%}
 \Upsilon_{II,m}
\begin{bNiceMatrix}
1 & \Cdots & 1
\end{bNiceMatrix}^\top
=
\Pi_{II,1,m}
\begin{bNiceMatrix}
1 & \Cdots & 1
\end{bNiceMatrix}^\top ,
\end{align*}
and so \( \Pi_{II,1,m}\) is stochastic.
 \( \Pi_{II,1,m}\) and \( \Pi_{II,2,m}\) are lower bidiagonal, while \( \Upsilon_{II,m}\) is upper bidiagonal.
 
 With a completely analogous procedure, we can define a type I stochastic bidiagonal factorization of \( P_{I,m}\) as follows:
 \begin{align*}
 P_{I,m} &= \frac{1}{x_{m,m}} \sigma_{I,m}^{-1} T_m^\top \sigma_{I,m}= \frac{1}{x_{m,m}} \sigma_{I,m}^{-1} U_m^\top L_{2,m}^\top L_{1,m}^\top \sigma_{I,m} 
 \\
 & = \frac{1}{x_{m,m}} \sigma_{I,m}^{-1} U_m^\top D_{I,1,m}^{-1} D_{I,1,m} L_{2,m}^\top D_{I,2,m}^{-1} D_{I,2,m} L_{1,m}^\top \sigma_{I,m} = \Upsilon_{I,m} \Pi_{I,2,m} \Pi_{I,1,m},
 \end{align*}
 where
 \( \Pi_{I,1,m}\) is stochastic by an appropriate chosen of \( D_{I,2,m} \),
 \( \Pi_{I,2,m}\) is stochastic by an appropriate chosen of \( D_{I,1,m} \),
 and\( \Upsilon_{I,m}\) is stochastic since 
 %the product of stochastic matrices is stochastic. 
\begin{align*}
%\text{\( P_m \) is stochastic and} &&
\begin{bNiceMatrix}
1 & \Cdots & 1
\end{bNiceMatrix}^\top
=
P_{I,m} 
\begin{bNiceMatrix}
1 && \Cdots & 1
\end{bNiceMatrix}^\top
= 
\Upsilon_{I,m} \Pi_{I,2,m} \Pi_{I,1,m} 
\begin{bNiceMatrix}
1 & \Cdots & 1
\end{bNiceMatrix}^\top
=
\Upsilon_{I,m}
\begin{bNiceMatrix}
1 & \Cdots & 1
\end{bNiceMatrix}^\top .
\end{align*}
In this case, \( \Pi_{I,1,m}\) and \( \Pi_{I,2,m}\) are upper bidiagonal, while \( \Upsilon_{I,m}\) is lower bidiagonal.
\end{proof}

\begin{rem}
Here, \( \Pi_{II,1,m}\) and \( \Pi_{II,2,m}\) are transition matrices for pure birth Markov chains, while \( \Upsilon_{II,m}\) is a transition matrix for a pure death Markov chain.
\end{rem}

%\enlargethispage{.25cm}

\section[Finite Markov chains in the multiple Askey scheme]{Finite Markov chains beyond birth and death in the multiple Askey scheme}
Let us now apply this general strategy to the families of multiple orthogonal polynomials, specifically focusing on the Hahn multiple orthogonal polynomials and its descendants within the Askey scheme \cite{AskeyII, Arvesu, HahnI}, for which the recurrence tetradiagonal matrix is nonnegative. This includes the Jacobi--Piñeiro, multiple Meixner of the second kind, and multiple Laguerre of the first kind.

In contrast to the standard Askey scheme, where all the descendants admit a Markov chain, in the multiple scenario, only approximately half of them have such a stochastic matrix.
%\textcolor{blue}{(see~Figure~\ref{fig:new})}.
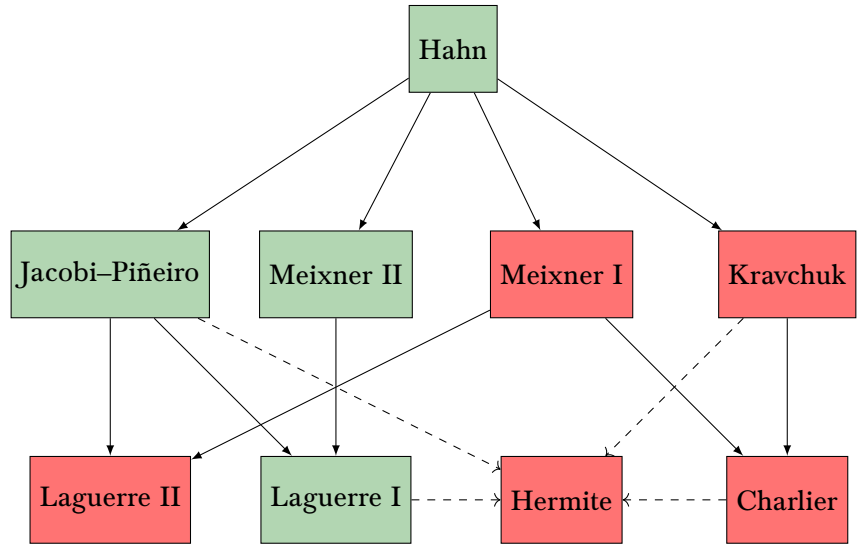
\begin{center}
%\begin{figure}[htbp]
%\centering
 \begin{tikzpicture}[node distance=3cm]
 \node[fill=ForestGreen!35,block] (a) {Jacobi--Piñeiro}; 
 \node[ fill=ForestGreen!35,block, right of=a] (b) {Meixner II};
 \node[ fill=Red!55,block, right of = b] (c) {Meixner I};
 \node[ fill=Red!55,block, right of=c] (f) {Kravchuk}; 
 \node[ fill=ForestGreen!35,block,above of = c, left =.85] (d) {Hahn};
 \node[fill=ForestGreen!35,block, below of = b] (e) {Laguerre I};
 \node[fill=Red!55,block, below of = a] (g) {Laguerre II};
 \node[ fill=Red!55,block, below of =f] (i) {Charlier};
 \node[ fill=Red!55, block, below of =c] (j) {Hermite};
 \draw[-latex] (d)--(a); 
 \draw[-latex] (d)--(b);
 \draw[-latex] (d)--(c);
 \draw[-latex] (d)--(f);
 \draw[-latex] (a)--(e);
 \draw[-latex] (b)--(e);
 \draw[-latex] (c)--(i);
 \draw[-latex] (c)--(g);
 \draw[-latex] (f)--(i);
 \draw[-latex] (a)--(g);
 \draw[dashed,->] (e)--(j);
 \draw[dashed,->] (a)--(j);
 \draw[dashed,->] (f)--(j);
 \draw[dashed,->] (i)--(j);
 \draw (4.5,-4.8) node
 {\begin{minipage}{0.7\textwidth}
 \begin{center}\small
 \textbf{\small
 Descendants of Hahn in the multiple Askey scheme} \\
 In green those with nonnegative Hessenberg matrices
 \end{center}
 \end{minipage}};
 \end{tikzpicture}
\end{center}
%\caption{Descendants of Hahn in the multiple Askey scheme: 
%In green those with nonnegative Hessenberg matrices}
%\label{fig:new}
%\end{figure}
Similar to the previous section, all of these polynomial families require the use of the generalized hypergeometric series \eqref{eq:generalized_hypergeometric_series}. However, in the multiple case, we need to go a step further since many of these polynomials are expressed in terms of the Kampé de Fériet series \cite{Srivastava, slater}:
\begin{multline*}
 \KF{p:r;s}{q:n;k}{(a_1,\ldots,a_p):(b_1,\ldots,b_r);(c_1,\ldots,c_s)}{(\alpha_1,\ldots,\alpha_q):(\beta_1,\ldots,\beta_n);(\gamma_1,\ldots,\gamma_k)}{x,y}
 \\
=
\sum_{l=0}^{\infty}\sum_{m=0}^{\infty}\dfrac{(a_1)_{l+m}\cdots(a_p)_{l+m}}{(\alpha_1)_{l+m}\cdots(\alpha_q)_{l+m}}\dfrac{(b_1)_l\cdots(b_r)_l}{(\beta_1)_l\cdots(\beta_n)_l}\dfrac{(c_1)_m\cdots(c_s)_m}{(\gamma_1)_m\cdots(\gamma_k)_m}\dfrac{x^l}{l!}\dfrac{y^m}{m!}.
\end{multline*}
For each polynomial family, we obtain examples of both type II and type I chains with \(7\) or \(6\) states. We also provide the corresponding bidiagonal factorizations and steady states.
As mentioned earlier for the tridiagonal case, a similar approach is 
used in this scenario. The strategy employed in this study is as follows: our main objective is to numerically approximate the largest zero, denoted as \( x_{m,m} \), of the orthogonal polynomial
%\textcolor{blue}{
\( B^{(m)}(x) \).
%}. 
To achieve this, we utilize the explicit hypergeometric representation of the orthogonal polynomials along with the given nonnegative recursion coefficients. By employing these expressions, we calculate numerical approximations for various quantities, including the associated stochastic matrix \( P_{I,m}\) and \( P_{II,m} \), the steady state of the Markov chain, the expected return times, and the stochastic factorization into pure birth and pure death factors.

Let us emphasize again that in order to carry out these computations accurately and efficiently, we have developed a dedicated Mathematica code that is specifically tailored for this task. The code incorporates the required algorithms and numerical techniques to effectively handle the calculations involved in obtaining the desired approximations.

%\enlargethispage{.25cm}

\subsection{Multiple Hahn Markov chains}
The multiple Hahn polynomials of type II, denoted as 
\( Q_{(n_1,n_2)}=Q_{(n_1,n_2)}(x;\alpha_1,\alpha_2,\beta,N)\)
 and of type I, denoted as
\( Q_{(n_1,n_2),i}=Q_{(n_1,n_2),i}(x;\alpha_1,\alpha_2,\beta,N) \), 
for \( i \in\{1,2\} \), respectively, satisfy their respective discrete orthogonality relations, given by:
\begin{align*}
%\begin{aligned}
& \sum_{k=0}^{N}(-N+k)_jQ_{(n_1,n_2)}(k)w_i(k) =0, 
 \quad j \in \{0,\ldots,n_i-1\} , \quad i\in\{1,2\} ,
 \\
%\end{aligned}
%\end{align*}
%and
%\[\begin{aligned}
& \sum_{k=0}^{N}(-N+k)_j\left(Q_{(n_1,n_2),1}(k)w_1(k)+Q_{(n_1,n_2),2}
 (k)w_2(k)\right)
 =0, \quad j \in\{0,\ldots,n_1+n_2-2\} ,
%\end{aligned}
\end{align*}
%for \( j\in\{0,\ldots,n_1+n_2-2\} \), 
 where the weight functions are defined as
\( w_i
%\textcolor{blue}{
 (x;\alpha_i,\beta,N)
 %}
\coloneq \dfrac{\Gamma(\alpha_i+x+1)}{\Gamma(\alpha_i+1)\Gamma(x+1)}\dfrac{\Gamma(\beta+N-x+1)}{\Gamma(\beta+1)\Gamma(N-x+1)} \), 
for \( i \in\{1,2\} \). These orthogonality relations hold over the set \( \lbrace0,\ldots,N\rbrace \), with \( N\in\mathbb N_0 \), \( \alpha_1,\alpha_2,\beta>-1 \), and to ensure an AT type system, the condition \( \alpha _1-\alpha_2\not\in\mathbb Z\) must be satisfied.

The coefficients \( b_{m}(\alpha_1,\alpha_2,\beta,N), c_{m}(\alpha_1,\alpha_2,\beta,N)\) and \( d_{m}(\alpha_1,\alpha_2,\beta,N)\) of the recurrence relations~\eqref{recurrenciaII} and~\eqref{recurrenciaI} can be found in~\cite[section 4.5]{Arvesu}:
\begin{align*}
%\label{aHahnmultiple}
%\begin{aligned}
 b_{2m}&=\begin{multlined}[t][.7\textwidth]
 A(m,m,\alpha_1,\alpha_2,\beta,N)+A(m,m,\alpha_2,\alpha_1+1,\beta,N)\\+C(m,m+1,\alpha_1,\alpha_2,\beta,N)+D(m,m,\alpha_1,\alpha_2,\beta,N),
 \end{multlined}\\
 b_{2m+1}&=\begin{multlined}[t][0.7\textwidth]
 A(m,m+1,\alpha_2,\alpha_1,\beta,N)+A(m+1,m,\alpha_1,\alpha_2+1,\beta,N)
 \\
 +C(m+1,m+2,\alpha_2,\alpha_1,\beta,N)+D(m,m+1,\alpha_2,\alpha_1,\beta,N),
 \end{multlined}\\
 c_{2m}&=\begin{multlined}[t][.7\textwidth]
 (A(m,m,\alpha_1,\alpha_2,\beta,N)+A(m,m,\alpha_2,\alpha_1+1,\beta,N)\\+D(m,m,\alpha_1,\alpha_2,\beta,N)) {C}(m,m+1,\alpha_2,\alpha_1,\beta,N),
 \\
 +A(m,m,\alpha_1,\alpha_2,\beta,N)B(m,m,\alpha_1,\alpha_2,\beta,N),
 \end{multlined}\\
 c_{2m+1}&=\begin{multlined}[t][.7\textwidth]
 (A(m,m+1,\alpha_2,\alpha_1,\beta,N)+A(m+1,m,\alpha_1,\alpha_2+1,\beta,N)
 \\
 +D(m,m+1,\alpha_2,\alpha_1,\beta,N)){C}(m+1,m+1,\alpha_1,\alpha_2,\beta,N)
 \\
+A(m,m+1,\alpha_2,\alpha_1,\beta,N)B(m,m+1,\alpha_2,\alpha_1,\beta,N),
 \end{multlined}\\
 d_{2m}&=A(m,m,\alpha_1,\alpha_2,\beta,N)B(m,m,\alpha_1,\alpha_2,\beta,N)C(m,m,\alpha_1,\alpha_2,\beta,N),\\
 d_{2m+1}&=\begin{multlined}[t][.8\textwidth]
 A(m,m+1,\alpha_2,\alpha_1,\beta,N)B(m,m+1,\alpha_2,\alpha_1,\beta,N) C(m,m+1,\alpha_2,\alpha_1,\beta,N)
 \end{multlined}.
%\end{aligned}
\end{align*}
where 
\begin{align}
\label{coefsauxiliares} 
\begin{aligned}
 A&=\dfrac{n_1(n_1+n_2+\alpha_2+\beta)(n_1+n_2+\beta)(N+n_1+\alpha_1+\beta+1)}{(n_1+2n_2+\alpha_2+\beta)(2n_1+n_2+\alpha_1+\beta)(2n_1+n_2+\alpha_1+\beta+1)},\\
 B&=\dfrac{(n_1+\alpha_1-\alpha_2)(n_1+n_2+\alpha_1+\beta)(n_1+n_2+\beta-1)(N-n_1-n_2+1)}{(n_1+2n_2+\alpha_2+\beta-1)(2n_1+n_2+\alpha_1+\beta)(2n_1+n_2+\alpha_1+\beta-1)},\\
 C&=\dfrac{(n_1+\alpha_1)(n_1+n_2+\alpha_1+\beta-1)(n_1+n_2+\alpha_2+\beta-1)(N-n_1-n_2+2)}{(n_1+2n_2+\alpha_2+\beta-2)(2n_1+n_2+\alpha_1+\beta-2)(2n_1+n_2+\alpha_1+\beta-1)},\\
 D&=\dfrac{n_1n_2(n_1+n_2+\beta)}{(2n_1+n_2+\alpha_1+\beta+1)(n_1+2n_2+\alpha_2+\beta)}.
\end{aligned}
\end{align}

\begin{pro}
 The recurrence matrix \( T_m\) is nonnegative whenever
 \begin{align*}
 -1<\alpha_1-\alpha_2<1.
 \end{align*}
\end{pro}

In \cite{Bidiagonal_factorization}, we proved the following result:
\begin{teo}\label{Pro:PBF_Hahn}
The recurrence matrix \( T_m\) admits a PBF in the form \eqref{factorizacionbidiagonalmultiple} whenever 
\begin{align*}
 -1<\alpha_1-\alpha_2<0
\end{align*} 
with coefficients \( a_{n}(\alpha_1,\alpha_2,\beta,N)\) given by
\begin{align*}
 a_{6n+1}&=
 \dfrac{(N-2n)(\alpha_1+1+n)(\alpha_1+\beta+2n+1)(\alpha_2+\beta+2n+1)}{(\alpha_1+\beta+3n+1)_2(\alpha_2+\beta+3n+1)},\\
 a_{6n+4}&=\dfrac{(N-2n-1)(\alpha_2+1+n)(\alpha_1+\beta+2n+2)(\alpha_2+\beta+2n+2)}{(\alpha_1+\beta+3n+3)(\alpha_2+\beta+3n+2)_2},\\
a_{6n+2}&=\begin{multlined}[t][.7\textwidth]\dfrac{(N-2n)(n)_n(\beta+2n+1)(\alpha_2-\alpha_1+n)(\alpha_2+\beta+n+1)}{(n+1)_n(\alpha_1+\beta+3n+2)(\alpha_2+\beta+3n+1)_2}
 \dfrac{\pFq{3}{2}{-n,-N,\alpha_2-\alpha_1-n}{-2n+1,\alpha_2+\beta+n+1}{1}}{\pFq{3}{2}{-n,-N,\alpha_2-\alpha_1-n}{-2n,\alpha_2+\beta+n+2}{1}},\end{multlined}\\
 a_{6n+5}&=\begin{multlined}[t][.87\textwidth]\dfrac{(n+1)(N-2n-1)(\beta+2n+2)(\alpha_1-\alpha_2+n+1)(\alpha_1+\beta+2+n+N)}{(2n+1)(\alpha_1+\beta+3n+3)_2(\alpha_2+\beta+3n+3)}
 \dfrac{\pFq{3}{2}{-n,-N,\alpha_2-\alpha_1-n}{-2n,\alpha_2+\beta+n+2}{1}}{\pFq{3}{2}{-n-1,-N,\alpha_2-\alpha_1-n-1}{-2n-1,\alpha_2+\beta+n+2}{1}},\end{multlined}\\
 a_{6n+3}&=\begin{multlined}[t][.8\textwidth]\dfrac{(2n+1)(\beta+2n+1)(\alpha_1+\beta+2n+2)(\alpha_2+\beta+2n+2)}{(\alpha_1+\beta+3n+2)_2(\alpha_2+\beta+3n+2)}
 \dfrac{\pFq{3}{2}{-n-1,-N,\alpha_2-\alpha_1-n-1}{-2n-1,\alpha_2+\beta+n+2}{1}}{\pFq{3}{2}{-n,-N,\alpha_2-\alpha_1-n}{-2n,\alpha_2+\beta+n+2}{1}},\end{multlined}\\
 a_{6n+6}&=\begin{multlined}[t][.87\textwidth]\dfrac{2(n+1)(\beta+2n+2)(\alpha_1+\beta+2n+3)(\alpha_2+\beta+2n+3)(\alpha_2+\beta+2+n+N)}{(\alpha_1+\beta+3n+4)(\alpha_2+\beta+3n+3)_2(\alpha_2+\beta+n+2)} \dfrac{\pFq{3}{2}{-n-1,-N,\alpha_2-\alpha_1-n-1}{-2n-2,\alpha_2+\beta+n+3}{1}}{\pFq{3}{2}{-n-1,-N,\alpha_2-\alpha_1-n-1}{-2n-1,\alpha_2+\beta+n+2}{1}}.\end{multlined}
\end{align*}
\end{teo}
\begin{rem}
 It is worth noting that the region of parameters \( - 1 < \alpha_1 - \alpha_2 < 1 \), with \( \alpha _1\) and \( \alpha _2\) both greater than \( - 1 \), guarantees a nonnegative matrix. This region forms a band that includes the semi-band \( -1 < \alpha_1 - \alpha_2 < 0 \), where \( \alpha _1\) and \( \alpha _2\) are both greater than \( - 1 \), ensuring a PBF. This situation will reappear for all the Askey descendant polynomials to be discussed later. See the diagram above.
\end{rem}
\begin{center}
 \begin{tikzpicture}[arrowmark/.style 2 args={decoration={markings,mark=at position #1 with \arrow{#2}}},scale=1]
 \begin{axis}[axis lines=middle,axis equal,grid=both,xmin=-1, xmax=3,ymin=-1.5, ymax=3.5,
 xticklabel,yticklabel,disabledatascaling,xlabel=\( \alpha _1 \), ylabel=\( \alpha _2 \), every axis x label/.style={
 at={(ticklabel* cs:1)},
 anchor=south west,
 },
 every axis y label/.style={
 at={(ticklabel* cs:1.0)},
 anchor=south west,
 },grid style={line width=.1pt, draw=Bittersweet!10},
 major grid style={line width=.2pt,draw=Bittersweet!50},
 minor tick num=4,
 enlargelimits={abs=0.09},
 axis line style={latex'-latex'},Bittersweet]
 
 \draw [fill=DarkSlateBlue!20,opacity=.2,dashed,thick] (-1,4)--(-1,-1)--(5,-1)--(5,4)--(-1,4) ;
 \draw [fill=DarkSlateBlue!30,opacity=.5,dashed,thick] (-1,-1)--(-1,0)--(3,4) node[above, Black,sloped, pos=0.5] {\( \alpha_1=\alpha_2+1\)}--(5,5)--(5,4) --(0,-1) node[below, Black,sloped, pos=0.6] {\( \alpha_1=\alpha_2-1\)}--(-1,-1);
 \draw [
 pattern=north west lines, pattern color=DarkSlateBlue!50,opacity=.6,dashed] (-1,-1)--(0,-1)--(5,4)--(5,5)--(-1,-1) ;
 \draw [fill=DarkSlateBlue!30,opacity=.5,dashed,thick] (-1,-1) -- (4,4 )node[below, Black,sloped, pos=0.4] {\( \alpha_1=\alpha_2\)};
 \draw[thick,black] (axis cs:-1,0) circle[radius=2pt,opacity=0.2,fill]node[left,above ] {\( -1\)} ;
 \draw[thick,black] (axis cs:0,-1)circle[radius=2pt,opacity=0.2,fill] node[right,below ] {\( -1\)} ;
 \node[anchor = north east,Bittersweet] at (axis cs: 4.1,2.7) {\( \mathcal R\)} ;
 \node[anchor = north east,Bittersweet] at (axis cs: 0.1,.55) {\( \mathcal S_+\)} ;
 \node[anchor = north east,Bittersweet] at (axis cs: 0.7,0.05) {\( \mathcal S_-\)} ;
 \end{axis}
 \draw (3.5,-0.55) node
 {\begin{minipage}{0.80\textwidth}
 \begin{center}\small
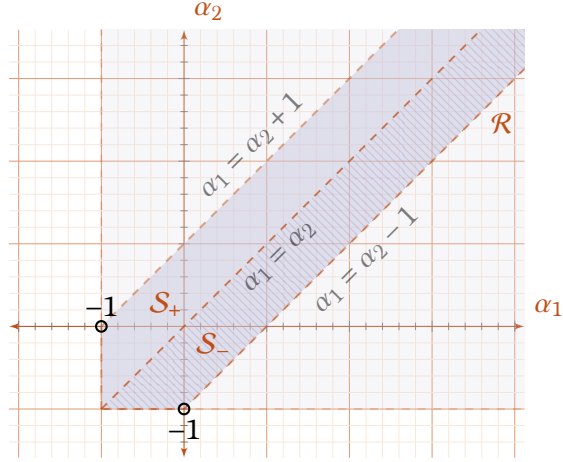
 \textbf{Allowed parameter region \(\mathcal R \), nonnegativeness band \(\mathcal S_+\cup\mathcal S_-\) and PBF semi-band \(\mathcal S_-\)}
 \end{center}
 \end{minipage}};
 \end{tikzpicture}
\end{center}

Finally, the type II polynomials are given by the expression provided in \cite[\S 4.5]{Arvesu}.
\begin{multline}
\label{HahnTipoII}
Q_{(n_1,n_2)}
=
 \dfrac{(\alpha_1+1)_{n_1}(\alpha_2+1)_{n_2}(-N)_{n_1+n_2}}{(\alpha_1+\beta+n_1+n_2+1)_{n_1}(\alpha_2+\beta+n_1+n_2+1)_{n_2}}\\\times
 \KF{2:3;1}{2:2;0}{(-x,\alpha_1+\beta+n_1+1):(-n_2,\alpha_1+n_1+1,\alpha_2+\beta+n_1+n_2+1);(-n_1)}{(-N,\alpha_1+1):(\alpha_2+1,\alpha_1+\beta+n_1+1);--}{1,1} ,
\end{multline}
and type I polynomials are given by the expression provided in \cite[\S 2]{HahnI}
\begin{multline}
\label{HahnTipoI}
 Q_{(n_1,n_2),i}=\dfrac{(-1)^{n_i-1}(N+1-n_1-n_2)!(n_1+n_2-2)!}{(n_1-1)!(n_2-1)!(\beta+1)_{n_1+n_2-1}(\alpha_i+\beta+n_1+n_2+n_i)_{N+1-n_1-n_2}} %\\\times
 \dfrac{(\hat{\alpha}_i+\beta+\hat{n}_i+1)_{n_1+n_2-1}}{(\alpha_i-\hat{\alpha}_i-\hat{n}_i+1)_{n_1+n_2-1}} \\\times
\KF{2:3;1}{2:2;0}{(-n_i+1,-N): (\alpha_i+\beta+n_1+n_2, \alpha_i-\hat{\alpha}_i-\hat{n}_i+1,-x); (\hat{\alpha}_i-\alpha_i-n_i+1)}{(-n_1-n_2+2,\hat{\alpha}_i+\beta+\hat{n}_i+1): (\alpha_i+1, -N); --}{1,1},
\end{multline}
for \( i \in\{1,2\} \), the type I polynomials can be obtained by substituting \( \hat {\alpha}_i\coloneq \delta_{i,2}\alpha_1+\delta_{i,1}\alpha_2\) and \( \hat {n}_i\coloneq \delta_{i,2}n_1+\delta_{i,1}n_2=n_1+n_2-n_i\) into the respective expressions. Both type II and I polynomials exist for \( n_1+n_2\leq N \). 

\begin{exa}
 For \( m=7\) and the chosen parameters \( \alpha _1=0.4 \), \( \alpha_2=0.6 \), \( \beta=0.75 \), and \( N=10 \).
 Note that in this case we can construct the polynomials \( B_{(n_1,n_2)} \) in the stepline from \((0,0) \) til \( (5,5 ) \), i.e. we can enumerate them from \( 0 \) til \( 10 \). Hence, \( m= 7\) is admissible and we obtain the following
 \( 7\times7\) stochastic matrices along with their corresponding pure birth/pure death factorization
\begin{align*}
 & P_{II,7}
 %\textcolor{blue}{
 (0.4,0.6,0.75,10) %}
 \approx
\begin{bNiceArray}[small]{ccccccc}
 0.46 & 0.55 & 0 & 0 & 0 & 0 & 0 \\
 0.14 & 0.43& 0.43& 0 & 0 & 0 & 0 \\
 0.02 & 0.19 & 0.48 & 0.31& 0 & 0 & 0 \\
 0 & 0.02& 0.24& 0.48 & 0.26 & 0 & 0 \\
 0 & 0 & 0.04& 0.27 & 0.50 & 0.19 & 0 \\
 0 & 0 & 0 & 0.04& 0.33 & 0.50 & 0.13\\
 0 & 0 & 0 & 0 & 0.06 & 0.42 & 0.52 \\
\end{bNiceArray}
\\
& \phantom{olaolaola}
\approx\begin{multlined}[t][.8\textwidth]
\begin{bNiceArray}[small]{ccccccc}
 1 & 0 & 0 & 0 & 0 & 0 & 0 \\
 0.06 & 0.94 & 0 & 0 & 0 & 0
 & 0 \\
 0 & 0.18& 0.82& 0 & 0 & 0 &
 0 \\
 0 & 0 & 0.13& 0.87 & 0 & 0 &
 0 \\
 0 & 0 & 0 & 0.18 & 0.82 & 0 &
 0 \\
 0 & 0 & 0 & 0 & 0.15 & 0.85 &
 0 \\
 0 & 0 & 0 & 0 & 0 & 0.22 &
 0.78 \\
\end{bNiceArray}
\begin{bNiceArray}[small]{ccccccc}
 1 & 0 & 0 & 0 & 0 & 0 & 0 \\
 0.28& 0.72& 0 & 0 & 0 & 0 &
 0 \\
 0 & 0.37& 0.63 & 0 & 0 & 0 &
 0 \\
 0 & 0 & 0.52& 0.48 & 0 & 0 &
 0 \\
 0 & 0 & 0 & 0.61& 0.39& 0 &
 0 \\
 0 & 0 & 0 & 0 & 0.73 & 0.27 &
 0 \\
 0 & 0 & 0 & 0 & 0 & 0.90&
 0.10 \\
\end{bNiceArray}
\begin{bNiceArray}[small]{ccccccc}
 0.45 & 0.55 & 0 & 0 & 0 & 0 &
 0 \\
 0 & 0.37& 0.63& 0 & 0 & 0 &
 0 \\
 0 & 0 & 0.40 & 0.60 & 0 & 0 &
 0 \\
 0 & 0 & 0 & 0.37 & 0.63 & 0 &
 0 \\
 0 & 0 & 0 & 0 & 0.40& 0.60&
 0 \\
 0 & 0 & 0 & 0 & 0 & 0.42 &
 0.58 \\
 0 & 0 & 0 & 0 & 0 & 0 & 1 \\
\end{bNiceArray},
\end{multlined}
\end{align*}
\begin{center}
 \tikzset{decorate sep/.style 2 args={decorate,decoration={shape backgrounds,shape=circle,shape size=#1,shape sep=#2}}}
 \begin{tikzpicture}[start chain = going right,
 -Triangle, every loop/.append style = {-Triangle}]
 \foreach \i in {1,...,7}
 \node[state, on chain] (\i) {\i};
 \foreach
 \i/\txt in {1/\( 0.55 \), 2/\( 0.43 \), 3/\( 0.31 \), 4/\( 0.26 \), 5/\( 0.19 \), 6/\( 0.13\)}
 \draw let \n1 = { int(\i+1) } in
 (\i) edge[bend left,"\txt",color=Periwinkle] (\n1);
 
 \foreach
 \i/\txt in {1/\( 0.14 \), 2/\( 0.19\)/,3/\( 0.24 \), 4/\( 0.27 \), 5/\( 0.33 \), 6/\( 0.42\)}
 \draw let \n1 = { int(\i+1) } in
 (\n1) edge[bend left,above, "\txt",color=Mahogany,auto=right] (\i);
 
 \foreach
 \i/\txt in {1/\( 0.02 \), 2/\( 0.02 \), 3/\( 0.04 \), 4/\( 0.04 \), 5/\( 0.06\)}
 \draw let \n1 = { int(\i+2) } in
 (\n1) edge[bend left=45,color=RawSienna,"\txt", auto =right] (\i);
 
 \foreach \i/\txt in {1/\( 0.46 \), 2/\( 0.43 \), 3/\( 0.48 \), 4/\( 0.48 \), 5/\( 0.50 \), 6/\( 0.50 \), 7/\( 0.52\)}
 \draw (\i) edge[loop above,color=NavyBlue, "\txt"] (\i);
 \end{tikzpicture}
 \begin{tikzpicture}
 \draw (4,-1.8) node
 {\begin{minipage}{0.8\textwidth}
 \begin{center}\small
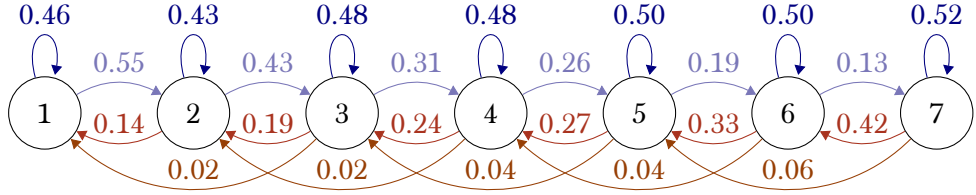
 \textbf{Type II multiple Hahn\( (0.4,0.6,0.75,
 %\textcolor{blue}{
 10 %}
 )\) Markov chain diagram}
 \end{center}
 \end{minipage}};
 \end{tikzpicture}
 \end{center}
as well as
\begin{align*}
 & P_{I,7}
 %\textcolor{blue}{
 (0.4,0.6,0.75,10) %}
 \approx
\begin{bNiceArray}[small]{ccccccc}
 0.45 & 0.43 & 0.12 & 0 & 0 & 0 & 0 \\
 0.18& 0.43& 0.35 & 0.04 & 0 & 0 & 0 \\
 0 & 0.23& 0.48 & 0.26 & 0.03 & 0 & 0 \\
 0 & 0 & 0.29& 0.48 & 0.22& 0.01 & 0 \\
 0 & 0 & 0 & 0.32 & 0.50 & 0.17 & 0.01 \\
 0 & 0 & 0 & 0 & 0.38& 0.50 & 0.12 \\
 0 & 0 & 0 & 0 & 0 & 0.48 & 0.52\\
\end{bNiceArray}
\\
&\phantom{olaolaola} \approx
\begin{multlined}[t][.8\textwidth]
 \begin{bNiceArray}[small]{ccccccc}
 1 & 0 & 0 & 0 & 0 & 0 & 0 \\
 0.41 & 0.59& 0 & 0 & 0 & 0 &
 0 \\
 0 & 0.54& 0.46 & 0 & 0 & 0 &
 0 \\
 0 & 0 & 0.66& 0.34& 0 & 0 &
 0 \\
 0 & 0 & 0 & 0.72& 0.28 & 0 &
 0 \\
 0 & 0 & 0 & 0 & 0.81 & 0.19 &
 0 \\
 0 & 0 & 0 & 0 & 0 & 0.92 &
 0.08 \\
\end{bNiceArray}
\begin{bNiceArray}[small]{ccccccc}
 0.53& 0.47 & 0 & 0 & 0 & 0 & 0
 \\
 0 & 0.59 & 0.41& 0 & 0 & 0 &
 0 \\
 0 & 0 & 0.52& 0.48 & 0 & 0 &
 0 \\
 0 & 0 & 0 & 0.52& 0.48 & 0 &
 0 \\
 0 & 0 & 0 & 0 & 0.51& 0.49&
 0 \\
 0 & 0 & 0 & 0 & 0 & 0.56 &
 0.44 \\
 0 & 0 & 0 & 0 & 0 & 0 & 1 \\
\end{bNiceArray}
\begin{bNiceArray}[small]{ccccccc}
 0.85 & 0.15& 0 & 0 & 0 & 0 &
 0 \\
 0 & 0.74& 0.26 & 0 & 0 & 0 &
 0 \\
 0 & 0 & 0.85 & 0.15 & 0 & 0 &
 0 \\
 0 & 0 & 0 & 0.86 & 0.14& 0 &
 0 \\
 0 & 0 & 0 & 0 & 0.91 & 0.09
 & 0 \\
 0 & 0 & 0 & 0 & 0 & 0.93 &
 0.07 \\
 0 & 0 & 0 & 0 & 0 & 0 & 1 \\
\end{bNiceArray}
\end{multlined}
\end{align*}
\begin{center}
 \tikzset{decorate sep/.style 2 args={decorate,decoration={shape backgrounds,shape=circle,shape size=#1,shape sep=#2}}}
\begin{tikzpicture}[start chain = going right,
 -Triangle, every loop/.append style = {-Triangle}]
 \foreach \i in {1,...,7}
 \node[state, on chain] (\i) {\i};
 \foreach
 \i/\txt in {1/\( 0.43 \), 2/\( 0.35 \), 3/\( 0.26 \), 4/\( 0.22 \), 5/\( 0.17 \), 6/\( 0.12\)}
 \draw let \n1 = { int(\i+1) } in
 (\i) edge[bend left,"\txt",color=Periwinkle,auto=right] (\n1);
 
 \foreach
 \i/\txt in {1/\( 0.18 \), 2/\( 0.23\)/,3/\( 0.29 \), 4/\( 0.32 \), 5/\( 0.38 \), 6/\( 0.48\)}
 \draw let \n1 = { int(\i+1) } in
 (\n1) edge[bend left,below, "\txt",color=Mahogany,] (\i);
 
 \foreach
 \i/\txt in {1/\( 0.12 \), 2/\( 0.04 \), 3/\( 0.03 \), 4/\( 0.01 \), 5/\( 0.01\)}
 \draw let \n1 = { int(\i+2) } in
 (\i) edge[bend left=45,below,color=MidnightBlue,"\txt", auto = right] (\n1);
 
 \foreach \i/\txt in {1/\( 0.45 \), 2/\( 0.43 \), 3/\( 0.48 \), 4/\( 0.48 \), 5/\( 0.50 \), 6/\( 0.50 \), 7/\( 0.52\)}
 \draw (\i) edge[loop below, 
 "\txt"] (\i);
 
\end{tikzpicture}
 \begin{tikzpicture}
 \draw (4,-1.8) node
 {\begin{minipage}{0.8\textwidth}
 \begin{center}\small
 \textbf{Type I multiple Hahn\( (0.4,0.6,0.75,
 %\textcolor{blue}{
 10 %}
 )\) Markov chain diagram}
 \end{center}
 \end{minipage}};
 \end{tikzpicture}
 \end{center}
The corresponding steady state vector for both Markov chains, calculated using the formula \eqref{estacionariomultiple}, is given~by:
\begin{align*}
\pi_7
%\textcolor{blue}{
(0.4,0.6,0.75,10) %}
\approx
\begin{bNiceMatrix}
0.04 &
0.13&
0.24&
0.25&
0.21&
0.10&
0.03
\end{bNiceMatrix},
\end{align*}
and the expected return times are
\begin{align*}
%\begin{aligned}
 (\bar t_7)_1&\approx 23.10, & (\bar t_7)_2&\approx 7.74, & (\bar t_7)_3&\approx 4.19, &
 (\bar t_7)_4&\approx 3.95,& (\bar t_7)_5&\approx 4.92, & (\bar t_7)_6&\approx 9.67, & (\bar t_7)_7&\approx 34.73.
% \end{aligned}
\end{align*}
\end{exa}

\subsection{Jacobi--Piñeiro finite Markov chains}
The Jacobi--Piñeiro type II 
 polynomials, denoted as 
\( P_{(n_1,n_2)}=P_{(n_1,n_2)}(x;\alpha_1,\alpha_2,\beta)\)
 and of type I 
\( P_{(n_1,n_2),i}=P_{(n_1,n_2),i}(x;\alpha_1,\alpha_2,\beta)\)
 for \( i \in\{1,2\} \), respectively, satisfy their respective continuous orthogonality relations of the form:
\begin{align*}
 \int_{0}^{1}x^jP_{(n_1,n_2)}(x)w_i(x)\d\mu(x)&=0, & j&\in\{0,\ldots,n_i-1\} , & i&\in\{1,2\} ,
\end{align*}
%for \( j\in\{0,\ldots,n_i-1\} \), \( i\in\{1,2\} \),
 and 
\begin{align*}
 \int_{0}^{1}x^j\left(P_{(n_1,n_2),1}(x)w_1(x) +P_{(n_1,n_2),2}(x)w_2(x)\right)\d\mu(x)=0, 
 & j&\in \{0,\ldots,n_1+n_2-2\} .
\end{align*}
%for \( j\in \{0,\ldots,n_1+n_2-2\} \). 
These relations hold with respect to the weight functions and measure defined as:
 \( w_i(x;\alpha_i)= x^{\alpha_i} \), \( i\in\{1,2\} \), \( \d\mu(x)=(1-x)^{\beta}\mathrm{d}x \). 
The weight functions and measure are defined over the interval \( [0,1] \). The parameters \( \alpha _1,\alpha_2,\beta\) are required to be greater than \( - 1 \), and in order to have an AT system, \( \alpha_1-\alpha_2\not\in\mathbb Z \).

The coefficients of the recurrence relations, \( b_{m}(\alpha_1,\alpha_2,\beta), c_{m}(\alpha_1,\alpha_2,\beta)\) and \( d_{m}(\alpha_1,\alpha_2,\beta) \), as shown in \eqref{recurrenciaII} and \eqref{recurrenciaI}, were first derived in \cite[\S 3.1]{Clasicos}. In \cite[\S 3.3]{ContinuosII}, they are expressed as follows:
\begin{align*}
 \begin{aligned}
 b_{2m}&=\begin{multlined}[t][.8\textwidth]\frac{A(m,m,\alpha_1,\alpha_2,\beta,N)}{N+\alpha_1+\beta+m+1}+\frac{A(m,m,\alpha_2,\alpha_1+1,\beta,N)}{N+\alpha_2+\beta+m+1}+\frac{C(m+1,m+1,\alpha_1,\alpha_2,\beta,N)}{N-2m},\end{multlined}\\
 b_{2m+1}&=\begin{multlined}[t][.8\textwidth]
 \frac{A(m,m+1,\alpha_2,\alpha_1,\beta,N)}{N+\alpha_2+\beta+m+1}+\frac{A(m+1,m,\alpha_1,\alpha_2+1,\beta,N)}{N+\alpha_1+\beta+m+2}
 +\frac{C(m+1,m+2,\alpha_2,\alpha_1,\beta,N)}{N-2m-1},
 \end{multlined}\\
 c_{2m}&=\begin{multlined}[t][.8\textwidth]
 \left(\frac{A(m,m,\alpha_1,\alpha_2,\beta,N)}{N+\alpha_1+\beta+m+1}+\frac{A(m,m,\alpha_2,\alpha_1+1,\beta,N)}{N+\alpha_2+\beta+m+1}\right)
 \frac{C(m,m+1,\alpha_2,\alpha_1,\beta,N)}{N-2m+1}
 \\
 +\frac{A(m,m,\alpha_1,\alpha_2,\beta,N)B(m,m,\alpha_1,\alpha_2,\beta,N)}{(N+\alpha_1+\beta+m+1)(N-2m+1)}, \end{multlined}\\
 c_{2m+1}&=\begin{multlined}[t][.8\textwidth]
 \left(\frac{A(m,m+1,\alpha_2,\alpha_1,\beta,N)}{N+\alpha_2+\beta+m+1}+\frac{A(m+1,m,\alpha_1,\alpha_2+1,\beta,N)}{N+\alpha_1+\beta+m+2}\right)
 \frac{C(m+1,m+1,\alpha_1,\alpha_2,\beta,N)}{N-2m}\\
 +\frac{A(m,m+1,\alpha_2,\alpha_1,\beta,N)B(m,m+1,\alpha_2,\alpha_1,\beta,N)}{(N+\alpha_2+\beta+m+1)(N-2m)},
 \end{multlined}\\
 d_{2m}&=\frac{A(m,m,\alpha_1,\alpha_2,\beta,N)B(m,m,\alpha_1,\alpha_2,\beta,N)C(m,m,\alpha_1,\alpha_2,\beta,N)}{(N+\alpha_1+\beta+m+1)(N-2m+1)(N-2m+2)},\\
 d_{2m+1}&=\frac{A(m,m+1,\alpha_2,\alpha_1,\beta,N)B(m,m+1,\alpha_2,\alpha_1,\beta,N)C(m,m+1,\alpha_2,\alpha_1,\beta,N)}{(N+\alpha_2+\beta+m+1)(N-2m)(N-2m+1)}.
 \end{aligned}
\end{align*}
Being \( A,B,C,D\) the functions defined in \eqref{coefsauxiliares}. These coefficients are all positive if \( - 1<\alpha_1-\alpha_2<1 \). 

In \cite{Bidiagonal_factorization}, we proved the following result:
\begin{align*}
 a_{6n+1}&=\dfrac{(\alpha_1+1+n)(\alpha_1+\beta+2n+1)(\alpha_2+\beta+2n+1)}{(\alpha_1+\beta+3n+1)_2(\alpha_2
 +\beta+3n+1)},\\
 a_{6n+4}&=\dfrac{(\alpha_2+1+n)(\alpha_1+\beta+2n+2)(\alpha_2+\beta+2n+2)}{(\alpha_1+\beta+3n+3)(\alpha_2+\beta+3n+2)_2},\\
 a_{6n+2}&=\dfrac{(\beta+2n+1)(\alpha_2-\alpha_1+n)(\alpha_2+\beta+2n+1)}{(\alpha_1+\beta+3n+2)(\alpha_2+\beta+3n+1)_2}, \\
 a_{6n+5}&=\dfrac{(n+1)(\beta+2n+2)(\alpha_2+\beta+2n+2)}{(\alpha_1+\beta+3n+3)_2(\alpha_2+\beta+3n+3)},\\
 a_{6n+3}&=\dfrac{(\beta+2n+1)(\alpha_1-\alpha_2+n+1)(\alpha_1+\beta+2n+2)}{(\alpha_1+\beta+3n+2)_2(\alpha_2+\beta+3n+2)},\\
 a_{6n+6}&=\dfrac{(n+1)(\beta+2n+2)(\alpha_1+\beta+2n+3)}{(\alpha_1+\beta+3n+4)(\alpha_2+\beta+3n+3)_2}.
\end{align*}
All of these coefficients are positive when \( - 1<\alpha_1-\alpha_2<0 \). This factorization has been proven to hold in \cite{genetico}, and it is also mentioned in \cite{Darboux}. For a recent discussion on the PBF of tetradiagonal matrices, see \cite{PBF_tetra}.

The type II polynomials were derived in \cite[\S 3.3]{ContinuosII}, and they can be expressed using hypergeometric functions
\begin{multline*}
 P_{(n_1,n_2)}
 =(-1)^{n_1+n_2}\dfrac{(\alpha_1+1)_{n_1}(\alpha_2+1)_{n_2}}{(n_1+n_2+\alpha_1+\beta+1)_{n_1}(n_1+n_2+\alpha_2+\beta+1)_{n_2}}\\
 \times\KF{1:3;1}{1:2;0}{(\alpha_1+\beta+n_1+1):(-n_2,\alpha_2+\beta+n_1+n_2+1,\alpha_1+n_1+1);(-n_1)}{(\alpha_1+1):(\alpha_2+1,\alpha_1+\beta+n_1+1);--}{x,x}.
\end{multline*}
The type I polynomials are \cite[\S 4.2]{JP}
\begin{multline*}
 P_{(n_1,n_2),i}
 =(-1)^{n_1+n_2-1}\dfrac{(\alpha_1+\beta+n_1+n_2)_{n_1}({\alpha}_2+\beta+n_1+n_2)_{{n}_2}}{(n_i-1)!(\hat{\alpha}_i-\alpha_i)_{\hat{n}_i}}
 \\\times\dfrac{\Gamma(\alpha_i+\beta+n_1+n_2)}{\Gamma(\beta+n_1+n_2)\Gamma(\alpha_i+1)}\pFq{3}{2}{-n_i+1,\alpha_i+\beta+n_1+n_2,\alpha_i-\hat{\alpha}_i-\hat{n}_i+1}{\alpha_i+1,\alpha_i-\hat{\alpha}_i+1}{x},
\end{multline*}
for \( i =1,2\) with \( \hat {\alpha}_i\coloneq \alpha_1\delta_{i,2}+\alpha_2\delta_{i,1}\) and \( \hat {n}_i\coloneq n_1\delta_{i,2}+n_2\delta_{i,1}=n_1+n_2-n_i \). 

\begin{exa}
 for \( m=7\) and the chosen parameters \( \alpha _1=0.4 \), \( \alpha_2=0.6 \), \( \beta=0.75 \), we obtain the following \( 7\times7\) stochastic matrices along with their corresponding pure birth/pure death factorizations:
\begin{align*}
 & P_{II,7} %\textcolor{blue}{
 (0.4,0.6,0.75) %}
 \approx
\begin{bNiceArray}[small]{ccccccc}
 0.47& 0.53 & 0 & 0 & 0 & 0 & 0 \\
 0.13& 0.44 & 0.43 & 0 & 0 & 0 & 0
 \\
 0.02& 0.17& 0.48 & 0.33& 0 &
 0 & 0 \\
 0 & 0.02 & 0.22 & 0.46 & 0.30&
 0 & 0 \\
 0 & 0 & 0.04& 0.24& 0.48 &
 0.24& 0 \\
 0 & 0 & 0 & 0.05& 0.31 & 0.46 &
 0.18 \\
 0 & 0 & 0 & 0 & 0.10& 0.42 & 0.48
 \\
\end{bNiceArray}\\
& \phantom{olaolaola}
\approx
%\begin{multlined}[t][.75\textwidth]
 \begin{bNiceArray}[small]{ccccccc}
 1 & 0 & 0 & 0 & 0 & 0 & 0 \\
 0.07& 0.93 & 0 & 0 & 0 & 0
 & 0 \\
 0 & 0.23 & 0.77& 0 & 0 & 0 &
 0 \\
 0 & 0 & 0.23& 0.77& 0 & 0 &
 0 \\
 0 & 0 & 0 & 0.32& 0.68& 0 &
 0 \\
 0 & 0 & 0 & 0 & 0.34& 0.66 &
 0 \\
 0 & 0 & 0 & 0 & 0 & 0.50 &
 0.50\\
\end{bNiceArray}
\begin{bNiceArray}[small]{ccccccc}
 1 & 0 & 0 & 0 & 0 & 0 & 0 \\
 0.22 & 0.78& 0 & 0 & 0 & 0 &
 0 \\
 0 & 0.22& 0.78& 0 & 0 & 0 &
 0 \\
 0 & 0 & 0.31 & 0.69 & 0 & 0 &
 0 \\
 0 & 0 & 0 & 0.32 & 0.68& 0 &
 0 \\
 0 & 0 & 0 & 0 & 0.42& 0.58&
 0 \\
 0 & 0 & 0 & 0 & 0 & 0.58 &
 0.42\\
\end{bNiceArray}
\begin{bNiceArray}[small]{ccccccc}
 0.47& 0.53 & 0 & 0 & 0 & 0 &
 0 \\
 0 & 0.41 & 0.59 & 0 & 0 & 0 &
 0 \\
 0 & 0 & 0.45 & 0.55& 0 & 0 &
 0 \\
 0 & 0 & 0 & 0.43& 0.57& 0 &
 0 \\
 0 & 0 & 0 & 0 & 0.49& 0.51 &
 0 \\
 0 & 0 & 0 & 0 & 0 & 0.54 &
 0.46 \\
 0 & 0 & 0 & 0 & 0 & 0 & 1 \\
\end{bNiceArray}.
%\end{multlined}
\end{align*}
\begin{center}
 \tikzset{decorate sep/.style 2 args={decorate,decoration={shape backgrounds,shape=circle,shape size=#1,shape sep=#2}}}
 \begin{tikzpicture}[start chain = going right,
 -Triangle, every loop/.append style = {-Triangle}]
 \foreach \i in {1,...,7}
 \node[state, on chain] (\i) {\i};
 \foreach
 \i/\txt in {1/\( 0.53 \), 2/\( 0.43 \), 3/\( 0.33 \), 4/\( 0.30 \), 5/\( 0.24 \), 6/\( 0.18\)}
 \draw let \n1 = { int(\i+1) } in
 (\i) edge[bend left,"\txt",color=Periwinkle] (\n1);
 
 \foreach
 \i/\txt in {1/\( 0.13 \), 2/\( 0.17\)/,3/\( 0.22 \), 4/\( 0.24 \), 5/\( 0.31 \), 6/\( 0.42\)}
 \draw let \n1 = { int(\i+1) } in
 (\n1) edge[bend left,above, "\txt",color=Mahogany,auto=right] (\i);
 
 \foreach
 \i/\txt in {1/\( 0.02 \), 2/\( 0.02 \), 3/\( 0.04 \), 4/\( 0.05 \), 5/\( 0.10\)}
 \draw let \n1 = { int(\i+2) } in
 (\n1) edge[bend left=45,color=RawSienna,"\txt", auto =right] (\i);
 
 \foreach \i/\txt in {1/\( 0.47 \), 2/\( 0.44 \), 3/\( 0.48 \), 4/\( 0.46 \), 5/\( 0.48 \), 6/\( 0.46 \), 7/\( 0.48\)}
 \draw (\i) edge[loop above,color=NavyBlue, "\txt"] (\i);
 \end{tikzpicture}
 \begin{tikzpicture}
 \draw (4,-1.8) node
 {\begin{minipage}{0.8\textwidth}
 \begin{center}\small
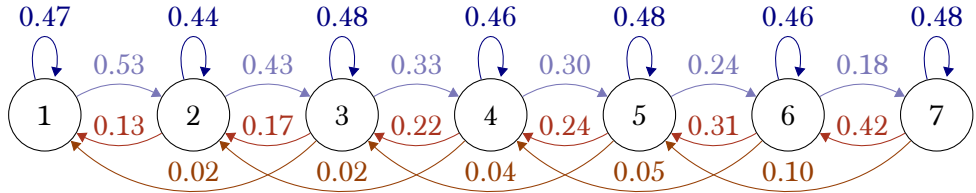
 \textbf{Type II Jacobi--Piñeiro \( (0.4,0.6,0.75)\) Markov chain diagram}
 \end{center}
 \end{minipage}};
 \end{tikzpicture}
\end{center}

\begin{align*}
 & P_{I,7}
 %\textcolor{blue}{
 (0.4,0.6,0.75) %}
 \approx
\begin{bNiceArray}[small]{ccccccc}
 0.47& 0.39 & 0.14& 0
 & 0 & 0 & 0 \\
 0.17 & 0.44 & 0.34& 0.05
 & 0 & 0 & 0 \\
 0 & 0.21& 0.48 & 0.26
 & 0.05 & 0 & 0 \\
 0 & 0 & 0.28& 0.46
 & 0.23& 0.03 & 0 \\
 0 & 0 & 0 & 0.32
 & 0.47 & 0.19 & 0.02\\
 0 & 0 & 0 & 0
 & 0.40 & 0.46 & 0.14\\
 0 & 0 & 0 & 0
 & 0 & 0.52& 0.48 \\
\end{bNiceArray}
\\
& \phantom{olaolaola} \approx
%\begin{multlined}[t][.75\textwidth]
\begin{bNiceArray}[small]{ccccccc}
 1 & 0 & 0 & 0
 & 0 & 0 & 0 \\
 0.37 & 0.63 & 0 & 0
 & 0 & 0 & 0 \\
 0 & 0.46& 0.54& 0
 & 0 & 0 & 0 \\
 0 & 0 & 0.57 & 0.43
 & 0 & 0 & 0 \\
 0 & 0 & 0 & 0.60
 & 0.40& 0 & 0 \\
 0 & 0 & 0 & 0
 & 0.69 & 0.31 & 0 \\
 0 & 0 & 0 & 0
 & 0 & 0.79& 0.21 \\
\end{bNiceArray}
\begin{bNiceArray}[small]{ccccccc}
 0.57& 0.43 & 0 & 0
 & 0 & 0 & 0 \\
 0 & 0.70 & 0.30 & 0
 & 0 & 0 & 0 \\
 0 & 0 & 0.67 & 0.33
 & 0 & 0 & 0 \\
 0 & 0 & 0 & 0.73
 & 0.27 & 0 & 0 \\
 0 & 0 & 0 & 0
 & 0.74 & 0.26 & 0 \\
 0 & 0 & 0 & 0
 & 0 & 0.83 & 0.17\\
 0 & 0& 0 & 0
 & 0 & 0 & 1 \\
\end{bNiceArray} 
\begin{bNiceArray}[small]{ccccccc}
 0.83& 0.17 & 0 & 0
 & 0 & 0 & 0 \\
 0 & 0.67 & 0.33& 0
 & 0 & 0 & 0 \\
 0 & 0 & 0.74 & 0.26
 & 0 & 0 & 0 \\
 0 & 0 & 0 & 0.72
 & 0.28& 0 & 0 \\
 0 & 0 & 0 & 0
 & 0.77 & 0.23& 0 \\
 0 & 0 & 0 & 0
 & 0 & 0.80& 0.20 \\
 0 & 0 & 0 & 0
 & 0 & 0 & 1 \\
\end{bNiceArray}.
%\end{multlined}
\end{align*}
\begin{center}
 \tikzset{decorate sep/.style 2 args={decorate,decoration={shape backgrounds,shape=circle,shape size=#1,shape sep=#2}}}
 \begin{tikzpicture}[start chain = going right,
 -Triangle, every loop/.append style = {-Triangle}]
 \foreach \i in {1,...,7}
 \node[state, on chain] (\i) {\i};
 \foreach
 \i/\txt in {1/\( 0.39 \), 2/\( 0.34 \), 3/\( 0.26 \), 4/\( 0.23 \), 5/\( 0.19 \), 6/\( 0.14\)}
 \draw let \n1 = { int(\i+1) } in
 (\i) edge[bend left,"\txt",color=Periwinkle,auto=right] (\n1);
 
 \foreach
 \i/\txt in {1/\( 0.17 \), 2/\( 0.21\)/,3/\( 0.28 \), 4/\( 0.32 \), 5/\( 0.40 \), 6/\( 0.52\)}
 \draw let \n1 = { int(\i+1) } in
 (\n1) edge[bend left,below, "\txt",color=Mahogany,] (\i);
 
 \foreach
 \i/\txt in {1/\( 0.14 \), 2/\( 0.05 \), 3/\( 0.05 \), 4/\( 0.03 \), 5/\( 0.02\)}
 \draw let \n1 = { int(\i+2) } in
 (\i) edge[bend left=45,below,color=MidnightBlue,"\txt", auto = right] (\n1);
 
 \foreach \i/\txt in {1/\( 0.47 \), 2/\( 0.44 \), 3/\( 0.48 \), 4/\( 0.46 \), 5/\( 0.47 \), 6/\( 0.46 \), 7/\( 0.48\)}
 \draw (\i) edge[loop below, 
 "\txt"] (\i);
 
 \end{tikzpicture}
 \begin{tikzpicture}
 \draw (4,-1.8) node
 {\begin{minipage}{0.8\textwidth}
 \begin{center}\small
 \textbf{Type I Jacobi--Piñeiro\( (0.4,0.6,0.75)\) Markov chain diagram}
 \end{center}
 \end{minipage}};
 \end{tikzpicture}
\end{center}
 For both Markov chains the corresponding steady state \eqref{estacionariomultiple} is
\begin{align*}
\pi_7(0.4,0.6,0.75)\approx \begin{bNiceMatrix}
 0.03&
 0.10&
 0.21&
 0.24&
 0.23&
 0.14&
 0.05
\end{bNiceMatrix},
\end{align*}
and the expected return times are
\begin{align*}
 (\bar t_7)_1&\approx 30.10, & (\bar t_7)_2&\approx 9.89, & (\bar t_7)_3&\approx 4.85,&
 (\bar t_7)_4&\approx 4.11,& (\bar t_7)_5&\approx 4.31, & (\bar t_7)_6&\approx 7.24, & (\bar t_7)_7&\approx 21.66.
\end{align*}
\end{exa}

\subsection{Multiple Meixner of the second kind finite Markov chains}
The multiple Meixner of the second kind type II and I polynomials, denoted respectively as \( M_{(n_1,n_2)}=M_{(n_1,n_2)}(x;\beta_1,\beta_2,c)\) and \( M_{(n_1,n_2),i}=M_{(n_1,n_2),i}(x;\beta_1,\beta_2,c) \), where \( i \in\{1,2\} \), respectively, satisfy respective discrete orthogonality relations of the~form:
\begin{align*}
 \sum_{k=0}^{\infty}(-k)_jM_{(n_1,n_2)}(k)w_i(k)&=0, 
 & j&\in\{0,\ldots,n_i-1\} , & i&\in\{1,2\} ,
\end{align*}
%for \( j\in\{0,\ldots,n_i-1\} \), \( i\in\{1,2\} \), 
 and
\begin{align*}
 \sum_{k=0}^{\infty}(-k)_j\left(M_{(n_1,n_2),1}(k)w_1(k)+M_{(n_1,n_2),2}(k)w_2(k)\right)&=0,
 & j&\in\{0,\ldots,n_1+n_2-2\} ,
\end{align*}
%for \( j\in\{0,\ldots,n_1+n_2-2\} \), 
with respect to the weight functions
\( w_i(x;\beta_i,c)=\dfrac{\Gamma(\beta_i+x)c^x}{\Gamma(\beta_i)\Gamma(x+1)} \), 
for \( i =\{1,2\} \). These ones are defined over the set \( \N_0\) with \( \beta _1,\beta_2>0 \), \( 0<c<1\) and, in order to have an AT system, \( \beta_1-\beta_2\not\in\mathbb Z \). 

The recurrence relation coefficients \eqref{recurrenciaII},\eqref{recurrenciaI} are \cite[\S 4.3]{Arvesu}
\begin{align*}
 b_{2m}(\beta_1,\beta_2,c)&=2m+\dfrac{c (\beta_1+3m) }{1-c},&
 b_{2m+1}(\beta_1,\beta_2,c)&=2m+1+\dfrac{c (\beta_2+3m+1)}{1-c},\\
 c_{2m}(\beta_1,\beta_2,c)&=\dfrac{c m (\beta_1+\beta_2+3m-2)}{(1-c)^2} ,&
 c_{2m+1}(\beta_1,\beta_2,c)&=\dfrac{c ((m+1)\beta_1+m(\beta_2+3m+1)) }{(1-c)^2} ,\\
 d_{2m}(\beta_1,\beta_2,c)&=\dfrac{c^2 m(m+\beta_1-1)(m+\beta_1-\beta_2) }{(1-c)^3},&
 d_{2m+1}(\beta_1,\beta_2,c)&=\dfrac{c^2 m(m+\beta_2-1)(m+\beta_2-\beta_1) }{(1-c)^3},
\end{align*}
which are all positive if \( - 1<\beta_1-\beta_2<1 \). 

In \cite{Bidiagonal_factorization}, we proved the following result:
\begin{pro}
For the PBF coefficients \( a_n(\beta_1,\beta_2,c) \), we can represent them in terms of the Gauss hypergeometric function as follows:
\begin{align*}
 a_{6n+1}&=\dfrac{(\beta_1+n)c}{1-c},&
 a_{6n+4}&=\dfrac{(\beta_2+n)c}{1-c},\\
 a_{6n+2}&=\dfrac{(n)_n(\beta_2-\beta_1+n)c \pFq{2}{1}{-n,\beta_2-\beta_1-n}{-2n+1}{\frac{c}{c-1}}}
 {(n+1)_n(1-c) \, \pFq{2}{1}{-n,\beta_2-\beta_1-n}{-2n}{\frac{c}{c-1}}}
 %\dfrac{\pFq{2}{1}{-n,\beta_2-\beta_1-n}{-2n+1}{\frac{c}{c-1}}}{\pFq{2}{1}{-n,\beta_2-\beta_1-n}{-2n}{\frac{c}{c-1}}}
 ,&
 a_{6n+5}&=\dfrac{(n+1)(\beta_1-\beta_2+n+1)c \pFq{2}{1}{-n,\beta_2-\beta_1-n}{-2n}{\frac{c}{c-1}} }{(2n+1)(1-c)^2 \pFq{2}{1}{-n-1,\beta_2-\beta_1-n-1}{-2n-1}{\frac{c}{c-1}}}
% \dfrac{\pFq{2}{1}{-n,\beta_2-\beta_1-n}{-2n}{\frac{c}{c-1}}}{\pFq{2}{1}{-n-1,\beta_2-\beta_1-n-1}{-2n-1}{\frac{c}{c-1}}}
 ,\\
 a_{6n+3}&=
 \dfrac{(2n+1) \, \pFq{2}{1}{-n-1,\beta_2-\beta_1-n-1}{-2n-1}{\frac{c}{c-1}}}{\pFq{2}{1}{-n,\beta_2-\beta_1-n}{-2n}{\frac{c}{c-1}}},&
 a_{6n+6}&=\dfrac{2(n+1) \, \pFq{2}{1}{-n-1,\beta_2-\beta_1-n-1}{-2n-2}{\frac{c}{c-1}} }{(1-c) \, \pFq{2}{1}{-n-1,\beta_2-\beta_1-n-1}{-2n-1}{\frac{c}{c-1}}}
% \dfrac{\pFq{2}{1}{-n-1,\beta_2-\beta_1-n-1}{-2n-2}{\frac{c}{c-1}}}{\pFq{2}{1}{-n-1,\beta_2-\beta_1-n-1}{-2n-1}{\frac{c}{c-1}}}
 ,
\end{align*}
which are all positive whenever \( - 1<\beta_1-\beta_2<0 \). 
\end{pro}

The type II polynomials are \cite[\S 4.3]{Arvesu}
\begin{align*}
 M_{(n_1,n_2)}=\left(\dfrac{c}{c-1}\right)^{n_1+n_2}(\beta_1)_{n_1}(\beta_2)_{n_2} \KF{1:1;2}{1:0;1}{(-x):(-n_1);(-n_2,\beta_1+n_1)}{(\beta_1):--;(\beta_2)}{\dfrac{c-1}{c},\dfrac{c-1}{c}}.
\end{align*}
The type I polynomials are \cite[\S 5]{HahnI}
\begin{multline*}
 M_{(n_1,n_2),i}=\dfrac{(1-c)^{\beta_i+n_1+n_2+n_i-2}}{c^{n_1+n_2-1}} \dfrac{(-1)^{n_i-1}(n_1+n_2-2)!}{(n_1-1)!(n_2-1)!}\dfrac{1}{(\beta_i-\hat{\beta}_i-\hat{n}_i+1)_{n_1+n_2-1}}\\
 \times
 \KF{1:2;1}{1:1;0}{(-n_a+1):(-x,\,\beta_a-\hat{\beta}_a-\hat{n}_a+1);(\hat{\beta}_a-\beta_a-n_a+1)}{(-n_1-n_2+2):(\beta_a);--}{1,\dfrac{c}{c-1}},
\end{multline*}
for \( i =1,2\) with \( \hat {\beta}_i\coloneq \beta_1\delta_{i,2}+\beta_2\delta_{i,1}\) and \( \hat {n}_i\coloneq n_1\delta_{i,2}+n_2\delta_{i,1}=n_1+n_2-n_i \).

\subsection{Multiple Laguerre of the first kind finite Markov chains}
The multiple Laguerre of the first kind type II and I polynomials, denoted as \( L_{(n_1,n_2)}=L_{(n_1,n_2)}(x;\alpha_1,\alpha_2)\) and \( L_{(n_1,n_2),i}=L_{(n_1,n_2),i}(x;\alpha_1,\alpha_2) \), where \( i \in\{1,2\} \), respectively, satisfy respective continuous orthogonality relations of the form:
\begin{align*}
 \int_{0}^{\infty}x^jL_{(n_1,n_2)}(x)w_i(x)\d\mu(x)&=0, 
 & j&\in \{0,\ldots,n_i-1\} , & i&\in\{1,2\} ,
\end{align*}
%for \( j\in \{0,\ldots,n_i-1\} \), \( i\in\{1,2\} \), 
 and 
\begin{align*}
 \int_{0}^{\infty}x^j\left(L_{(n_1,n_2),1}(x)w_1(x)+L_{(n_1,n_2),2}(x)w_2(x)\right)
 \d\mu(x)&=0
 & j&\in\{0,\ldots,n_1+n_2-2\} ,
\end{align*}
%for \( j\in\{0,\ldots,n_1+n_2-2\} \), 
with respect to the weight functions and measure
\( w_i(x;\alpha_i)=\Exp{-x}x^{\alpha_i} \), \( i\in \{1,2\}\) and \( \d\mu(x)=\d x \). 
The support is \( [0,\infty)\) with \( \alpha_1,\alpha_2>-1\) and, in order to have an AT system, \( \alpha_1-\alpha_2\not\in\mathbb Z \). 

The recurrence coefficients, see \eqref{recurrenciaII} and \eqref{recurrenciaI}, are \cite[section 3.2]{Clasicos}
\begin{align*}
 b_{2m}(\alpha_1,\alpha_2)&=3m+1+\alpha_1,&
 b_{2m+1}(\alpha_1,\alpha_2)&=3m+2+\alpha_2,\\
 c_{2m}(\alpha_1,\alpha_2)&=m(3m+\alpha_1+\alpha_2),&
 c_{2m+1}(\alpha_1,\alpha_2)&=3m^2+m(\alpha_1+\alpha_2+3)+\alpha_1+1,\\
 d_{2m}(\alpha_1,\alpha_2)&=m(m+\alpha_1)(m+\alpha_1-\alpha_2),&
 d_{2m+1}(\alpha_1,\alpha_2)&=m(m+\alpha_2)(m+\alpha_2-\alpha_1),
\end{align*}
which are all positive if \( - 1<\alpha_1-\alpha_2<1 \). 

In \cite{Bidiagonal_factorization}, we proved the following result:
\begin{pro}
 For the PBF coefficients \( a_n(\alpha_1,\alpha_2)\) we find
 \begin{align*}
 a_{6n+1}&=\alpha_1+1+n,&
 a_{6n+4}&=\alpha_2+1+n,\\
 a_{6n+2}&=\alpha_2-\alpha_1+n,&
 a_{6n+5}&=n+1,\\
 a_{6n+3}&=\alpha_1-\alpha_2+n+1,&
 a_{6n+6}&=n+1.
\end{align*}
These ones are all positive whenever \( - 1<\alpha_1-\alpha_2<0 \). 
\end{pro}

The type II polynomials are \cite[\S 3.2]{ContinuosII}
\begin{align*}
 L_{(n_1,n_2)} =(-1)^{n_1+n_2}(\alpha_1+1)_{n_1}(\alpha_2+1)_{n_2} \KF{0:2;1}{1:1;0}{--:(-n_2,\alpha_1+n_1+1);(-n_1)}{(\alpha_1+1):(\alpha_2+1);--}{x,x}.
\end{align*}
The type I polynomials are \cite[\S 7]{HahnI}
\begin{align*}
 L_{(n_1,n_2),i}%(x;\alpha_1,\alpha_2)
 =(-1)^{n_1+n_2-1}\dfrac{1}{(n_i-1)!\Gamma(\alpha_i+1)(\hat{\alpha}_i-\alpha_i)_{\hat{n}_i}}
 \pFq{2}{2}{-n_i+1,\alpha_i-\hat{\alpha}_i-\hat{n}_i+1}{\alpha_i+1,\alpha_i-\hat{\alpha}_i+1}{x},
\end{align*}
for \( i =1,2\) with \( \hat {\alpha}_i\coloneq \alpha_1\delta_{i,2}+\alpha_2\delta_{i,1}\) and \( \hat {n}_i\coloneq n_1\delta_{i,2}+n_2\delta_{i,1}=n_1+n_2-n_i\).

\section*{Conclusions and outlook}

In the 1950s, significant advancements were made in understanding the connections between orthogonal polynomials and stochastic processes. Influential papers by Kendall, Ledermman, Reuter, Karlin, and McGregor focused on the spectral representation of probabilities in birth and death processes. These contributions paved the way for integral representations of probabilistic quantities of countable Markov chains using orthogonal polynomials.

Building upon previous works for countable (infinite) Markov chains, we presented a general construction for finite Markov chains that applies to families of orthogonal and multiple orthogonal polynomials in the Askey scheme, provided they have nonnegative recursion matrices. By applying explicit hypergeometric expressions, we were able to find numerical examples of finite Markov chains associated with these families. For that aim we have provided corresponding Mathematica codes.

%The paper thoroughly investigated the properties of these finite Markov chains, including classes, recurrence, transience, periodicity, ergodicity, stationary states, expected return times and time reversal. Additionally, we described a procedure to factor stochastic matrices into bidiagonal stochastic matrices, enabling us to model pure birth or pure death Markov chains.

%In conclusion, our study has provided valuable insights into the interplay between orthogonal polynomials and stochastic processes. By extending the theory to finite Markov chains and truncations of the recursion matrix, we have opened up new avenues for future research. Furthermore, the explicit numerical examples of finite Markov chains presented in this paper serve as a starting point for further exploration and applications in diverse fields.

Looking ahead, the paper opens up 
%exciting prospects for 
future research in the realm of Markov chains and orthogonal polynomials. One 
%intriguing 
direction involves %delving into 
Markov chains generated by multiple orthogonal polynomials in the stepline with more than two weights (i.e., \( p>2\)), or even more to mixed multiple orthogonal polynomials, expanding beyond the scope of this current work. However, achieving this objective will require explicit hypergeometric expressions for the corresponding polynomials and the recursion matrix (see for instance~\cite{HahnI_nuevo}).
%The task of obtaining such expressions presents a stimulating challenge for researchers in the field, and its successful accomplishment would undoubtedly enrich our understanding of the interplay between orthogonal polynomials and stochastic processes.
%Another promising avenue for future exploration is the investigation of families of mixed multiple orthogonal polynomials within a potential mixed multiple Askey scheme, using nonnegative recursion matrices. 
%By studying the associated Markov chains arising from such polynomials, researchers have the opportunity to gain valuable insights into the behavior of more complex stochastic processes. This line of inquiry holds great potential in enhancing our understanding of the dynamics and properties of these intricate processes. The insights derived from this research could find applications in diverse fields, ranging from physics and engineering to finance and beyond.
Moreover, the exploration of possible permutations of the pure birth/pure death stochastic factorization 
%presents an exciting opportunity. By 
by studying Darboux transformations of the spectral measures resulting from such permutations, we can gain a deeper insight into the underlying dynamics of the corresponding Markov chains. 
%These transformations may reveal hidden patterns and relationships, leading to significant advancements in the theory of stochastic processes.

%Furthermore, studying the connections between orthogonal polynomials and other types of stochastic processes, such as diffusions and Brownian motion, could reveal new relationships and applications. Exploring how orthogonal polynomials interact with these different stochastic processes could lead to valuable insights and potentially open up new areas of research.

\section*{Acknowledgments}
The authors are really grateful to the anonymous referees for their careful revision of
the manuscript. Their valuable comments and suggestions have contributed to improve
this work.

AB
acknowledges Centre for Mathematics of the University of Coimbra 
(funded by the Portuguese Government through FCT/MCTES, doi: 10.54499/UIDB/00324/2020).

JEFD \& AF acknowledge the CIDMA Center for Research and Development in Mathematics and Applications (University of Aveiro) and the Portuguese Foundation for Science and Technology (FCT) for their support within projects
%\hyperref{https://sciproj.ptcris.pt/157514UID}{}{}{https://doi.org/10.54499/UIDB/04106/2020}
doi: 10.54499/UIDB/04106/2020
 \&
% \hyperref{https://sciproj.ptcris.pt/157829UID}{}{}{https://doi.org/10.54499/UIDP/04106/2020}.
doi: 10.54499/UIDP/04106/2020.

JEFD 
\& 
MM
acknowledge
PID2021-122154NB-I00, entitled ``Ortogonalidad y Aproximación con Aplicaciones en Machine Learning y Teoría de la Probabilidad,'' funded by
\href{https://doi.org/10.13039/501100011033}{MICIU/AEI/10.13039 /501100011033} and by "ERDF A Way of making Europe.” 
Additionally,
JEFD acknowledges the PhD contract doi: 
\linebreak
10.54499/UI/BD/152576/2022 from FCT Portugal.

\section*{Declarations}

\begin{enumerate}[\rm i)]
 \item \textbf{Conflict of interest:} The authors declare no conflict of interest.
 \item \textbf{Ethical approval:} Not applicable.
 \item \textbf{Contributions:} All the authors have contribute equally.
 \item \textbf{Generative AI and AI-assisted technologies in the writing process:}
 During the preparation of this work the authors used ChatGPT in order to improve English grammar, syntax, spelling and wording. After using this tool/service, the authors reviewed and edited the content as needed and take full responsibility for the content of the publication.
 \item \textbf{Data availability:} This paper has associated data. There are two Mathematica notebooks: 
 \begin{enumerate}[\(\bullet\)]
 \item 
 \hyperref{https://notebookarchive.org/finite-markov-chains-and-orthogonal-polynomials--2023-07-cf7w96a/}{}{}{ \texttt{MarkovChains\&OrthogonalPolynomials.nb}} 
 \item \hyperref{https://notebookarchive.org/finite-markov-chains-and-multiple-orthogonal-polynomials--2023-07-cf82tof/}{}{}{ \texttt{MarkovChains\&MultipleOrthogonalPolyomials.nb}} 
 \end{enumerate}
 These notebooks have been uploaded to the \hyperref{https://notebookarchive.org}{}{}{Mathematica Notebook Archive} and to \hyperref{https://github.com/ManuelManas/Markov-chains-and-orthogonal-polynomials/tree/main}{}{}{GitHub}.
\end{enumerate}


\begin{thebibliography}{11}

\bibitem{afm}
 C. Álvarez-Fernández, U. Fidalgo, and M. Mañas,
 \emph{Multiple orthogonal polynomials of mixed type: Gauss--Borel factorization and the multi-component 2D Toda hierarchy},
 Advances in Mathematics~\textbf{227} (2011) 1451–1525.
 
\bibitem{andrews}
G. E. Andrews, R. Askey, and R. Roy,
\textit{Special Functions},
Cambridge University Press, Cambridge, 1999.

\bibitem{ContinuosII}
A.I. Aptekarev, A. Branquinho, and W. Van Assche,
\emph{Multiple orthogonal polynomials for classical weights}, Transactions of the
American Mathematical Society \textbf{355} (10) (2003) 3887–3914.

\bibitem{genetico}
A. I. Aptekarev, V. Kaliaguine, and J. Van Iseghem,
\emph{The genetic sums' representation for the moments of a system of Stieltjes functions and its application}, Constructive Approximation \textbf{16} (2000) 487–524.

\bibitem{AGMM}
G. Ariznabarreta, J. C. García-Ardila, M. Mañas, and F. Marcellán,
\emph{Non-Abelian integrable hierarchies: matrix biorthogonal polynomials and perturbations},
Journal of Physics A: Mathematical and Theoretical \textbf{51} (20), 205204.


\bibitem{Arvesu}
J. Arvesu, J. Coussement, and W. Van Assche,
\textit{Some discrete multiple orthogonal polynomials},
Journal of Computational and Applied Mathematics \textbf{153} (2003) 19–45.

\bibitem{nnm}
R. B. Bapat and T. E. S. Raghavan,
\emph{Nonnegative Matrices and Applications},
Encyclopedia of Mathematics and its Applications \textbf{64},
Cambridge University Press, Cambridge, 1997.


\bibitem{AskeyII}
B. Beckerman, J. Coussement, and W. Van Assche,
\textit{Multiple Wilson and Jacobi--Piñeiro polynomials},
Journal of Approximation Theory \textbf{132} (2005) 155–181.



\bibitem{HahnI}
A. Branquinho, J. E. F. Díaz, A. Foulquié-Moreno, and M. Mañas,
\emph{Hahn multiple orthogonal polynomials of type I: Hypergeometric expressions},
Journal of Mathematical Analysis and Applications \textbf{528} (2023) 1277471.


\bibitem{JP}
A. Branquinho, J. E. F. Díaz, A. Foulquié-Moreno, M. Mañas, and C. Álvarez-Fernández,
\emph{Jacobi--Piñeiro random walks},
Revista de la Real Academia de Ciencias Exactas, Físicas y Naturales. Serie A. Matemáticas \textbf{118} (2024) article number: 15. 



\bibitem{Bidiagonal_factorization}
A. Branquinho, J. E. F. Díaz, A. Foulquié-Moreno, and M. Mañas,
\emph{Bidiagonal factorization of the recurrence matrix for the multiple Hahn orthogonal polynomials}, Linear Algebra and Its Applications
\hyperref{https://doi.org/10.1016/j.laa.2024.03.033}{}{}{doi:10.1016/j.laa.2024.03.033} (2024).
%\hyperref{https://arxiv.org/abs/2308.01288}{}{}{arXiv:2308.01288}.







\bibitem{HahnI_nuevo}
A. Branquinho, J. E. F. Díaz, A. Foulquié-Moreno, and M. Mañas,
\emph{Hypergeometric expressions for type I Jacobi--Pi\~neiro orthogonal polynomials with arbitrary number of weights},
 Proceedings of the American Mathematical Society, Ser. B, 11 (2024) 
 200–210.
% \hyperref{https://doi.org/10.48550/arXiv.2310.18294}{}{}{\texttt{arXiv:2310.18294}}.




\bibitem{Hypergeometric} 
A. Branquinho, J. E. F. Díaz, A. Foulquié-Moreno, and M. Mañas,
\emph{Hypergeometric Multiple Orthogonal Polynomials and Random Walks},
\hyperref{https://arxiv.org/pdf/2107.00770.pdf}{}{}{\texttt{arXiv:2107.00770}}.

\bibitem{bfm}
A. Branquinho, A. Foulquié-Moreno, and M. Mañas, 
\emph{Multiple orthogonal polynomials: Pearson equations and Christoffel formulas},
Analysis and Mathematical Physics \textbf{12} (2022) paper 129.

\bibitem{Darboux}
A. Branquinho, A. Foulquié-Moreno, and M. Mañas,
\emph{Bidiagonal factorization of tetradiagonal matrices and Darboux transformations},
Analysis and Mathematical Physics \textbf{13} (2023) paper 42.

\bibitem{PBF_tetra}
A. Branquinho, A. Foulquié-Moreno, and M. Mañas, 
\emph{Positive bidiagonal factorization of tetradiagonal Hessenberg matrices},
Linear Algebra and its Applications \textbf{667} (2023) 132–160.

\bibitem{stbbm0}
A. Branquinho, A. Foulquié-Moreno, and M. Mañas,
\emph{Oscillatory banded Hessenberg matrices, multiple orthogonal polynomials and random walks},
Physica Scripta \textbf{98} (2023) 105223.

\bibitem{stbbm}
A. Branquinho, A. Foulquié-Moreno, and M. Mañas,
\emph{Spectral theory for bounded banded matrices with positive bidiagonal factorization and mixed multiple orthogonal polynomials}, 
Advances in Mathematics \textbf{434} (2023) 109313.

\bibitem{stbbm-nuevo}
A. Branquinho, A. Foulquié-Moreno, and M. Mañas, \emph{Banded totally positive matrices and normality for mixed multiple orthogonal polynomials}, 
\hyperref{https://arxiv.org/pdf/2404.13965.pdf}{}{}{\texttt{arXiv:2404.13965}}.


\bibitem{Bremaud}
P. Brémaud, \emph{Markov Chains},
Second Edition, Text in Applied Mathematics \textbf{31}, 
Springer-Verlag, Berlin, 2020.

\bibitem{mirta}
M. M. Castro and F. A. Grünbaum,
\emph{On a Seminal Paper by Karlin and McGregor},
Symmetry, Integrability and Geometry: Methods and Applications SIGMA \textbf{9} (2013) 020, 11 pages.

\bibitem{Charris}
J. Charris and M. E. H. Ismail,
\emph{On sieved orthogonal polynomials II: random walk polynomials},
Canadian Journal of Mathematics \textbf{38} (1986) 397–415.

\bibitem{Chiara}
T. S. Chihara and M. E. H. Ismail,
\emph{Orthogonal polynomials suggested by a queueing model},
Advances in Applied Mathematics \textbf{3} (1982) 441–462.

\bibitem{Coolen}
P. Coolen-Schrijner and E. A. van Doom,
\emph{Analysis of random walks using orthogonal polynomials},
Journal of Computational and Applied Mathematics \textbf{99} (1998) 387–399.

\bibitem{gqmop}
J. Coussement and W. Van Assche,
\emph{Gaussian quadrature for multiple orthogonal polynomials}, 
Journal of Computational and Applied Mathematics \textbf{178} (2005) 131–145.

\bibitem{Diaconis}
P. Diaconis and S. Zabell,
\emph{Closed form summation for classical distributions: Variations on a theme of de Moivre}, 
Statistical Science \textbf{6} (1991) 284–302.

\bibitem{Engel}
D. D. Engel,
\emph{The Multiple Stochastic Integral},
Memoires of the American Mathematical Society \textbf{265}, AMS, Providence, 1982.



 \bibitem{Fallat-Johnson}
 S. M. Fallat and Ch. R. Johnson, 
 \emph{Totally Nonnegative Matrices}, 
 Princeton Series in Applied Mathematics,
 Princeton University Press, Princeton, 2011.
 
 
 
 
\bibitem{feller}
W. Feller,
\emph{An introduction to probability theory and its applications, Volume I},
John Wiley and Sons, New York, 1967.

\bibitem{gallager}
R. Gallager,
\emph{Stochastic Processes, Theory for Applications},
Cambridge University Press, Cambridge, 2013.


%\bibitem{Horn-Johnson} R. A. Horn and C. R. Johnson, \emph{Matrix Analysis}, second edition, Cambridge University Press, 2013.


 \bibitem{Gantmacher-Krein} 
 F. P. Gantmacher and M. G. Krein, 
 \emph{Oscillation and Kernels and Small Vibrations of Mechanical Systems},
 revised second edition, AMS Chelsea Publishing, American Mathematical Society, Providence, Rhode Island, 2002.
 
 
\bibitem{Grunbaum1}
F. A. Grünbaum,
\emph{Random walks and orthogonal polynomials: some challenges},
Probability, Geometry and Integrable Systems MSRI Publications \textbf{55} (2007) 241–260.


\bibitem{Grunbaum11} 
F. A. Grünbaum,
\emph{The Karlin--McGregor formula for a variant of a discrete version of Walsh's spider},
Journal of Physics~A: Mathematical and Theoretical \textbf{42} (2009) 454010.

\bibitem{Grunbaum2}
F. A. Grünbaum, 
\emph{An urn model associated with recursion polynomials},
Communications in Applied Mathematics and Computational Science \textbf{5} (2010) 55–63.



\bibitem{grunbaum}
F. A. Grünbaum and M. D. de la Iglesia,
\emph{An urn model for the Jacobi--Piñeiro polynomials},
Proceedings of the American Mathematical Society \textbf{150} (2022) 3613–3625.



\bibitem{Grunbaum3}
F. A. Grünbaum, I. Pacharoni, and J. Tirao,
\emph{Two stochastic models of a random walk in the \( U(n)\)-spherical duals of \( U(n + 1)\)}, 
Annali di Matematica Pura ed Applicata \textbf{192} (2013) 447–473.


\bibitem{Haggstrom}
O. Häggstrom,
\emph{Finite Markov Chains and Algorithmic Applications},
London Mathematical Society,
Student Texts \textbf{52}, Cambridge University Press, Cambridge, 2002.

\bibitem{delaIglesia}
 M. D. de la Iglesia, 
 \emph{Orthogonal Polynomials in the Spectral Analysis of Markov Processes}, 
 Cambridge University Press, Cambridge, 2022.
 
\bibitem{Ismail}
M. E. H. Ismail,
\textit{Classical and Quantum Orthogonal Polynomials in One Variable},
Cambridge University Press, Cambridge, 2005.

\bibitem{Ito}
K. Itō,
\emph{Multiple Wiener integral},
Journal of the Mathematical Society of Japan \textbf{3} (1951) 157–169.

\bibitem{KmcG1957-1}
S. Karlin and J. McGregor,
\emph{The differential equations of birth-and-death processes, and the Stieltjes moment problem},
Transactions of the American Mathematical Society \textbf{85} (1957) 489–546.

\bibitem{KmcG1957-2}
S. Karlin and J. McGregor, 
\emph{The classification of birth and death processes},
Transactions of the American Mathematical Society \textbf{86} (1957) 333–400.

\bibitem{KmcG}
S. Karlin and J. McGregor,
\emph{Random walks},
Illinois Journal of Mathematics \textbf{3} (1) (1959) 66–81.

\bibitem{Kendall}
D. G. Kendall,
\emph{Unitary dilations of one-parameter semigroups of Markov transition operators, and the corresponding integral representation for Markov processes with a countable infinity of states}, 
Proceedings of the London Mathematical Society~\textbf{3} (1959) 417–431.

\bibitem{Kendall2}
D. G. Kendall and G. E. H. Reuter,
\emph{The calculation of the ergodic projection for Markov chains and processes with a countable infinity of states},
Acta Mathematica \textbf{97} (1959) 103–144.

\bibitem{Kijima}
M. Kijima, M. Gopalan Nair, P. K. Pollett, and E. A. Van Doorn,
\emph{Limiting conditional distributions for birth-death processes},
Advances in Applied Probability \textbf{29} (1997) 185–204.

\bibitem{askeyescalar}
R. Koekoek, P. Lesky, and R. F. Swarttouw, 
\emph{Hypergeometric Orthogonal Polynomials and Their \( q\)-Analogues}, 
Springer Monographs in Mathematics, 
Springer-Verlag, Berlin, 2010.

\bibitem{Kovchegov}
Y. Kovchegov, 
\emph{Orthogonality and probability: beyond nearest neighbor transitions},
Electronic Communications in Probability \textbf{14} (2009) 90–103.

\bibitem{Ledermann}
W. Ledermann and G. E. H. Reuter,
\emph{Spectral theory for the differential equations of simple birth and death processes},
Philosophical Transactions of the Royal Society of London, Series A \textbf{246} (1954) 321–369.

 \bibitem{nikishin_sorokin}
E. M. Nikishin and V. N. Sorokin,
\emph{Rational Approximations and Orthogonality},
Translations of Mathematical Monographs \textbf{92},
American Mathematical Society, Providence, 1991.

\bibitem{Obata}
N. Obata,
\emph{The Karlin--McGregor formula for paths connected with a clique}, 
Probability and Mathematical Statistics \textbf{33} (2013) 451–466.

\bibitem{Ogura}
H. Ogura,
\emph{Orthogonal functionals of the Poisson process},
IEEE Transactions on Information Theory \textbf{IT-18} (1972) 474–481.

 \bibitem{Schoutens}
W. Schoutens,
\emph{Stochastic Processes and Orthogonal Polynomials},
Lecture Notes in Statistics \textbf{146}, Springer-Verlag, Berlin, 2000.

\bibitem{slater}
L. J. Slater,
\textit{Generalized Hypergeometric Functions},
Cambridge University Press, Cambridge, 2008.



\bibitem{Sokal}
A. Sokal, 
\emph{Multiple orthogonal polynomials, \( d\)-orthogonal polynomials, production matrices, and branched continued fractions},
Transactions of the American Mathematical Society, Series B, 11(23) 
(2024) 762–797.



\bibitem{Srivastava}
 H. M. Srivastava and P W. Karlsson, 
 \emph{Multiple Gaussian Hypergeometric Series}, 
 Ellis Horwood Limited, John Wiley and Sons,
Chichester, 1985.

\bibitem{Clasicos}
W. Van Assche and E. Coussement,
\textit{Some classical multiple orthogonal polynomials},
Journal of Computational and Applied Mathematics {\textbf{127}} (2001) 317–347.
%,\url{https://doi.org/10.1016/S0377-0427(00)00503-3}

\bibitem{Van Doorn0}
E. A. Van Doorn,
\emph{The transient state probabilities for a queueing model where potential customers are discouraged by queue length},
Journal of Applied Probability \textbf{18} (1981) 499–506.

\bibitem{Van Doorn}
E. A. Van Doorn and P. Schrijner,
\emph{Random walk polynomials and random walk measures},
Journal of Computational and Applied Mathematics \textbf{49} (1993) 289–296.


\bibitem{Viennot}
G. Viennot, 
\emph{Une th\'eorie combinatoire des polyn\^omes orthogonaux g\'en\'eraux}, Notes de conf\'erences donn\'ees \`a l'Universit\'e du Qu\'ebec \`a Montr\'eal, septembre-octobre 1983,
\hyperref{http://www.xavierviennot.org/xavier/polynomes_orthogonaux.html}{}{}{\texttt{Available on-line}}.


\bibitem{Whitehurst} 
T. Whitehurst, 
\emph{An application of orthogonal polynomials to random walks}, 
Pacific Journal of Mathematics \textbf{99} (1982) 205–213.


\bibitem{wiener}
N. Wiener,
\emph{The homogeneous chaos},
American Journal of Mathematics \textbf{60} (1930) 897–936.


\end{thebibliography}
\end{document}